\documentclass[12pt]{article}

\usepackage{graphicx}
\usepackage{authblk}
\usepackage{amssymb,amsmath,amsthm}
\usepackage{subfig}
\usepackage{enumerate}
\usepackage[psdextra]{hyperref}
\hypersetup{colorlinks,linkcolor=blue,urlcolor=blue,citecolor=blue}
\usepackage[title]{appendix}
\usepackage[
    backend=biber,
    style=numeric,
    giveninits=true,
    sorting=nyt,
    maxbibnames=99
]{biblatex}

\DeclareFieldFormat[article]{volume}{\textbf{#1}}
\newtheorem{theorem}{Theorem}[section]
\newtheorem{proposition}[theorem]{Proposition}
\newtheorem{corollary}[theorem]{Corollary}
\newtheorem{lemma}[theorem]{Lemma}
\newtheorem{example}[theorem]{Example}
\newtheorem{remark}[theorem]{Remark}

\makeatletter
\def\th@plain{%
  \thm@notefont{}%
  \itshape
}
\def\th@definition{%
  \thm@notefont{}%
  \normalfont
}
\makeatother

\newcommand{\bu}{\mathbf{u}}
\newcommand{\bv}{\mathbf{v}}

\newcommand{\bq}{\mathbf{q}}
\newcommand{\br}{\mathbf{r}}

\DeclareMathOperator{\diag}{diag}

\title{The Hessian of Planar Central Configurations in Pair Space: Decomposition, 
Morse Index and Symmetry Reduction}

\author{Manuele Santoprete Santoprete}
\affil{Department of Mathematics, Wilfrid Laurier University,\\ 75 University Ave. West, Waterloo,
ON, Canada\\ \texttt{msantoprete@wlu.ca}}
\date{\today}

\begin{document}

\maketitle

\begin{abstract}

We give a variational derivation of the central configuration equations in
pair space, where the relative position vectors between pairs of bodies
serve as the primary variables. Vector Lagrange multipliers enforce the
linear triangle relations required for these pair vectors to be realizable
in the plane. For arbitrary $N$, we show that the constrained Hessian decomposes as
$H_{\mathcal C}=L^\Delta+\widetilde L$, where the gap Laplacian
$L^\Delta$ is positive semidefinite and the transverse Laplacian
$\widetilde L$ is signed and contains all possible negative directions.

For non-collinear planar four-body central configurations, the signed part
has rank two. This yields a new proof of the known bound that the Morse
index is at most two. We then characterize positive definiteness and
degeneracy of the Hessian by the eigenvalues of an explicit \(2\times2\)
matrix or, equivalently, by two generalized eigenvalues of a pencil of
symmetric matrices. We illustrate the general criterion using the
equal-mass square central configuration.

For reflection-symmetric configurations, the generalized eigenvalue problem
decomposes into two independent smaller problems. This reduction applies to
kite and isosceles trapezoidal configurations. For the trapezoid, however,
the reflection acts non-orthogonally on the pair coordinates,  giving a different,
non-orthogonal reduction.
Applying the reflection-symmetry reduction to
the rhombus family, we prove that the Hessian is positive definite modulo
the rotational zero mode; equivalently, every rhombus central configuration
is nondegenerate modulo rotations and has Morse index zero.
\end{abstract}

\noindent\textbf{Keywords:}
central configurations; \(N\)-body problem; pair space; 
Morse index; reflection symmetry.

\medskip
\noindent\textbf{2020 Mathematics Subject Classification:}
Primary 70F10; Secondary 58E05.
\tableofcontents
\section{Introduction}
\label{sec:introduction}

Central configurations play a fundamental role in the planar Newtonian
\(N\)-body problem.  They generate the homographic solutions, describe the
possible limiting shapes of total-collision motions, and enter the study of
the topology of the energy--momentum level sets.  Moulton proved that, for
each ordering of \(N\) positive masses on a line, there is a unique
collinear central configuration, up to scaling.  The corresponding
questions for non-collinear configurations are much less well understood.
In particular, Smale included the finiteness of the number of central
configurations, for fixed positive masses and modulo the natural
similarities, among his problems for the twenty-first century
\cite{smale1998mathematical}. 
Finiteness has been established in several important cases, notably for
the planar four-body problem by Hampton and Moeckel
\cite{hampton2006finiteness}, and for the planar five-body problem by
Albouy and Kaloshin \cite{albouy2012finiteness}, who proved finiteness for
all positive masses outside a codimension-two algebraic subvariety of the
mass space. Building on Albouy and Kaloshin's method of $zw$-diagrams,
Chang and Chen \cite{ChangChen2024,ChangChen2025} made substantial algorithmic progress on
the planar six-body problem, reducing the finiteness question to a small
number of unresolved cases.
The general problem, however, remains open.

Morse theory provides a bridge between the local second-variation
properties of central configurations and global information on their
number. After fixing the center of mass at the origin and normalizing the
moment of inertia, thereby removing translations and scaling, planar
central configurations correspond to critical points of the induced
Newtonian potential. Configurations that differ by a rotation are
considered equivalent.

A planar central configuration is called {\bf nondegenerate} if the kernel of
the Hessian of the induced potential consists precisely of the tangent
direction to its orbit under planar rotations. Equivalently, after
identifying configurations that differ by a rotation, the Hessian has no
zero eigenvalues. The Morse index is the number of negative eigenvalues
after this rotational direction has been omitted.

When all critical points are nondegenerate in this sense, the Morse
inequalities relate the topology of the space of configurations modulo
rotation to the number of critical points of each Morse index. The Morse
index of the reduced Hessian is therefore the key local ingredient in
such arguments, yielding lower bounds in general and, in special cases
such as the collinear problem \cite{smale1970topology}, convex
cocircular configurations \cite{santoprete2021uniqueness}, or the
trapezoidal problem \cite{santoprete2021uniqueness1}, exact counts.

This program was initiated by Smale \cite{smale1970topology} in the
collinear case, where Morse theory yields an exact count, and was extended
by Palmore
\cite{palmore1973classifying,palmore1975classifying,palmore1975classifyingIII}
to the planar problem, where the resulting Morse inequalities give lower
bounds rather than an exact count; it has since been developed in several
directions.

The possible presence of additional zero modes is equally significant.
The center-of-mass and moment-of-inertia constraints remove the
translational and scaling directions, while identifying configurations
that differ by a planar rotation removes the remaining rotational
symmetry. If the Hessian has a nontrivial kernel after these symmetry
directions have been removed, then the corresponding central
configuration is Morse degenerate. Such degeneracy prevents the direct
application of the standard Morse inequalities and can signal a
parameter value at which a branch of central configurations bifurcates
as the masses, or other parameters, vary. Mass-parameter bifurcations
of central configurations and relative equilibria have been studied in
\cite{MeyerSchmidt1988,MeyerSchmidt1988b,MeyerSchmidt1989,
rusu2016bifurcations,liu2026concave}.

Consequently, a method that determines both the Morse index and the
dimension of the kernel of the Hessian after the symmetry directions have
been removed provides the local information needed for both
Morse-theoretic counting arguments and the study of bifurcations in
parameter-dependent families.

These issues are particularly relevant in the planar four-body problem in
connection with the Sim\'o--Yoccoz conjecture. The conjecture predicts that,
for every choice of four positive masses and each prescribed cyclic ordering
of the bodies, there is, up to translations, rotations, and scaling, a
unique strictly convex planar central configuration realizing that ordering
\cite{santoprete2021uniqueness}. Existence for a prescribed cyclic ordering
is classical; the central difficulty is uniqueness. From this perspective,
degeneracy is especially important: it is the local mechanism through which
solution branches can potentially meet, bifurcate, or change their number as
the masses vary. Substantial partial results are known for symmetric or
geometrically restricted families, including kite and isosceles trapezoidal
configurations, cocircular configurations, and families studied by
bifurcation methods; see, for example,
\cite{leandro2003finiteness,albouy2008symmetry,perez2007convex,
fernandes2017convex,santoprete2021uniqueness,
santoprete2021uniqueness1,roberts2025kite,corbera2014central}.
Complementing these geometric, analytic, and
bifurcation-based results, rigorous computer-assisted methods using interval arithmetic have
established the conjecture on explicit regions of the mass space
\cite{sun2023uniqueness}.

Although these results provide strong evidence for the Sim\'o--Yoccoz
conjecture, a general proof for arbitrary positive masses remains
unavailable.  A possible approach is to study the potential on the
normalized space of convex configurations with fixed cyclic ordering.  If
all critical points were nondegenerate minima, suitably sharp
Morse-theoretic information could imply uniqueness.  This approach is
realized in special cases, including the co-circular and trapezoidal
problems \cite{santoprete2021uniqueness,santoprete2021uniqueness1}.  Thus, index and
degeneracy do not provide just local information, but are potential tools for a possible approach to the Sim\'o--Yoccoz conjecture.

The purpose of this paper is to develop a new method for determining
the index and degeneracy of the constrained Hessian at planar central
configurations.

A standard approach to planar central configurations, particularly in the
four-body problem, is to use the mutual distances as basic variables.
In these coordinates, the central-configuration equations take a
relatively simple form, and many results have been obtained by exploiting
this description. 
The mutual distances, however, are not independent. For four points in the
plane, a tuple of six pairwise distances is realizable by an actual planar
configuration if and only if the associated Cayley--Menger determinant
vanishes and each of the four triples of distances satisfies the ordinary
triangle inequality \cite{CorsRoberts2012}.
Thus, a variational treatment in mutual-distance
coordinates requires the use of the Cayley--Menger relation as a nonlinear
constraint.  While effective for the central-configuration equations
themselves, this formulation becomes less convenient at the level of the
second variation: the constrained Hessian contains contributions involving
the  second derivatives of the Cayley--Menger determinant, whose
explicit expressions are cumbersome.

To avoid this difficulty, following an idea introduced by Drory
\cite{drory2025pair,drory2025pairII}, we work instead in the space of relative position vectors
\[
	{\bf q}_{ij}={\bf r}_i-{\bf r}_j.
\]
These pair vectors retain the direct geometric meaning of the
mutual-distance variables, but their realizability is described by linear
triangle relations.  We impose these relations by vector-valued Lagrange
multipliers and show that the resulting pair-space equations are
equivalent to the classical central-configuration equations.  Since the
realizability subspace is linear, its curvature vanishes; consequently,
the constrained Hessian is obtained simply by restricting the ambient
pair-space Hessian to this subspace.  This eliminates the second-derivative
terms associated with the Cayley--Menger constraint.

The pair-space formulation also reveals a useful spectral structure.
Each pair contribution admits a decomposition into a positive semidefinite
term, determined by the gap between its radial and transverse eigenvalues,
and a signed term determined by its transverse eigenvalue alone. 
When these edge contributions are assembled, the positive semidefinite
terms combine into a single matrix: a graph Laplacian whose entries are
themselves $2\times2$ blocks rather than scalars, one block for each pair
of bodies, so that each edge weight is a $2\times2$ positive semidefinite
matrix rather than a number. Such an object is called a matrix-weighted
graph Laplacian \cite{Trinh2018,Zhang2020}. Grounding body $1$ then yields
a reduced version of this matrix, which we call the \textbf{gap Laplacian}
$L^\Delta$. The signed terms define the \textbf{transverse Laplacian}
$\widetilde L$, which has the form
\[
\widetilde L=\ell\otimes\mathbf I_2,
\]
where $\ell$ is the \textbf{reduced transverse Laplacian}, an ordinary
scalar-weighted graph Laplacian, and $\otimes$ denotes the Kronecker
product. Thus,
\[
H_{\mathcal C}=L^\Delta+\widetilde L,
\]
and all possible negative directions of $H_{\mathcal C}$ are contained in
$\widetilde L$.
For a non-collinear planar four-body central
configuration, \(\ell\) has rank one.  Consequently, \(\widetilde L\) has
rank two, which gives a direct proof that the Morse index of
\(H_{\mathcal C}\) is at most two.

After regularizing \(L^\Delta\) by removing its one-dimensional kernel,
corresponding to the rotational zero mode, the resulting positive definite
symmetric form \( L^\Delta_{\mathrm{reg}}\) and the form \(\widetilde L\) are simultaneously
diagonalizable by congruence.

Since $\widetilde{L}$ has rank two, the definiteness and degeneracy of the Hessian depend entirely on its restriction to two effective directions. We present two equivalent criteria to evaluate these properties. The first criterion determines definiteness and degeneracy through the eigenvalues of an explicit $2\times2$ symmetric matrix $\Omega$. 
This requires inverting the $6\times6$ matrix $L^\Delta_{\mathrm{reg}}$. The second, inverse-free criterion works directly with the generalized eigenvalues of the pencil $(\widetilde{L}, L^\Delta_{\mathrm{reg}})$, computed from the characteristic polynomial $\det(\lambda L^\Delta_{\mathrm{reg}} - \widetilde{L})$.
Because
\(\widetilde L\) is a $ 6 \times 6 $ matrix of  rank two, four of the six generalized eigenvalues vanish
identically, and the polynomial factors as
\(\lambda^4\bigl(c_2\lambda^2+c_1\lambda+c_0\bigr)\); the two roots of the
quadratic factor are precisely the two nonzero generalized eigenvalues, and
they coincide with the eigenvalues of \(\Omega\).
In either formulation, the Hessian is positive definite on the complement
of the rotational zero mode when both eigenvalues exceed
\(-1\), while degeneracy occurs when at least one of them equals \(-1\).

The second, inverse-free criterion is
considerably more efficient in practice, since it requires only a single
\(6\times6\) determinant rather than a symbolic matrix inversion. We
illustrate the second criterion with an application to the equal-mass square
central configuration.

The Hessian and the gap Laplacian are invariant under every reflection
symmetry $R$ of the configuration. The reflection induces a linear map on
the space in which the Hessian is defined, and this map satisfies $R^2=\mathbf I$. Its $+1$- and $-1$-eigenspaces consist, respectively, of vectors
left unchanged by $R$ and vectors whose sign is reversed by $R$.
Consequently, the associated quadratic forms split according to these two
eigenspaces, and the generalized eigenvalue problem reduces to two
independent, smaller problems, one on each subspace. We refer to this
reduction as the parity decomposition induced by $R$.

For four bodies, two cases arise, depending on how the axis of symmetry
meets the configuration: either it passes through two bodies, yielding a
kite configuration, or it passes through none, yielding an isosceles
trapezoidal configuration. The parity decomposition takes a slightly
different form in these two cases.

As an application, we prove that, for every rhombus central configuration,
the Hessian is positive definite after the rotational direction has been
removed. Thus every rhombus has Morse index zero and is nondegenerate
modulo rotations. This nondegeneracy was recently established by different
means in \cite{sun2025degeneracy}.

The paper is organized as follows.
Section~\ref{sec:central-configurations} recalls the classical
central-configuration equations and fixes notation. Section~\ref{sec:pairspace}
introduces the pair-space formulation of central configurations, proves
its equivalence with the classical equations, and derives the
corresponding matrix equations for central configurations.
Section~\ref{sec:hessian} constructs the constrained pair-space Hessian,
derives its reduced form, and decomposes it into the gap Laplacian and the
transverse Laplacian.
 In Section~\ref{sec:four-body-problem},
this construction is specialized to the planar four-body problem: we
establish the rank-one property of the reduced transverse-Laplacian, show
that the index of the reduced Hessian is at most two, and derive the
resulting positivity criteria. Section~\ref{sec:reflection_symmetry}
analyzes the role of reflection symmetries, showing how symmetry
considerations reduce the Hessian to block-diagonal form and simplify
the resulting index computations. 
Section~\ref{sec:rhombus} applies the general theory to the rhombus family,
showing that the reduced Hessian has Morse index zero and is nondegenerate
modulo the symmetries of the problem.  Consequently, rhombus central
configurations are nondegenerate local minima of the constrained
potential on the reduced configuration space.
\section{Central Configurations}
\label{sec:central-configurations}
Consider the planar $N$-body problem with masses $m_1, m_2, \dots, m_N > 0$ and total mass $M = \sum_{i=1}^N m_i$. Let  \(\mathbf r_i=(x_i,y_i)^T\in\mathbb R^2,\)
denote the position of the $i$-th body, and let the global position vector be
$\mathbf r=(\mathbf r_1^T,\mathbf r_2^T,\ldots,\mathbf r_N^T)^T
\in \mathbb R^{2N}$. 
We assume the center of mass is fixed at the origin, so $\sum_{i=1}^N m_i \mathbf{r}_i = 0$.

The classical Newtonian potential $U$ and the moment of inertia $I$ are functions of the position vector $\mathbf{r}$:
\begin{equation}
    U(\mathbf{r}) = \sum_{1 \le i < j \le N} \frac{m_i m_j}{|\mathbf{r}_i - \mathbf{r}_j|}, 
    \qquad
    I(\mathbf{r}) = \sum_{i=1}^N m_i |\mathbf{r}_i|^2.
\end{equation}

A central configuration is a collision-free configuration for which the
gravitational acceleration of each body is directed toward the center of
mass and is proportional to its distance from it.  Thus, there exists a
constant \(\omega^2>0\) such that
\begin{equation}\label{eq:classical_cc}
    \omega^2 m_i \mathbf{r}_i
    =
    \sum_{j\neq i}
    \frac{m_i m_j}{|\mathbf{r}_i-\mathbf{r}_j|^3}
    (\mathbf{r}_i-\mathbf{r}_j),
    \qquad i=1,\dots,N.
\end{equation}
The constant \(\omega\) is the angular velocity of the relative
equilibrium obtained by rigidly rotating the configuration.  By Euler's
theorem for homogeneous functions,
\[
    \omega^2=\frac{U}{I}.
\]
\section{Pair Space Formulation}
\label{sec:pairspace}

We reformulate the central-configuration equations in terms of the
relative displacement vectors between pairs of bodies.  In these variables,
the Newtonian potential is a sum of functions of individual pair vectors.
The price for this simplification is that the pair vectors cannot be chosen
independently: in order to represent an actual configuration of \(N\)
points in the plane, they must satisfy compatibility relations around
triangles.

\subsection{The Ambient Pair Space and the Augmented Functional}
For each unordered pair \(1\leq i<j\leq N\), let
\[
    \mathbf q_{ij}=\mathbf r_i-\mathbf r_j\in\mathbb R^2,
    \qquad
    q_{ij}=|\mathbf q_{ij}|.
\]
When ordered indices are useful, we adopt the convention
\[
    \mathbf q_{ji}=-\mathbf q_{ij},
    \qquad
    \mathbf q_{ii}=0.
\]
The {\bf ambient pair space} is
\[
    \mathcal P
    :=
    \prod_{1\leq i<j\leq N}\mathbb R^2
    \cong
    (\mathbb R^2)^{\binom N2}
    \cong
    \mathbb R^{N(N-1)}.
\]
Thus, an element of \(\mathcal P\) assigns a planar vector to each edge of
the complete graph on \(N\) vertices.  We order the edges
lexicographically,
\(
    (1,2),(1,3),\ldots,(1,N),(2,3),\ldots,(N-1,N),
\)
and identify such an assignment with the column vector
\[
    \mathbf q
    =
    \bigl(
        \mathbf q_{12}^{T},
        \mathbf q_{13}^{T},
        \ldots,
        \mathbf q_{1N}^{T},
        \mathbf q_{23}^{T},
        \ldots,
        \mathbf q_{N-1,N}^{T}
    \bigr)^{T}
    \in\mathcal P.
\]

The ambient space allows these pair vectors to vary independently. In
general, however, an element of $\mathcal P$ does not arise from a
configuration of $N$ points in the plane. The set of all pair vectors that
arise from an actual configuration forms a linear subspace
\[\mathcal C\subset\mathcal P,\] which we call the {\bf realizable pair
subspace}. Because $\mathcal C$ is linear, its tangent space at every
$\bq\in\mathcal C$ is naturally identified with $\mathcal C$ itself.
Thus an infinitesimal realizable perturbation of $\bq$ is represented by
a vector $\bu\in\mathcal C$, so that a nearby pair-vector configuration
has the form
\[
    \bq+\varepsilon\bu,
    \qquad \varepsilon\ \text{small}.
\]
If the perturbation is induced by infinitesimal body displacements
$\delta\br_1,\ldots,\delta\br_N$, then
\[
    \bu_{ij}=\delta\br_i-\delta\br_j.
\]

Pair vectors are unchanged by a common translation of all bodies. Fixing
body $1$ as a reference point, the $N-1$ vectors $\bq_{1k}$,
$k=2,\dots,N$, determine all relative positions. Hence
\[
    \dim\mathcal C=2(N-1).
\]

To obtain a non-redundant description of \(\mathcal C\), we partition the
pairs into two sets:
\begin{enumerate}
    \item \textbf{Independent pairs:} the \(N-1\) pairs incident to body
    \(1\), namely \(\mathbf q_{1k}\) for \(k=2,\ldots,N\).
    These form a spanning tree (a ``star" graph) that uniquely determines
    the positions of all bodies relative to body 1.

    \item \textbf{Dependent pairs:} the remaining
    $(N-1)(N-2)/2$ pairs not connected to body 1
    ($\mathbf{q}_{ij}$ for $2 \le i < j \le N$).
\end{enumerate}
The dependent pairs are determined by the independent pairs through the
independent triangle relations
\begin{equation}\label{eq:triangle_relations}
    \mathbf q_{1i}+\mathbf q_{ij}-\mathbf q_{1j}=0,
    \qquad
    2\leq i<j\leq N.
\end{equation}

Using the identity
\(\sum_{i=1}^N m_i|\mathbf r_i|^2=\frac1M\sum_{i<j}m_im_jq_{ij}^2\), the
Newtonian potential and moment of inertia, restricted to realizable
configurations, agree with the following functions defined on the entire
ambient pair space:
\begin{align}
    U_\pi(\mathbf q)
    &=\sum_{1\leq i<j\leq N}\frac{M\mu_{ij}}{q_{ij}},
    \label{eq:Upi}\\
    I_\pi(\mathbf q)
    &=\sum_{1\leq i<j\leq N}\mu_{ij}q_{ij}^2,
    \label{eq:Ipi}
\end{align}
where \(\mu_{ij}=m_im_j/M\).  For pair vectors associated with a centered
physical configuration,
\[
    U(\mathbf r)=U_\pi(\mathbf q(\mathbf r)),
    \qquad
    I(\mathbf r)= I_\pi(\mathbf q(\mathbf r)).
\]

We seek central configurations as critical points of the pair-space
functional \(U_\pi+\frac{\omega^2}{2}I_\pi\), restricted to the realizable
subspace \(\mathcal C\).  Rather than eliminate the dependent pair vectors,
we impose the independent relations \eqref{eq:triangle_relations} by
vector-valued Lagrange multipliers
\[
    \boldsymbol\phi_{ij}\in\mathbb R^2,
    \qquad
    2\leq i<j\leq N.
\]
The augmented functional on \(\mathcal P\) is therefore
\begin{equation}
\label{eq:Fambient}
    \mathcal F(\mathbf q,\boldsymbol \phi)
    =
    U_\pi(\mathbf q)
    +\frac{\omega^2}{2}I_\pi(\mathbf q)
    +\sum_{2\leq k<l\leq N}
    \boldsymbol\phi_{kl}\cdot
    \bigl(\mathbf q_{1k}+\mathbf q_{kl}-\mathbf q_{1l}\bigr).
\end{equation}

\subsection{The Pair-Space Equations}

Taking the gradient of the augmented functional \(\mathcal{F}\) with respect
to each pair vector \(\mathbf{q}_{ij}\) in the ambient space and setting it
to zero yields the exact pair-space central configuration equations:
\begin{equation}
\label{eq:pair_CC_force}
    \mu_{ij}\left(\omega^2 - \frac{M}{q_{ij}^3}\right)\mathbf{q}_{ij} + \mathbf{K}_{ij} = 0,
    \qquad \text{for all } 1 \le i < j \le N,
\end{equation}
where
\(
    \mathbf{K}_{ij}
    =
    \nabla_{\mathbf{q}_{ij}}
    \left(
        \sum_{2 \le k < l \le N}
        \boldsymbol\phi_{kl} \cdot (\mathbf{q}_{1k} + \mathbf{q}_{kl} - \mathbf{q}_{1l})
    \right)
\)
is the net constraint force acting on the pair \((i,j)\) generated by the
independent triangle relations.

Because we formulated the augmented functional using only the minimal
constraint basis associated with the multipliers \(\boldsymbol \phi_{kl}\)
(for \(2 \le k < l \le N\)), the constraint forces evaluate to explicit
algebraic sums. Differentiating the constraint sum directly with respect
to the dependent and independent pair vectors yields:
\begin{align}
    \text{Dependent pairs:} \quad & \mathbf{K}_{ij} = \boldsymbol \phi_{ij},
    \qquad \text{for } 2 \le i < j \le N, \label{eq:K_dependent} \\[6pt]
    \text{Independent pairs:} \quad & \mathbf{K}_{1j} = \sum_{k=j+1}^N \boldsymbol \phi_{jk} - \sum_{k=2}^{j-1} \boldsymbol \phi_{kj},
    \qquad \text{for } 2 \le j \le N. \label{eq:K_independent}
\end{align}
(In \eqref{eq:K_independent}, any sum where the lower bound exceeds the
upper bound is empty and evaluates to zero).

Thus, on each dependent pair \((i,j)\), the force associated with the
geometric constraint is simply the corresponding multiplier
\(\boldsymbol \phi_{ij}\).  On an independent pair \((1,j)\), the constraint force is
the signed sum of the multipliers associated with all triangle relations
containing \(\mathbf q_{1j}\).

\subsection{Equivalence to the Classical Equations}

We now show that the constrained pair-space equations are equivalent to the
classical central-configuration equations \eqref{eq:classical_cc} in
position space.

\begin{lemma}\label{lem:alpha_perp_equiv}
A configuration $\mathbf{r}$ satisfies the classical central configuration
equations \eqref{eq:classical_cc} if and only if
\begin{equation}\label{eq:alpha_perp_balance}
\sum_{j\ne i}\alpha_{ij}^\perp\mathbf{q}_{ij}=0
\qquad\text{for every } i=1,\dots,N,
\end{equation}
where 
\begin{equation}\label{eq:alpha_perp_def} 
 \alpha_{ij}^\perp=\mu_{ij}(\omega^2-M/q_{ij}^3) .
\end{equation} 
\end{lemma}

\begin{proof}
Expanding $\alpha_{ij}^\perp=\frac{m_im_j}{M}\bigl(\omega^2-M/q_{ij}^3\bigr)$,
multiplying \eqref{eq:alpha_perp_balance} by $M/m_i$, and using the
center-of-mass identity $\sum_{j\ne i}m_j\mathbf{q}_{ij}=M\mathbf{r}_i$, we
obtain
\[
\omega^2 M\mathbf{r}_i
-
\sum_{j\ne i}\frac{m_jM}{q_{ij}^3}(\mathbf{r}_i-\mathbf{r}_j)=0.
\]
Dividing by $M$ and multiplying by $m_i$ recovers
\eqref{eq:classical_cc} exactly. Each step is reversible, so the converse
implication follows in the same way.
\end{proof}

The preceding lemma rewrites the classical equations in a form adapted to
pair vectors. We now show that the augmented pair-space critical-point
equations are equivalent to the classical central-configuration equations.
\begin{proposition}
\label{prop:equivalence}
A pair-space configuration $\mathbf{q}$ satisfying the triangle relations is a critical point of the augmented functional $\mathcal{F}(\mathbf{q}, \boldsymbol \phi)$ for some multipliers $\boldsymbol \phi_{ij}$ if and only if its underlying position vector $\mathbf{r}$ satisfies the classical central configuration equations \eqref{eq:classical_cc}.
\end{proposition}

\begin{proof}
The pair-space equations \eqref{eq:pair_CC_force} read
\begin{equation}\label{eq:alpha_K}
\alpha_{ij}^\perp \mathbf{q}_{ij} = -\mathbf{K}_{ij},
\end{equation}
where we extend the notation skew-symmetrically by
$\mathbf{q}_{ji}=-\mathbf{q}_{ij}$, $\mathbf{K}_{ji}=-\mathbf{K}_{ij}$, and
$\alpha_{ji}^\perp=\alpha_{ij}^\perp$.

\medskip\noindent
\textbf{Claim: the constraint forces are internal, i.e.\ $\sum_{j\ne i}\mathbf{K}_{ij}=0$
for every $i$.} We verify this separately for the grounded body and for an
arbitrary ungrounded body.

For $i=1$, formula \eqref{eq:K_independent} gives
\[
\sum_{j=2}^N \mathbf{K}_{1j}
=
\underbrace{\sum_{j=2}^N\sum_{k=j+1}^N\boldsymbol \phi_{jk}}_{2\le j<k\le N}
-
\underbrace{\sum_{j=2}^N\sum_{k=2}^{j-1}\boldsymbol \phi_{kj}}_{2\le k<j\le N}.
\]
Both double sums range over the same set of index pairs $2\le j<k\le N$
once we relabel $k\leftrightarrow j$ in the second sum; hence the two sums
cancel and $\sum_{j=2}^N \mathbf{K}_{1j}=0$.

For $i\ge2$, separate the term $j=1$ from the rest:
\[
    \sum_{j\ne i}\mathbf{K}_{ij}
    =
    \mathbf{K}_{i1}+\sum_{\substack{j=2\\j\ne i}}^N \mathbf{K}_{ij}.
\]

Since $\mathbf{K}_{i1}=-\mathbf{K}_{1i}$, applying \eqref{eq:K_independent} at the index
$i$ (in place of $j$) gives
\[
    \mathbf{K}_{i1}
    =
    -\sum_{k=i+1}^N\boldsymbol \phi_{ik}+\sum_{k=2}^{i-1}\boldsymbol \phi_{ki}.
\]

For the remaining sum, recall that \eqref{eq:K_dependent} states
$\mathbf{K}_{ab}=\boldsymbol \phi_{ab}$ for every pair $2\le a<b\le N$.  For $j>i$, apply this
with $(a,b)=(i,j)$ to get $\mathbf{K}_{ij}=\boldsymbol \phi_{ij}$ directly.  For $j<i$, apply it
instead with $(a,b)=(j,i)$ to get $\mathbf{K}_{ji}=\boldsymbol \phi_{ji}$, and then invoke
skew-symmetry $\mathbf{K}_{ij}=-\mathbf{K}_{ji}$ to obtain $\mathbf{K}_{ij}=-\boldsymbol \phi_{ji}$.  Hence
\[
    \sum_{\substack{j=2\\j\ne i}}^N \mathbf{K}_{ij}
    =
    \sum_{j=i+1}^N\boldsymbol \phi_{ij}
    -\sum_{j=2}^{i-1}\boldsymbol \phi_{ji}.
\]

Renaming the dummy index $k\to j$ in the expression for $\mathbf{K}_{i1}$ and
adding the two displays,
\[
    \sum_{j\ne i}\mathbf{K}_{ij}
    =
    \left(-\sum_{j=i+1}^N\boldsymbol \phi_{ij}+\sum_{j=2}^{i-1}\boldsymbol \phi_{ji}\right)
    +
    \left(\sum_{j=i+1}^N\boldsymbol \phi_{ij}-\sum_{j=2}^{i-1}\boldsymbol \phi_{ji}\right)
    =0.
\]
This proves the claim for $i\ge2$.

\medskip
Summing \eqref{eq:alpha_K} over $j\ne i$ and applying the claim to the
right-hand side gives
\[
\sum_{j\ne i}\alpha_{ij}^\perp\mathbf{q}_{ij}=0
\qquad\text{for every }i,
\]
which by Lemma~\ref{lem:alpha_perp_equiv} is equivalent to the classical
equations \eqref{eq:classical_cc}. This proves the forward implication.

\medskip
Conversely, suppose \eqref{eq:classical_cc} holds. By
Lemma~\ref{lem:alpha_perp_equiv}, $\sum_{k\ne j}\alpha_{jk}^\perp\mathbf{q}_{jk}=0$
at every node $j$. We must produce multipliers $\boldsymbol \phi_{ij}$ for which
$\alpha_{ij}^\perp\mathbf{q}_{ij}=-\mathbf{K}_{ij}$ holds for every pair. We construct
them by forcing this equation to hold on the dependent pairs: for
$2\le i<j\le N$, define
\begin{equation}\label{eq:phi_construct}
\boldsymbol \phi_{ij}=-\alpha_{ij}^\perp\mathbf{q}_{ij}.
\end{equation}
Since $\mathbf{K}_{ij}=\boldsymbol \phi_{ij}$ for dependent pairs by \eqref{eq:K_dependent}, this
choice trivially gives $\alpha_{ij}^\perp\mathbf{q}_{ij}=-\mathbf{K}_{ij}$ for all
dependent pairs.

It remains to check that the same multipliers reproduce the correct force
$\mathbf{K}_{1i}$ on each independent pair.  Applying \eqref{eq:alpha_perp_balance}
at the node $i\ge2$, and isolating the term $j=1$, gives
\[
    \alpha_{i1}^\perp\mathbf{q}_{i1}
    +\sum_{\substack{j=2\\j\ne i}}^N\alpha_{ij}^\perp\mathbf{q}_{ij}=0.
\]
Using $\mathbf{q}_{i1}=-\mathbf{q}_{1i}$ and
$\alpha_{i1}^\perp=\alpha_{1i}^\perp$ to move the $j=1$ term to the other
side,
\[
    \alpha_{1i}^\perp\mathbf{q}_{1i}
    =\sum_{\substack{j=2\\j\ne i}}^N\alpha_{ij}^\perp\mathbf{q}_{ij}.
\]
Splitting the sum into indices $j<i$ and $j>i$, and applying
skew-symmetry $\alpha_{ij}^\perp\mathbf{q}_{ij}=-\alpha_{ji}^\perp\mathbf{q}_{ji}$
to the first part, this becomes
\[
    \alpha_{1i}^\perp\mathbf{q}_{1i}
    =-\sum_{j=2}^{i-1}\alpha_{ji}^\perp\mathbf{q}_{ji}
    +\sum_{j=i+1}^N\alpha_{ij}^\perp\mathbf{q}_{ij}.
\]
Substituting the definition \eqref{eq:phi_construct}---replacing
$-\alpha_{ji}^\perp\mathbf{q}_{ji}$ by $\boldsymbol \phi_{ji}$ and
$\alpha_{ij}^\perp\mathbf{q}_{ij}$ by $-\boldsymbol \phi_{ij}$---gives
\[
    \alpha_{1i}^\perp\mathbf{q}_{1i}
    =\sum_{j=2}^{i-1}\boldsymbol \phi_{ji}-\sum_{j=i+1}^N\boldsymbol \phi_{ij}.
\]
Multiplying by $-1$ and comparing with \eqref{eq:K_independent}, evaluated
at the index $i$ in place of $j$,
\[
    -\alpha_{1i}^\perp\mathbf{q}_{1i}
    =\sum_{j=i+1}^N\boldsymbol \phi_{ij}-\sum_{j=2}^{i-1}\boldsymbol \phi_{ji}
    =\mathbf{K}_{1i}.
\]
Thus $\alpha_{1i}^\perp\mathbf{q}_{1i}=-\mathbf{K}_{1i}$ holds for independent pairs
as well.  Since the multipliers \eqref{eq:phi_construct} satisfy the
augmented pair-space equations for all pairs simultaneously, the equivalence
is proven.

\end{proof}

\subsection{Laplacian Matrix Formulation of Central Configurations}

Recall that  \(\mathbf r_i=(x_i,y_i)^T\in\mathbb R^2,\)
and define the global coordinate vectors
\[
\mathbf X=(x_1,\ldots,x_N)^T,
\qquad
\mathbf Y=(y_1,\ldots,y_N)^T
\]
in $\mathbb R^N$. We now recast the summation identity
\eqref{eq:alpha_perp_balance} of
Lemma~\ref{lem:alpha_perp_equiv} as a pair of matrix equations for
$\mathbf X$ and $\mathbf Y$.

Define the matrix
$\ell_{\mathrm{full}}\in\mathbb R^{N\times N}$ by
\begin{equation}
\label{eq:Lfull_def}
(\ell_{\mathrm{full}})_{ij}
=
\begin{cases}
-\alpha_{ij}^\perp, & i\neq j,\\
\displaystyle\sum_{k\neq i}\alpha_{ik}^\perp, & i=j,
\end{cases}
\end{equation}
where
\(\alpha_{ij}^\perp=\mu_{ij}\left(\omega^2-\frac{M}{q_{ij}^3}\right).\)
Thus $\ell_{\mathrm{full}}$ is symmetric and has zero row sums. It is
the Laplacian of the complete graph on the bodies with signed edge
weights $\alpha_{ij}^\perp$; see, for example,
\cite[Chapter~18]{gallier2020linear} for the standard graph-Laplacian
construction.
The balance equations may now be encoded as the kernel condition of a
signed weighted graph Laplacian. The following proposition makes this
equivalence precise.

\begin{proposition}[Matrix Formulation]
\label{prop:laplacian_kernel}
A configuration $\mathbf{r}$ is a central configuration if and only if its
global coordinate vectors $\mathbf{X}$ and $\mathbf{Y}$ both lie in the
kernel of $\ell_{\mathrm{full}}$:
\begin{equation}
\label{eq:matrix_CC}
    \ell_{\mathrm{full}} \mathbf{X} = 0 \quad \text{and} \quad
    \ell_{\mathrm{full}} \mathbf{Y} = 0.
\end{equation}
\end{proposition}

\begin{proof}
By the definition of the signed weighted Laplacian $\ell_{\mathrm{full}}$,
the diagonal entry in its $i$-th row is the sum $\sum_{j\neq i}\alpha_{ij}^\perp$
of the coefficients $\alpha_{ij}^\perp$ incident to body $i$, whereas the
off-diagonal entry in column $j$ is $-\alpha_{ij}^\perp$.  Thus, the $i$-th
component of $\ell_{\mathrm{full}}\mathbf X$ is obtained by multiplying
the diagonal entry by $x_i$ and each off-diagonal entry by the
corresponding coordinate $x_j$.

Since $x_i$ is independent of the summation index $j$, it may be
placed inside the first sum. Combining the resulting sums gives
\[
\begin{aligned}
(\ell_{\mathrm{full}}\mathbf X)_i
&=
\underbrace{\left(\sum_{j\neq i}\alpha_{ij}^\perp\right)x_i}
_{\text{diagonal term}}
-
\underbrace{\sum_{j\neq i}\alpha_{ij}^\perp x_j}
_{\text{off-diagonal terms}} \\
&=
\sum_{j\neq i}\alpha_{ij}^\perp(x_i-x_j).
\end{aligned}
\]
An identical computation for the $y$-coordinates gives
\[
(\ell_{\mathrm{full}}\mathbf Y)_i
=
\sum_{j\neq i}\alpha_{ij}^\perp(y_i-y_j).
\]
Therefore,
\[
\ell_{\mathrm{full}}\mathbf X=0
\qquad\text{and}\qquad
\ell_{\mathrm{full}}\mathbf Y=0
\]
if and only if
\[
\sum_{j\neq i}\alpha_{ij}^\perp
\begin{pmatrix}
x_i-x_j\\
y_i-y_j
\end{pmatrix}
=
\sum_{j\neq i}\alpha_{ij}^\perp
(\mathbf r_i-\mathbf r_j)
=
0,
\qquad i=1,\ldots,N.
\]
By \eqref{eq:alpha_perp_balance} of
Lemma~\ref{lem:alpha_perp_equiv}, these equations are equivalent to the
classical central-configuration equations
\eqref{eq:classical_cc}.
\end{proof}

\begin{remark}[Block-vector formulation]
Since
\(\mathbf r=(x_1,y_1,\ldots,x_N,y_N)^T\in\mathbb R^{2N},\)
then the block Laplacian
\[
L_{\mathrm{full}}
=
\ell_{\mathrm{full}}\otimes \mathbf I_2
\]
satisfies
\[
(L_{\mathrm{full}}\mathbf r)_i
=
\sum_{j\neq i}\alpha_{ij}^{\perp}
(\mathbf r_i-\mathbf r_j),
\qquad i=1,\ldots,N.
\]
Consequently, the central-configuration equations are equivalently
expressed as
\[
L_{\mathrm{full}}\mathbf r=0.
\]
\end{remark}

\begin{remark}[Translational degeneracy and grounding]\label{rmk:transl-degeneracy-and-grounding}
Since every row of $\ell_{\mathrm{full}}$ sums to zero,
\[
\ell_{\mathrm{full}}\mathbf{1}=0,
\qquad
\mathbf{1}=(1,\ldots,1)^T\in\mathbb{R}^N.
\]
Because $\ell_{\mathrm{full}}$ acts identically and independently on
$\mathbf{X}$ and $\mathbf{Y}$, this single scalar kernel relation
accounts for both translational zero-modes: it gives
$\ell_{\mathrm{full}}\mathbf{X}=0$ whenever
$\mathbf{X}=c\,\mathbf{1}$, corresponding to a rigid
$x$-translation, and likewise
$\ell_{\mathrm{full}}\mathbf{Y}=0$ whenever
$\mathbf{Y}=c\,\mathbf{1}$, corresponding to a rigid
$y$-translation.

To remove these translational modes, we ground body $1$ and pass to
the pair vectors
\[
\mathbf{q}_{1j}=\mathbf{r}_1-\mathbf{r}_j,
\qquad j=2,\ldots,N.
\]

At the matrix level, this corresponds to deleting the row and column of
\(\ell_{\mathrm{full}}\) associated with body \(1\). The resulting
grounded matrix \(\ell\) acts on the relative coordinates
\(\mathbf q_{1j}=\mathbf r_1-\mathbf r_j\), \(j=2,\ldots,N\), and is
independent of the absolute position of the configuration. The grounding
is only a coordinate reduction: the configuration remains centered at the
origin, while the matrix \(\ell\) is expressed in the relative
coordinates \(\mathbf q_{1j}\). The matrix \(\ell\) will reappear
throughout the remainder of the paper as the central object governing the
reduced Hessian once the translational degeneracy has been removed.
\end{remark}
\section{The Hessian of the constrained pair-space functional}
\label{sec:hessian}

We now turn to the second variation of the pair-space functional at a planar 
central configuration. In general, computing the Hessian of a constrained 
variational problem requires accounting for the curvature of the constraint 
manifold. Here, however, the realizability subspace \(\mathcal{C}\) is completely 
flat. Because the triangle relations are strictly linear, their second 
derivatives vanish, and they contribute no geometric curvature terms to the 
Hessian.

Consequently, the second variation of the fully constrained problem coincides 
exactly with the ambient Hessian of the free pair-space functional
\begin{equation}\label{eq:Ffreehessiana}
F = U_{\pi}+\frac{\omega^2}{2}I_\pi.
\end{equation}
To evaluate the second variation at a central configuration, it therefore 
suffices to compute the Hessian of \(F\) in the full ambient pair space 
\(\mathcal{P}\), and then restrict the resulting quadratic form to the admissible 
tangent space \(T_{\bq}\mathcal{C}\).

\subsection{Block-diagonal structure}

The first step in evaluating this ambient Hessian is to recognize its natural 
decoupling. Since
\[
F = \sum_{1 \le i < j \le N} F_{ij},
\qquad
F_{ij}=\frac{M\mu_{ij}}{q_{ij}}+\frac{\omega^2}{2}\mu_{ij}q_{ij}^2,
\]
and each $F_{ij}$ depends only on $\bq_{ij}=(x_{ij},y_{ij})$, all mixed
partial derivatives between different pairs vanish:
\[
\frac{\partial^2F}{\partial x_{ij}\partial x_{kl}}
=
\frac{\partial^2F}{\partial x_{ij}\partial y_{kl}}
=
\frac{\partial^2F}{\partial y_{ij}\partial y_{kl}}
=0
\qquad
\text{whenever }\{i,j\}\neq\{k,l\}.
\]
Hence, in the standard lexicographical coordinate ordering of the pairs $(12, 13, \dots, (N-1)N)$, the ambient Hessian is block diagonal, consisting of $K = N(N-1)/2$ diagonal blocks of size $2 \times 2$:
\begin{equation}\label{eq:Hessian_block_diag}
D_{\mathcal P}^2 F
=
\begin{pmatrix}
\mathbf A_{12} & \mathbf 0 & \mathbf 0 & \cdots & \mathbf 0 \\
\mathbf 0 & \mathbf A_{13} & \mathbf 0 & \cdots & \mathbf 0 \\
\mathbf 0 & \mathbf 0 & \mathbf A_{14} & \cdots & \mathbf 0 \\
\vdots & \vdots & \vdots & \ddots & \vdots \\
\mathbf 0 & \mathbf 0 & \mathbf 0 & \cdots & \mathbf A_{(N-1)N}
\end{pmatrix},
\end{equation}
where each $\mathbf A_{ij}$ is a $2\times 2$ symmetric matrix and each $\mathbf 0$ denotes a $2\times 2$ zero block.

\subsection{The explicit \texorpdfstring{$2\times2$}{2x2} blocks}
Having established that the ambient Hessian is composed entirely of diagonal 
blocks, we now determine their precise algebraic form. The next proposition 
gives the explicit form of each diagonal block of the ambient Hessian.

\begin{proposition}\label{prop:Aij}
For each pair $i<j$, the corresponding $2\times2$ block of the ambient Hessian is
\begin{equation}\label{eq:Aij_formula}
\mathbf A_{ij}
=
D^2_{(x_{ij},y_{ij})} F_{ij}
=
\alpha_{ij}^\perp
\begin{pmatrix}
1 & 0\\
0 & 1
\end{pmatrix}
+
\frac{3M\mu_{ij}}{q_{ij}^5}
\begin{pmatrix}
x_{ij}^2 & x_{ij}y_{ij}\\
x_{ij}y_{ij} & y_{ij}^2
\end{pmatrix},
\end{equation}
where $\alpha_{ij}^\perp = \mu_{ij}\left(\omega^2-\frac{M}{q_{ij}^3}\right)$ is the transverse coefficient defined in Eq.~\eqref{eq:alpha_perp_def}.
\end{proposition}

\begin{proof}
Since $q_{ij}^2=x_{ij}^2+y_{ij}^2$, we compute the first derivatives:
\[
\frac{\partial}{\partial x_{ij}}
\left(\frac{M\mu_{ij}}{q_{ij}}\right)
=
-\frac{M\mu_{ij}x_{ij}}{q_{ij}^3},
\qquad
\frac{\partial}{\partial y_{ij}}
\left(\frac{M\mu_{ij}}{q_{ij}}\right)
=
-\frac{M\mu_{ij}y_{ij}}{q_{ij}^3}.
\]
Differentiating once more gives the second derivative of the potential with respect to $x_{ij}$:
\[
\frac{\partial^2}{\partial x_{ij}^2}
\left(\frac{M\mu_{ij}}{q_{ij}}\right)
=
-\frac{M\mu_{ij}}{q_{ij}^3}
+\frac{3M\mu_{ij}x_{ij}^2}{q_{ij}^5},
\]
and, by the symmetry $x_{ij}\leftrightarrow y_{ij}$, the corresponding derivative with respect to $y_{ij}$:
\[
\frac{\partial^2}{\partial y_{ij}^2}
\left(\frac{M\mu_{ij}}{q_{ij}}\right)
=
-\frac{M\mu_{ij}}{q_{ij}^3}
+\frac{3M\mu_{ij}y_{ij}^2}{q_{ij}^5}.
\]
The mixed partial derivative is
\[
\frac{\partial^2}{\partial x_{ij}\partial y_{ij}}
\left(\frac{M\mu_{ij}}{q_{ij}}\right)
=
\frac{3M\mu_{ij}x_{ij}y_{ij}}{q_{ij}^5}.
\]

For the quadratic (centrifugal) term, the second derivatives are simply constants:
\[
\frac{\partial^2}{\partial x_{ij}^2}
\left(\frac{\omega^2}{2}\mu_{ij}q_{ij}^2\right)
=
\frac{\partial^2}{\partial y_{ij}^2}
\left(\frac{\omega^2}{2}\mu_{ij}q_{ij}^2\right)
=
\omega^2\mu_{ij},
\]
and the mixed partial is zero:
\[
\frac{\partial^2}{\partial x_{ij}\partial y_{ij}}
\left(\frac{\omega^2}{2}\mu_{ij}q_{ij}^2\right)
=
0.
\]

Adding the two contributions yields the entries of the Hessian block:
\[
\frac{\partial^2 F_{ij}}{\partial x_{ij}^2}
=
\mu_{ij}\left(\omega^2-\frac{M}{q_{ij}^3}\right)
+\frac{3M\mu_{ij}x_{ij}^2}{q_{ij}^5}
=
\alpha_{ij}^\perp+\frac{3M\mu_{ij}x_{ij}^2}{q_{ij}^5},
\]
\[
\frac{\partial^2 F_{ij}}{\partial y_{ij}^2}
=
\alpha_{ij}^\perp+\frac{3M\mu_{ij}y_{ij}^2}{q_{ij}^5},
\qquad
\frac{\partial^2 F_{ij}}{\partial x_{ij}\partial y_{ij}}
=
\frac{3M\mu_{ij}x_{ij}y_{ij}}{q_{ij}^5}.
\]
Separating the $\alpha_{ij}^\perp$ terms from the coordinate-dependent terms immediately yields the matrix decomposition \eqref{eq:Aij_formula}.
\end{proof}

\subsection{Eigenvalues and Spectral Decomposition}
To analyze the definiteness of the ambient Hessian, it is highly convenient 
to diagonalize each block $\mathbf{A}_{ij}$ individually. This local spectral 
decomposition naturally separates the variations along the edge from those 
orthogonal to it. For each pair $i<j$, let
\[
\hat{\bq}_{ij}=\frac{(x_{ij},y_{ij})}{q_{ij}},
\qquad
\hat{\bq}_{ij}^{\perp}=\frac{(-y_{ij},x_{ij})}{q_{ij}}.
\]
These are, respectively, the unit radial and transverse directions associated
with the pair vector $\mathbf{q}_{ij}$.

In the orthonormal basis $\{\hat{\bq}_{ij},\hat{\bq}_{ij}^{\perp}\}$, the matrix $\mathbf{A}_{ij}$ is diagonal, with eigenvalues $\alpha_{ij}^\parallel$ and $\alpha_{ij}^\perp$ associated respectively to the eigenvectors $\hat{\bq}_{ij}$ and $\hat{\bq}_{ij}^\perp$. More precisely, writing the rank-one projectors explicitly as $\hat{\bq}_{ij}\hat{\bq}_{ij}^{T}$ and $\hat{\bq}_{ij}^{\perp}(\hat{\bq}_{ij}^{\perp})^{T}$,
\begin{equation}\label{eq:Aij_diag}
\mathbf{A}_{ij}
=
\alpha_{ij}^{\parallel}\,
\hat{\bq}_{ij}\hat{\bq}_{ij}^{T}
+
\alpha_{ij}^{\perp}\,
\hat{\bq}_{ij}^{\perp}(\hat{\bq}_{ij}^{\perp})^{T},
\end{equation}
where $\alpha_{ij}^{\perp}=\mu_{ij}\left(\omega^2-\frac{M}{q_{ij}^3}\right)$ is the transverse coefficient defined in Eq.~\eqref{eq:alpha_perp_def}, and the radial coefficient $\alpha_{ij}^{\parallel}$ is given by:
\begin{equation}\label{eq:alpha_rad}
\alpha_{ij}^{\parallel}
=
\alpha_{ij}^{\perp} + \frac{3M\mu_{ij}}{q_{ij}^3}
=
\mu_{ij}\left(\omega^2+\frac{2M}{q_{ij}^3}\right).
\end{equation}
Thus, $\alpha_{ij}^{\parallel}$ is always strictly positive, whereas the sign of $\alpha_{ij}^{\perp}$ depends on the geometry of the configuration.

Expanding the projectors $\hat{\bq}_{ij}\hat{\bq}_{ij}^{T}$ and $\hat{\bq}_{ij}^{\perp}(\hat{\bq}_{ij}^{\perp})^{T}$ into Cartesian coordinates, $\mathbf{A}_{ij}$ admits the explicit matrix expression:
\begin{equation}\label{eq:Aij_diag_explicit}
\mathbf{A}_{ij}
=
\frac{\alpha_{ij}^{\parallel}}{q_{ij}^2}
\begin{pmatrix}
x_{ij}^2 & x_{ij}y_{ij}\\
x_{ij}y_{ij} & y_{ij}^2
\end{pmatrix}
+
\frac{\alpha_{ij}^{\perp}}{q_{ij}^2}
\begin{pmatrix}
y_{ij}^2 & -x_{ij}y_{ij}\\
-x_{ij}y_{ij} & x_{ij}^2
\end{pmatrix}.
\end{equation}

Because the sum of these two coordinate matrices (scaled by $1/q_{ij}^2$) is exactly the $2 \times 2$ identity matrix, substituting $\alpha_{ij}^{\parallel} = \alpha_{ij}^{\perp} + 3M\mu_{ij}/q_{ij}^3$ into \eqref{eq:Aij_diag_explicit} exactly recovers the previous Cartesian form \eqref{eq:Aij_formula}.

\begin{remark}[Sign pattern for planar four-body configurations]
Suppose $N=4$, and set $\sigma=\omega^2/M$, $s_{ij}=q_{ij}^{-3}$. From
\eqref{eq:alpha_perp_def},
\[
    \alpha_{ij}^{\perp}
    =\mu_{ij}\Bigl(\omega^2-\tfrac{M}{q_{ij}^3}\Bigr)
    =m_i m_j(\sigma-s_{ij}),
\]
so $\alpha_{ij}^{\perp}>0$ iff $q_{ij}>\sigma^{-1/3}$. By
\eqref{eq:Aij_diag}, each block $\mathbf A_{ij}$ has eigenvalues
$\alpha_{ij}^{\parallel}>0$ and $\alpha_{ij}^{\perp}$, so the signature
of $D_{\mathcal P}^2F$ is $(6+p,\,6-p)$, where $p$ is the number of
edges with $\alpha_{ij}^\perp>0$.

\medskip\noindent\textbf{Convex case.}
Label the bodies $1,2,3,4$ cyclically on the convex hull, so the
diagonals are $\{1,3\},\{2,4\}$ and the remaining four edges are the
boundary. By Proposition~2.5 of Long--Sun \cite{long2002four}, $\sigma$
lies strictly between the diagonal and boundary values of $s_{ij}$,
i.e.\ each diagonal is strictly longer than every boundary edge. Hence
$\alpha_{ij}^\perp>0$ exactly on the two diagonals ($p=2$), giving
signature $(8,4)$.

\medskip\noindent\textbf{Concave case.}
Suppose body $4$ lies inside the triangle formed by bodies $1,2,3$, so
the inner edges are $\{1,4\},\{2,4\},\{3,4\}$ and the outer edges are
$\{1,2\},\{2,3\},\{1,3\}$. By Proposition~2.6 of Long--Sun
\cite{long2002four}, each inner edge is strictly shorter than every
outer edge, so $\sigma$ lies strictly between the corresponding
$s_{ij}$ values. Hence $\alpha_{ij}^\perp>0$ exactly on the three outer
edges ($p=3$), giving signature $(9,3)$.
\end{remark}

\subsection{Derivation of the Transformation Matrix \texorpdfstring{$B$}{B} and the Restricted Hessian}

The $N(N-1)/2$ pair vectors are coupled by the linear triangle constraints. 
These constraints dictate that the realizability subspace $\mathcal{C}$ has 
dimension $2(N-1)$. A natural choice for the independent variables is the set of 
$N-1$ pairs sharing body 1: $\mathbf{q}_{12}, \mathbf{q}_{13}, \dots, \mathbf{q}_{1N}$. 

Rearranging the triangle relations \eqref{eq:triangle_relations} expresses
the remaining dependent pairs purely as linear combinations of the
independent ones:
\begin{equation}
\label{eq:dep_N}
\mathbf{q}_{ij} = -\mathbf{q}_{1i} + \mathbf{q}_{1j}, \qquad \text{for } 2 \le i < j \le N.
\end{equation}

Let the full vector of pair coordinates in $\mathcal P$ be ordered such that the 
independent pairs come first, followed by the dependent pairs grouped sequentially 
by their first index ($i = 2, 3, \dots, N-1$), with the range of the second index 
made explicit in each block:
\begin{equation}
\label{eq:x_ordering}
    \mathbf{x} = \Big( (\mathbf{q}_{1j})_{j=2}^{N} \;\Big|\; (\mathbf{q}_{2j})_{j=3}^{N} \;\Big|\; (\mathbf{q}_{3j})_{j=4}^{N} \;\Big|\; \dots \;\Big|\; \mathbf{q}_{(N-1)N} \Big)^T.
\end{equation}
Let $\mathbf{u} = (\mathbf{q}_{12}, \mathbf{q}_{13}, \dots, \mathbf{q}_{1N})^T$ be the 
state vector of our chosen independent variables in $\mathbb{R}^{2(N-1)}$. 
The linear mapping $\mathbf{x}=B\mathbf{u}$ is defined by an
$N(N-1)\times 2(N-1)$ transformation matrix $B$, which decomposes into
horizontal block rows corresponding to these pair groups.

Let $\mathbf{I}_{2k}$ denote the $2k \times 2k$ identity matrix, and write $\mathbf{I}_2$ 
for the single $2\times2$ identity block. Let $\mathbf{1}_k \in \mathbb{R}^k$ denote the 
all-ones column vector, so that $\mathbf{1}_k \otimes \mathbf{I}_2$ is the $2k \times 2$ 
block-column matrix formed by stacking $k$ copies of $\mathbf{I}_2$. Let 
$\mathbf{0}_{2k \times 2m}$ denote a zero matrix of the specified scalar dimensions. 
With this notation, the hierarchical structure of $B$ is given by:
\[
B =
\begin{pmatrix}
\mathbf{I}_{2(N-1)} \\[4pt]
\hline \\[-10pt]
-(\mathbf{1}_{N-2} \otimes \mathbf{I}_2) & \mathbf{I}_{2(N-2)} \\[4pt]
\hline \\[-10pt]
\mathbf{0}_{2(N-3) \times 2} & -(\mathbf{1}_{N-3} \otimes \mathbf{I}_2) & \mathbf{I}_{2(N-3)} \\[4pt]
\hline \\[-10pt]
\mathbf{0}_{2(N-4) \times 4} & -(\mathbf{1}_{N-4} \otimes \mathbf{I}_2) & \mathbf{I}_{2(N-4)} \\[4pt]
\hline \\[-10pt]
\vdots & \vdots & \ddots \\[4pt]
\hline \\[-10pt]
\mathbf{0}_{2 \times 2(N-3)} & -\mathbf{I}_2 & \mathbf{I}_2
\end{pmatrix}.
\]
The block row governing the $N-i$ pairs starting with body $i$ (for $i = 2, \dots, N-1$) 
is the $(i-1)$-th block row below the top partition. In that row, the $-\mathbf{I}_2$ 
block sits at block-column $(i-1)$, effectively subtracting $\mathbf{q}_{1i}$ from the 
remaining independent pairs, exactly as prescribed by Eq.~\eqref{eq:dep_N}.

The intrinsic Hessian $H_{\mathcal{C}}$ on the subspace $\mathcal{C}$ is 
obtained via the congruence transformation $H_{\mathcal{C}} = B^T (D_{\mathcal{P}}^2 F) B$. 
Because $D_{\mathcal{P}}^2 F$ is composed entirely of $2 \times 2$ diagonal blocks 
$\mathbf{A}_{ij}$ (defined for $i<j$; we extend the definition by symmetry, 
$\mathbf{A}_{ji} := \mathbf{A}_{ij}$, so that the sums below are unambiguous), the 
block multiplication $B^T(D_{\mathcal P}^2F)B$ aggregates these terms in a simple 
pattern: each diagonal block of $H_{\mathcal C}$ collects the contributions of every 
pair incident to the corresponding body, while each off-diagonal block picks up exactly 
the one pair directly joining the two bodies, with a sign flip inherited from the 
$-\mathbf{I}_2$ entries of $B$.
The restricted Hessian $H_{\mathcal C}$ is the $2(N-1)\times 2(N-1)$ matrix,
logically partitioned into an $(N-1)\times(N-1)$ grid of $2\times 2$ blocks

\begin{equation}\label{eq:HC_entrywise}
(H_{\mathcal C})_{a-1,a-1}
=
\sum_{k\neq a}\mathbf{A}_{ak},
\qquad
(H_{\mathcal C})_{a-1,b-1}
=
-\mathbf{A}_{ab}
\quad (a\neq b),
\end{equation}

for $a,b \in \{2,\dots,N\}$, where each $\mathbf{A}_{ij}$ is the $2\times 2$
Hessian block defined in Eq.~\eqref{eq:Aij_formula}, extended by symmetry via
$\mathbf{A}_{ji}:=\mathbf{A}_{ij}$. Written out explicitly, this reads:
\[
H_{\mathcal{C}} =
\begin{pmatrix}
\sum\limits_{k \neq 2} \mathbf{A}_{2k} & -\mathbf{A}_{23} & -\mathbf{A}_{24} & \cdots & -\mathbf{A}_{2N} \\[10pt]
-\mathbf{A}_{23} & \sum\limits_{k \neq 3} \mathbf{A}_{3k} & -\mathbf{A}_{34} & \cdots & -\mathbf{A}_{3N} \\[10pt]
-\mathbf{A}_{24} & -\mathbf{A}_{34} & \sum\limits_{k \neq 4} \mathbf{A}_{4k} & \cdots & -\mathbf{A}_{4N} \\
\vdots & \vdots & \vdots & \ddots & \vdots \\
-\mathbf{A}_{2N} & -\mathbf{A}_{3N} & -\mathbf{A}_{4N} & \cdots & \sum\limits_{k \neq N} \mathbf{A}_{Nk}
\end{pmatrix}.
\]

This explicit block matrix represents the second variation of the free 
functional, evaluated strictly on the admissible configurations defined by 
the triangle constraints.

\begin{remark}[Connection to Graph Laplacians]
The explicit matrix $H_{\mathcal{C}}$ takes the exact form of a reduced, 
matrix-weighted graph Laplacian. Consider the $N$ bodies as vertices of a complete 
graph $K_N$, where each edge $\{i,j\}$ is assigned the $2\times 2$ matrix weight 
$\mathbf{A}_{ij} = \mathbf{A}_{ji}$. The full block-Laplacian $\mathbf{L}$ of this 
graph has diagonal blocks $\mathbf{L}_{ii} = \sum_{k \neq i} \mathbf{A}_{ik}$ and 
off-diagonal blocks $\mathbf{L}_{ij} = -\mathbf{A}_{ij}$.

Choosing $\mathbf{q}_{12}, \dots, \mathbf{q}_{1N}$ as the independent variables is 
physically equivalent to placing body 1 at the origin and measuring the 
positions of bodies $2, \dots, N$ relative to it. Mathematically, this is 
achieved by deleting the row and column blocks associated with node 1 from 
the full Laplacian $\mathbf{L}$. The resulting principal submatrix is 
precisely the restricted Hessian $H_{\mathcal{C}}$. 

Note that this identification is exact as a matrix identity, but it does not carry
over the usual positivity guarantee for grounded scalar Laplacians: for
a connected graph with nonnegative edge weights, the grounded Laplacian
is positive definite. Here, however, the matrix weights $\mathbf{A}_{ij}$
are generic Hessian blocks of the potential and need not even  be
positive semidefinite, so that the reduced Hessian need not be positive definite. 
\end{remark}

\subsection{Splitting the Restricted Hessian}
\label{subsec:hessian_splitting}

Recall from Equation~\eqref{eq:alpha_rad} that each $2\times 2$ block
admits the spectral decomposition
\[
\mathbf{A}_{ij}
= \alpha_{ij}^{\parallel}\,\hat{\bq}_{ij}\hat{\bq}_{ij}^{T}
+ \alpha_{ij}^{\perp}\,\hat{\bq}_{ij}^{\perp}(\hat{\bq}_{ij}^{\perp})^{T}.
\]
Define the {\bf radial gaps}
\begin{equation}\label{eq:delta_def}
\delta_{ij} := \alpha_{ij}^{\parallel}-\alpha_{ij}^{\perp}
=
\frac{3M\mu_{ij}}{q_{ij}^3} > 0,
\end{equation}
which is always strictly positive since $\alpha_{ij}^\parallel>0$
dominates $\alpha_{ij}^\perp$ by exactly this amount.

Substituting $\alpha_{ij}^{\parallel} = \alpha_{ij}^{\perp} + \delta_{ij}$ into the decomposition of $\mathbf{A}_{ij}$ and using the identity $\hat{\bq}_{ij}\hat{\bq}_{ij}^{T} + \hat{\bq}_{ij}^{\perp}(\hat{\bq}_{ij}^{\perp})^{T} = \mathbf{I}_2$, 
the edge matrix simplifies significantly to:
\begin{equation}\label{eq:Aij_gap_split}
\mathbf{A}_{ij}
=
\delta_{ij}\,\hat{\bq}_{ij}\hat{\bq}_{ij}^{T}
+
\alpha_{ij}^{\perp} \mathbf{I}_2.
\end{equation}

Since the restricted Hessian $H_{\mathcal{C}}$ is formed by linear combinations of the $\mathbf{A}_{ij}$ blocks, this local decomposition naturally induces a global splitting:
\begin{equation}\label{eq:HC_split_final}
H_{\mathcal C}
=
L^\Delta + \widetilde L,
\end{equation}
where $L^\Delta$ will be called the {\bf gap Laplacian}, since it is built from
the radial gaps $\delta_{ij}$, and $\widetilde L$ the {\bf transverse Laplacian},
since it is built from the transverse eigenvalues $\alpha_{ij}^{\perp}$.

More precisely, $\widetilde L$ is the reduced matrix-weighted Laplacian with edge
weights $\alpha_{ij}^{\perp}\mathbf{I}_2$. Because each weight is a scalar multiple
of the identity, the two spatial coordinate directions decouple completely, and
$\widetilde L$ factors as the Kronecker product
\begin{equation}\label{eq:Ltilde_tensor_final}
\widetilde L=\ell\otimes \mathbf{I}_2,
\end{equation}
where $\ell \in \mathbb{R}^{(N-1)\times (N-1)}$ is the {\bf reduced transverse
Laplacian}.

Recall that $\ell_{\mathrm{full}}\in\mathbb R^{N\times N}$ is the full
scalar transverse Laplacian, with edge weights $\alpha_{ij}^{\perp}$,
and that its associated block Laplacian is
\[
L_{\mathrm{full}}
=
\ell_{\mathrm{full}}\otimes\mathbf I_2.
\]
Grounding body $1$ amounts to deleting the row and column of
$\ell_{\mathrm{full}}$ corresponding to body $1$, producing exactly the
matrix $\ell$ introduced --- without being  named --- in 
Remark \ref{rmk:transl-degeneracy-and-grounding}. We now refer to it as the
\textbf{reduced transverse Laplacian}:
\[
\ell
=
(\ell_{\mathrm{full}})_{\{2,\ldots,N\},\{2,\ldots,N\}}.
\]

Accordingly, $\widetilde L$ is the principal block submatrix of
$L_{\mathrm{full}}$ obtained by deleting the $2\times2$ block row and
block column corresponding to body $1$:
\[
\widetilde L
=
(L_{\mathrm{full}})_{\{2,\ldots,N\},\{2,\ldots,N\}}
=
\ell\otimes\mathbf I_2.
\]
The entries of $ \ell $ indexed by bodies $a,b\in\{2,\dots,N\}$, are given by

\begin{equation}\label{eq:ell_entrywise}
\ell_{a-1,a-1} = \sum_{k\neq a}\alpha_{ak}^\perp,
\qquad
\ell_{a-1,b-1} = -\alpha_{ab}^\perp
\quad (a\neq b).
\end{equation}

This isolates the potentially indefinite part of the restricted Hessian into the much simpler reduced transverse  Laplacian $\ell$.

The gap Laplacian $L^\Delta$ is the reduced Laplacian of the complete
graph $K_N$ with matrix weights
\[
    \mathbf{D}_{ij} = \delta_{ij}\,\hat{\bq}_{ij}\hat{\bq}_{ij}^{T},
\]
built from the radial gaps $\delta_{ij}$. Fixing body 1 as the grounded
vertex, $L^\Delta$ is the $(N-1)\times(N-1)$ block matrix (with blocks
of size $2\times 2$, indexed by bodies $a,b\in\{2,\dots,N\}$) whose
entries are given by
\begin{equation}\label{eq:LDelta_entrywise}
(L^\Delta)_{a-1,a-1} = \sum_{k\neq a}\mathbf{D}_{ak},
\qquad
(L^\Delta)_{a-1,b-1} = -\mathbf{D}_{ab}
\quad (a\neq b).
\end{equation}

In contrast to the potentially indefinite $\widetilde{L}$, the geometric matrix $L^\Delta$ is inherently positive semidefinite, as established in the following proposition.

\begin{proposition}\label{prop:LDelta_psd}
The gap Laplacian $L^\Delta$ is positive semidefinite. More precisely, for any
$\mathbf{v} = (\mathbf{v}_2, \dots, \mathbf{v}_N) \in \mathbb{R}^{2(N-1)}$, its
quadratic form decomposes as
\begin{equation}\label{eq:LDelta_qf_final}
\mathbf{v}^T L^\Delta \mathbf{v}
=
\sum_{a=2}^{N} \delta_{1a}\bigl(\mathbf{v}_a\cdot\hat{\bq}_{1a}\bigr)^2
+
\sum_{2\le a<b\le N} \delta_{ab}\bigl((\mathbf{v}_a-\mathbf{v}_b)\cdot\hat{\bq}_{ab}\bigr)^2
\ge 0.
\end{equation}
In particular, $L^\Delta \succeq 0$.
\end{proposition}

\begin{proof}
The matrix $L^\Delta$ is the reduced matrix-weighted Laplacian, at body
$1$, of the complete graph $K_N$ with edge weights
\[
\mathbf D_{ij}=\delta_{ij}\,\hat{\bq}_{ij}\hat{\bq}_{ij}^T.
\]
Since $\delta_{ij}>0$ and $\hat{\bq}_{ij}\hat{\bq}_{ij}^T$ is a rank-one
orthogonal projection, each edge weight $\mathbf D_{ij}$ is positive
semidefinite, so by Lemma~\ref{lem:matrix_weighted_laplacian_psd} the
full (ungrounded) Laplacian $L^\Delta_{\mathrm{full}}$ satisfies, for
every $\mathbf x=(\mathbf x_1,\ldots,\mathbf x_N)$,
\[
\mathbf x^T L^\Delta_{\mathrm{full}} \mathbf x
=
\sum_{1\le i<j\le N}
\delta_{ij}\bigl((\mathbf x_i-\mathbf x_j)\cdot\hat{\bq}_{ij}\bigr)^2
\ge 0.
\]

Given $\mathbf v=(\mathbf v_2,\ldots,\mathbf v_N)\in\mathbb R^{2(N-1)}$, set
$\mathbf x_1:=0$ and $\mathbf x_a:=\mathbf v_a$ for $a\ge2$. Writing
$L^\Delta_{\mathrm{full}}$ in block form with the body-$1$ block
separated out,
\[
L^\Delta_{\mathrm{full}}
=
\begin{pmatrix}
(L^\Delta_{\mathrm{full}})_{11} & \ast\\
\ast & L^\Delta
\end{pmatrix},
\]
every term in $\mathbf x^TL^\Delta_{\mathrm{full}}\mathbf x$ involving the
first block is multiplied by $\mathbf x_1=0$ and vanishes, leaving
\[
\mathbf x^T L^\Delta_{\mathrm{full}} \mathbf x
=
\mathbf v^T L^\Delta \mathbf v.
\]

Splitting the edge sum above according to whether an edge touches body
$1$,
\[
\mathbf x^T L^\Delta_{\mathrm{full}} \mathbf x
=
\sum_{a=2}^{N}\delta_{1a}\bigl((\mathbf x_1-\mathbf x_a)\cdot\hat{\bq}_{1a}\bigr)^2
+
\sum_{2\le a<b\le N}\delta_{ab}\bigl((\mathbf x_a-\mathbf x_b)\cdot\hat{\bq}_{ab}\bigr)^2.
\]

Since \(\mathbf x_1=0\) and \(\mathbf x_a=\mathbf v_a\) for \(a\ge2\),
the edges incident to body \(1\) give the first sum in
\eqref{eq:LDelta_qf_final}, while the remaining edges give the second.
Thus \eqref{eq:LDelta_qf_final} follows. In particular,
\(L^\Delta\succeq0\).

\end{proof}
\begin{remark}[Topological decoupling and the transverse network]
The factorization \(\widetilde L=\ell\otimes\mathbf I_2\) separates the
topology of the interaction network from the spatial geometry of the
configuration: the reduced transverse Laplacian $\ell$ encodes the graph
structure through the signed scalar weights $\alpha_{ij}^{\perp}$, while
the factor $\mathbf I_2$ reflects the isotropy of the transverse
contribution in the plane. Thus $\ell$ is exactly the stiffness matrix
of an idealized one-dimensional network in which the $\alpha_{ij}^\perp$
act as signed spring constants. By contrast, $L^\Delta$ admits no such
factorization: its rank-one weights
$\delta_{ij}\hat{\bq}_{ij}\hat{\bq}_{ij}^{T}$ retain the directional
geometry of the pair vectors, so its topology and geometry cannot be
disentangled in the same way.
\end{remark}

\section{Specialization to the Four Body Problem}
\label{sec:four-body-problem}
We now specialize the preceding constructions to the planar four-body problem ($N=4$). In this setting, the ambient pair space consists of six mutual distance vectors. We partition these into three independent pairs connected to the first body ($\mathbf{q}_{12}, \mathbf{q}_{13}, \mathbf{q}_{14}$) and three dependent pairs ($\mathbf{q}_{23}, \mathbf{q}_{24}, \mathbf{q}_{34}$). The linear realizability subspace $\mathcal{C}$ is defined by three fundamental triangle relations:
\begin{align}
    \mathbf{q}_{12} + \mathbf{q}_{23} - \mathbf{q}_{13} &= 0, \nonumber\\
    \mathbf{q}_{12} + \mathbf{q}_{24} - \mathbf{q}_{14} &= 0, \label{eq:triangle_4body} \\
    \mathbf{q}_{13} + \mathbf{q}_{34} - \mathbf{q}_{14} &= 0. \nonumber
\end{align}

Assigning vector Lagrange multipliers $\boldsymbol \phi_{23}, \boldsymbol \phi_{24}, \boldsymbol \phi_{34} \in \mathbb{R}^2$ to these constraints, the augmented pair-space functional takes the explicit form:
\begin{equation}
\label{eq:F_4body}
\begin{aligned}
    \mathcal{F}(\mathbf{q}, \boldsymbol \phi) &= U_\pi(\mathbf{q}) + \frac{\omega^2}{2} I_\pi(\mathbf{q}) + \boldsymbol \phi_{23} \cdot (\mathbf{q}_{12} + \mathbf{q}_{23} - \mathbf{q}_{13}) \\
    &\quad + \boldsymbol \phi_{24} \cdot (\mathbf{q}_{12} + \mathbf{q}_{24} - \mathbf{q}_{14}) + \boldsymbol \phi_{34} \cdot (\mathbf{q}_{13} + \mathbf{q}_{34} - \mathbf{q}_{14}).
\end{aligned}
\end{equation}

Taking gradients with respect to the pair vectors yields the exact constraint forces $\mathbf{K}_{ij}$. 
For the dependent pairs, the general formula \eqref{eq:K_dependent} reduces to
\[
\mathbf{K}_{23}={\boldsymbol \phi}_{23}, \qquad
\mathbf{K}_{24}={\boldsymbol \phi}_{24}, \qquad
\mathbf{K}_{34}={\boldsymbol \phi}_{34}.
\]
For the independent pairs, the general formula \eqref{eq:K_independent} reduces to
\begin{align*} 
    \mathbf{K}_{12} = {\boldsymbol \phi}_{23} + {\boldsymbol \phi}_{24},  \qquad
    \mathbf{K}_{13} = -{\boldsymbol \phi}_{23} +{\boldsymbol \phi}_{34}, \qquad    \mathbf{K}_{14} = -{\boldsymbol \phi}_{24} - {\boldsymbol \phi}_{34}. 
\end{align*}

We next specialize the Hessian decomposition. In the four-body case, the
restricted Hessian \(H_{\mathcal C}\) is a \(6\times 6\) matrix, naturally
organized as a \(3\times 3\) array of \(2\times 2\) blocks corresponding to the
independent pair vectors
\(\mathbf{q}_{12}, \mathbf{q}_{13}, \mathbf{q}_{14}\). The general splitting
\[
H_{\mathcal C}=L^\Delta+\ell\otimes\mathbf{I}_2
\]
therefore becomes particularly concrete. Here \(L^\Delta\) is the  gap
Laplacian, encoding the radial contributions, while \(\ell\) is the reduced
transverse Laplacian, now a \(3\times 3\) signed scalar matrix. Its entries are

\begin{equation}
\label{eq:ell_4body}
    \ell = 
    \begin{pmatrix}
        \alpha_{12}^\perp + \alpha_{23}^\perp + \alpha_{24}^\perp & -\alpha_{23}^\perp & -\alpha_{24}^\perp \\[6pt]
        -\alpha_{23}^\perp & \alpha_{13}^\perp + \alpha_{23}^\perp + \alpha_{34}^\perp & -\alpha_{34}^\perp \\[6pt]
        -\alpha_{24}^\perp & -\alpha_{34}^\perp & \alpha_{14}^\perp + \alpha_{24}^\perp + \alpha_{34}^\perp
    \end{pmatrix}.
\end{equation}
Notice that $\ell$ is precisely the full $4 \times 4$ transverse Laplacian matrix $\ell_{\mathrm{full}}$ with the first row and column deleted. 

The four-body form of the reduced transverse Laplacian $\ell$ is especially
useful, because in this setting the transverse part becomes highly degenerate.
This degeneracy will allow us to derive simple eigenvalue criteria for the
positive definiteness of $H_{\mathcal{C}}$. We begin by recording the basic
algebraic fact underlying this simplification:

\begin{proposition}\label{prop:rank1}
For any planar, non-collinear four-body central configuration, the reduced transverse Laplacian matrix $\ell$ has rank $1$.
\end{proposition}

\begin{proof}
At a planar four-body central configuration, the coefficients
\(\alpha_{ij}^{\perp}\) satisfy Equation \eqref{eq:alpha_perp_balance}, that is, 
\(\sum_{j\ne i}\alpha_{ij}^{\perp}\,\bq_{ij}=0,\) for \(i=1,2,3,4.\)
We now ground body \(1\). Using the triangle relations
\[
\bq_{23}=\bq_{13}-\bq_{12},\qquad
\bq_{24}=\bq_{14}-\bq_{12},\qquad
\bq_{34}=\bq_{14}-\bq_{13},
\]
the equations corresponding to \(i=2,3,4\) become
\begin{align*}
-(\alpha_{12}^{\perp}+\alpha_{23}^{\perp}+\alpha_{24}^{\perp})\,\bq_{12}
+\alpha_{23}^{\perp}\,\bq_{13}
+\alpha_{24}^{\perp}\,\bq_{14}
&=0,\\
\alpha_{23}^{\perp}\,\bq_{12}
-(\alpha_{13}^{\perp}+\alpha_{23}^{\perp}+\alpha_{34}^{\perp})\,\bq_{13}
+\alpha_{34}^{\perp}\,\bq_{14}
&=0,\\
\alpha_{24}^{\perp}\,\bq_{12}
+\alpha_{34}^{\perp}\,\bq_{13}
-(\alpha_{14}^{\perp}+\alpha_{24}^{\perp}+\alpha_{34}^{\perp})\,\bq_{14}
&=0.
\end{align*}
Multiplying by $ -1 $ and recalling the explicit entries of the matrix $ \ell $, this is exactly equivalent to the block-matrix equation:
\[
	(\ell\otimes {\mathbf  I}_2)
\begin{pmatrix}
\bq_{12}\\
\bq_{13}\\
\bq_{14}
\end{pmatrix}
=0.
\]

By resolving the vectors into their Cartesian components, 
\[
\bq_{12}=(x_{12},y_{12}),\qquad
\bq_{13}=(x_{13},y_{13}),\qquad
\bq_{14}=(x_{14},y_{14}),
\]
and  defining the three-dimensional coordinate vectors
\[
x=
\begin{pmatrix}
x_{12}\\ x_{13}\\ x_{14}
\end{pmatrix},
\qquad
y=
\begin{pmatrix}
y_{12}\\ y_{13}\\ y_{14}
\end{pmatrix},
\]
the block-matrix equation decouples into:
\[
\ell x=0,
\qquad
\ell y=0.
\]
Since the configuration is non-collinear, the vectors \(x\) and \(y\) are
linearly independent (if they were dependent, all bodies would lie on a single line passing through body 1). Hence, the kernel of $ \ell $ contains at least two linearly independent vectors:
\[
\dim\ker(\ell)\ge 2,
\]
and therefore
\[
\operatorname{rank}(\ell)\le 1.
\]

It remains to show that \(\ell\neq 0\). If \(\ell=0\), then \(\alpha_{ij}^{\perp}=0\) for every pair \(i<j\). By the definition of
\(\alpha_{ij}^{\perp}\), this means
\[
\omega^2=\frac{M}{q_{ij}^{3}}
\qquad\text{for all } i<j,
\]
so all six mutual distances are equal. However, four distinct points in
\(\mathbb R^2\) cannot be pairwise equidistant. This contradiction ensures that
\(\ell\neq 0\). Consequently,
\[
\operatorname{rank}(\ell)=1.
\qedhere
\]
\end{proof}
Since $\ell$ has rank $1$ by Proposition~\ref{prop:rank1}, its spectrum consists
of a single non-zero eigenvalue (equal to its trace) together with a zero
eigenvalue of multiplicity $2$; the following corollary records the resulting
spectrum of $\widetilde L=\ell\otimes\mathbf I_2$.
\begin{corollary}\label{cor:tildeL-eigenvalues}
For any planar non-collinear four-body central configuration, the matrix
\(\widetilde{L}=\ell\otimes \mathbf{I}_2\) has exactly two non-zero eigenvalues, both
equal to
\[
\tau:=\operatorname{tr}(\ell)
=
\alpha_{12}^{\perp}+\alpha_{13}^{\perp}+\alpha_{14}^{\perp}
+2(\alpha_{23}^{\perp}+\alpha_{24}^{\perp}+\alpha_{34}^{\perp}).
\]
In particular,
\[
\operatorname{spec}(\widetilde{L})
=
\{\tau,\tau,0,0,0,0\}.
\]
\end{corollary}

\begin{proof}
Since \(\ell\) has rank \(1\) by Proposition~\ref{prop:rank1}, it has
exactly one non-zero eigenvalue \(\tau\), and \(\tau=\operatorname{tr}(\ell)\)
because the trace equals the sum of the eigenvalues. Reading off the diagonal
entries of the grounded scalar Laplacian gives
\begin{equation}\label{eq:trace-ell}
\begin{aligned}
\operatorname{tr}(\ell)
&=
(\alpha_{12}^{\perp}+\alpha_{23}^{\perp}+\alpha_{24}^{\perp})
+
(\alpha_{13}^{\perp}+\alpha_{23}^{\perp}+\alpha_{34}^{\perp})
+
(\alpha_{14}^{\perp}+\alpha_{24}^{\perp}+\alpha_{34}^{\perp})\\
&=
\alpha_{12}^{\perp}+\alpha_{13}^{\perp}+\alpha_{14}^{\perp}
+2(\alpha_{23}^{\perp}+\alpha_{24}^{\perp}+\alpha_{34}^{\perp}).
\end{aligned}
\end{equation}
Finally, since \(\widetilde{L}=\ell\otimes \mathbf{I}_2\), the spectrum of
\(\widetilde{L}\) consists of each eigenvalue of \(\ell\) repeated with
multiplicity \(2\). Therefore
\[
\operatorname{spec}(\widetilde{L})=\{\tau,\tau,0,0,0,0\}.
\qedhere
\]
\end{proof}
\begin{corollary}\label{cor:ell-factorization}
For any planar non-collinear four-body central configuration with
\(\alpha_{23}^\perp,\alpha_{24}^\perp,\alpha_{34}^\perp \neq 0\), the reduced
transverse Laplacian \(\ell\) admits the rank-one factorization
\begin{equation}\label{eq:ell-factorization}
\ell=\sigma\,\mathbf w\mathbf w^T,
\qquad
\sigma=-\alpha_{23}^\perp\alpha_{24}^\perp\alpha_{34}^\perp,
\qquad
\mathbf w=
\begin{pmatrix}
1/\alpha_{34}^\perp\\[4pt]
1/\alpha_{24}^\perp\\[4pt]
1/\alpha_{23}^\perp
\end{pmatrix}.
\end{equation}
In particular, \(\mathbf w\) is an eigenvector of \(\ell\) with eigenvalue
\(\sigma\|\mathbf w\|^2=\operatorname{tr}(\ell)\).
Consequently,
\begin{equation}\label{eq:tildeL-factorization}
\widetilde L=\ell\otimes \mathbf I_2
=\sigma\,\mathbf W\mathbf W^T,
\qquad
\mathbf W:=\mathbf w\otimes \mathbf I_2,
\end{equation}
and the independent pair vectors satisfy the linear relation
\begin{equation}\label{eq:pair-dependency}
\frac{1}{\alpha_{34}^\perp}\mathbf q_{12}
+\frac{1}{\alpha_{24}^\perp}\mathbf q_{13}
+\frac{1}{\alpha_{23}^\perp}\mathbf q_{14}
=0.
\end{equation}
\end{corollary}

\begin{proof}
By Proposition~\ref{prop:rank1}, the matrix \(\ell\) has rank \(1\). Since \(\ell\)
is symmetric, it must therefore be expressible as a scalar multiple of an outer
product \(\mathbf w\mathbf w^T\). We now determine \(\mathbf w\) and the scalar
factor explicitly.

From \eqref{eq:ell_4body}, the off-diagonal entries of \(\ell\) are
\[
\ell_{12}=-\alpha_{23}^\perp,
\qquad
\ell_{13}=-\alpha_{24}^\perp,
\qquad
\ell_{23}=-\alpha_{34}^\perp.
\]
If we set
\[
\mathbf w=
\begin{pmatrix}
1/\alpha_{34}^\perp\\[4pt]
1/\alpha_{24}^\perp\\[4pt]
1/\alpha_{23}^\perp
\end{pmatrix},
\qquad
\sigma=-\alpha_{23}^\perp\alpha_{24}^\perp\alpha_{34}^\perp,
\]
then the off-diagonal entries of \(\sigma\,\mathbf w\mathbf w^T\) are
\[
\sigma w_1w_2=-\alpha_{23}^\perp,
\qquad
\sigma w_1w_3=-\alpha_{24}^\perp,
\qquad
\sigma w_2w_3=-\alpha_{34}^\perp.
\]
Thus \(\sigma\,\mathbf w\mathbf w^T\) has exactly the same off-diagonal entries as
\(\ell\).

To conclude that the two matrices agree, it remains only to check the
diagonal entries.
By definition \eqref{eq:ell_4body}, $\ell_{11}=\alpha_{12}^{\perp}+\alpha_{23}^{\perp}+\alpha_{24}^{\perp}$.
On the other hand, since $\operatorname{rank}(\ell)=1$ by Proposition~\ref{prop:rank1},
every $2\times2$ minor of $\ell$ vanishes; in particular
$\ell_{11}\ell_{23}=\ell_{12}\ell_{13}$, so
\[
\ell_{11}=\frac{\ell_{12}\ell_{13}}{\ell_{23}}
=\frac{(-\alpha_{23}^{\perp})(-\alpha_{24}^{\perp})}{-\alpha_{34}^{\perp}}
=-\frac{\alpha_{23}^{\perp}\alpha_{24}^{\perp}}{\alpha_{34}^{\perp}}.
\]
Comparing the two expressions for $\ell_{11}$ gives the identity
\[
\alpha_{12}^{\perp}+\alpha_{23}^{\perp}+\alpha_{24}^{\perp}
=-\frac{\alpha_{23}^{\perp}\alpha_{24}^{\perp}}{\alpha_{34}^{\perp}},
\]
which is exactly the diagonal entry needed for
\eqref{eq:ell-factorization}. The cyclic minors give the analogous
identities for $\ell_{22}$ and $\ell_{33}$.

Multiplying $\ell$ by $\mathbf w$ gives
\[
\ell \mathbf{w} = (\sigma \mathbf{w} \mathbf{w} ^T) \mathbf{w}
= \sigma \mathbf{w} (\mathbf{w} ^T \mathbf{w})
= (\sigma \| \mathbf{w} \| ^2) \mathbf{w}.
\]
so $\mathbf w$ is an eigenvector of $\ell$ with eigenvalue
$\lambda=\sigma\|\mathbf w\|^2$. This eigenvalue is nonzero: $\sigma\neq0$ by
hypothesis, and $\mathbf w\neq0$ since each of its entries is the reciprocal
of a nonzero quantity. Since $\ell$ is symmetric of rank $1$ by
Proposition~\ref{prop:rank1} $\lambda$ is its unique nonzero eigenvalue.


Next, using the mixed-product property of the Kronecker product,
\((A\otimes B)(C\otimes D)=(AC)\otimes(BD)\), we obtain
\[
\widetilde L
=
(\sigma\,\mathbf w\mathbf w^T)\otimes \mathbf I_2
=
\sigma\,(\mathbf w\otimes \mathbf I_2)(\mathbf w^T\otimes \mathbf I_2)
=
\sigma\,\mathbf W\mathbf W^T,
\]
which proves \eqref{eq:tildeL-factorization}.

Finally, let
\[
x=(x_{12},x_{13},x_{14})^T,
\qquad
y=(y_{12},y_{13},y_{14})^T.
\]
By Proposition~\ref{prop:rank1}, we have \(\ell x=0\) and \(\ell y=0\).
Substituting \(\ell=\sigma\,\mathbf w\mathbf w^T\) gives
\[
\sigma\,\mathbf w(\mathbf w^Tx)=0,
\qquad
\sigma\,\mathbf w(\mathbf w^Ty)=0.
\]
Since \(\sigma\neq0\) and \(\mathbf w\neq0\), it follows that
\[
\mathbf w^Tx=0,
\qquad
\mathbf w^Ty=0.
\]
Writing these two scalar identities in vector form yields
\[
\frac{1}{\alpha_{34}^\perp}\mathbf q_{12}
+\frac{1}{\alpha_{24}^\perp}\mathbf q_{13}
+\frac{1}{\alpha_{23}^\perp}\mathbf q_{14}
=0,
\]
which is exactly \eqref{eq:pair-dependency}.
\end{proof}

\begin{remark}[Fictitious center of mass]
Equation~\eqref{eq:pair-dependency} admits a geometric interpretation in 
the spirit of classical celestial mechanics. By expanding the pair vectors into 
absolute positions, $\bq_{1k} = \br_1 - \br_k$, the linear dependence rearranges 
into a generalized barycentric coordinate formula:
\[
\br_1 = \frac{ \tilde{m}_2 \br_2 + \tilde{m}_3 \br_3 + \tilde{m}_4 \br_4 }{ \tilde{m}_2 + \tilde{m}_3 + \tilde{m}_4 },
\]
where we have defined the fictitious masses 
$\tilde{m}_2 = 1/\alpha_{34}^\perp$, 
$\tilde{m}_3 = 1/\alpha_{24}^\perp$, and 
$\tilde{m}_4 = 1/\alpha_{23}^\perp$. 
This reveals that body $1$ sits exactly at the center of mass of bodies
$2$, $3$, and $4$, provided each is assigned a fictitious mass equal to the reciprocal
$1/\alpha_{ij}^\perp$ of the transverse eigenvalue of the edge opposite to it.
 An
identical argument, grounding at any other body in place of body $1$, shows
that in fact every body sits at the fictitious barycenter of the other
three. Note that these fictitious masses need not be positive --- unlike a physical center of
mass, this is a signed affine combination --- so the barycentric picture
is formal rather than literal. 
\end{remark}
\begin{remark}[Connection with Dziobek areas]
Since the three pair vectors $\mathbf q_{12},\mathbf q_{13},\mathbf q_{14}$
lie in $\mathbb R^2$, they are automatically linearly dependent. To make the
coefficients of this dependence explicit, for $i=1,\dots,4$ let
\[
\Delta_i = (-1)^{i}\,\frac12\det
\begin{pmatrix}
1 & x_j & y_j \\
1 & x_k & y_k \\
1 & x_l & y_l
\end{pmatrix},
\qquad \{j,k,l\}=\{1,2,3,4\}\setminus\{i\}\ \text{in increasing order},
\]
so that $\Delta_i$ is, up to the alternating sign $(-1)^i$, the signed area
of the triangle formed by the three bodies other than body $i$. 
With this convention, the linear dependence of the pair
vectors takes the form
\begin{equation}\label{eq:coplanarity}
    \Delta_2\mathbf q_{12}
    +\Delta_3\mathbf q_{13}
    +\Delta_4\mathbf q_{14}=0.
\end{equation}
It is a classical fact, due to Dziobek~\cite{dziobek1900ueber} (see
also~\cite{moeckel2001generic} for a modern treatment), that
\[
    \alpha_{ij}^{\perp}=\kappa\,\Delta_i\Delta_j,
\]
where $\kappa\neq 0$ is a global constant. 
Substituting this relation into Eq.~\eqref{eq:pair-dependency} gives
\[
    \frac{\mathbf q_{12}}{\kappa\Delta_3\Delta_4}
    +\frac{\mathbf q_{13}}{\kappa\Delta_2\Delta_4}
    +\frac{\mathbf q_{14}}{\kappa\Delta_2\Delta_3}=0,
\]
and multiplying through by $\kappa\Delta_2\Delta_3\Delta_4$ recovers
Eq.~\eqref{eq:coplanarity}. Thus, Eq.~\eqref{eq:pair-dependency} is
precisely the classical Dziobek coplanarity relation among the grounded
pair vectors, expressed here in terms of the $\alpha_{ij}^\perp$ instead of the $ \Delta _i$'s.
\end{remark}

In the classical planar \(N\)-body problem, Palmore proved that the Morse
index of a central configuration is at most \(N-2\); see
\cite{moeckel2014lectures} for a modern account.  In particular, the index
is at most two for a planar four-body central configuration.  We recover this bound using the following lemma and the decomposition
\[
H_{\mathcal C}=\widetilde L+L^\Delta.
\]

\begin{lemma}\label{lem:negative-index-monotone}
For a real symmetric matrix \(A\), let \(n_-(A)\) denote the number of
negative eigenvalues of \(A\), counted with multiplicity. If \(A,B\) are
symmetric \(n\times n\) matrices and \(B\succeq 0\), then
\[
n_-(A+B)\leq n_-(A).
\]
\end{lemma}
\begin{proof}
Recall that for any real symmetric matrix, the number of strictly negative eigenvalues is equal to the maximum dimension of a subspace on which its quadratic form is negative definite.

Let \(W \subseteq \mathbb{R}^n\) be a subspace of dimension \(n_-(A+B)\) on which \(A+B\) is negative definite. For every non-zero \(x \in W\), we have
\[
x^TAx = x^T(A+B)x - x^TBx.
\]
Because \(B\) is positive semidefinite (\(B \succeq 0\)), it follows that \(x^TBx \ge 0\). Therefore,
\[
x^TAx \le x^T(A+B)x < 0.
\]
This demonstrates that \(A\) is strictly negative definite on the entire subspace \(W\). Since \(n_-(A)\) represents the maximum dimension of any such subspace for \(A\), it must be that
\[
n_-(A) \ge \dim(W) = n_-(A+B),
\]
which completes the proof.
\end{proof}

\begin{proposition}\label{prop:negative-index-total}
For any planar non-collinear four-body central configuration the restricted
Hessian
\[
H_{\mathcal C}=\widetilde L+L^\Delta
\]
has at most two negative eigenvalues.
\end{proposition}
\begin{proof}
By Corollary~\ref{cor:tildeL-eigenvalues},
$\operatorname{spec}(\widetilde L)=\{\tau,\tau,0,0,0,0\}$, so
$n_-(\widetilde L)\le 2$. 
Since $L^\Delta\succeq 0$ (Proposition~\ref{prop:LDelta_psd}),
Lemma~\ref{lem:negative-index-monotone} gives
\[
n_-(H_{\mathcal C}) = n_-\bigl(\widetilde L+L^\Delta\bigr) \le n_-(\widetilde L) \le 2.
\]
Hence $H_{\mathcal C}$ has at most two negative eigenvalues.
\end{proof}
To relate the index of $H_{\mathcal C}$ to the Morse index on the
normalized configuration space, it remains to examine the direction
removed by fixing the scale. The next lemma shows that this dilation
direction is strictly positive.
\begin{lemma}[Positivity of the dilation mode]
\label{lem:dilation-positive}
At any planar central configuration, the quadratic form of $H_{\mathcal C}$
is strictly positive along the dilation direction $\mathbf v=\mathbf q$.
\end{lemma}
\begin{proof}
Recall that if $g$ is homogeneous of degree $k$, then
$\mathbf q^T\nabla^2g(\mathbf q)\,\mathbf q=k(k-1)g(\mathbf q)$, obtained
by differentiating $g(\lambda\mathbf q)=\lambda^kg(\mathbf q)$ twice in
$\lambda$ at $\lambda=1$. Since $U_\pi$ and $I_\pi$ are homogeneous of
degrees $-1$ and $2$, respectively, applying the identity to each yields:
\[
    \mathbf q^T\nabla^2U_\pi(\mathbf q)\,\mathbf q=2\,U_\pi(\mathbf q),
    \qquad
    \mathbf q^T\nabla^2I_\pi(\mathbf q)\,\mathbf q=2\,I_\pi(\mathbf q).
\]

The matrix $H_{\mathcal C}$ is defined as the Hessian $H_{\mathcal C} = \nabla^2 \left( U_\pi + \frac{\omega^2}{2}I_\pi \right)$. Evaluating the quadratic form along the dilation direction $\mathbf{v} = \mathbf{q}$ and substituting the central configuration relation $\omega^2 = \frac{U_\pi(\mathbf{q})}{I_\pi(\mathbf{q})}$, we obtain:

\[
    \mathbf v^TH_{\mathcal C}\mathbf v
    =
    2\,U_\pi(\mathbf q)+\omega^2 I_\pi(\mathbf q)
    =
    3\,U_\pi(\mathbf q).
\]
Since the potential energy $U_\pi(\mathbf{q})$ is strictly positive for any physical configuration, it follows that $\mathbf{v}^T H_{\mathcal C} \mathbf{v} > 0$. Thus, the quadratic form of $H_{\mathcal C}$ is strictly positive along the dilation direction.
\end{proof}

The matrix $H_{\mathcal{C}}$ is defined on the $6$-dimensional space of
configurations modulo translation alone. Palmore's classical Morse index,
on the other hand, is computed on the space of configurations with a fixed
center of mass and fixed size (moment of inertia $I_{\pi} = 1$), as
in~\cite{moeckel2014lectures}---a space obtained by additionally fixing
the scale. By Lemma~\ref{lem:dilation-positive}, the quadratic form of
$H_{\mathcal{C}}$ is strictly positive along this additional dilation
direction $\mathbf{v} = \mathbf{q}$. Since this direction is strictly
positive, it may be eliminated without changing $n_-$. Consequently,
$n_-(H_{\mathcal{C}})$ equals Palmore's Morse index, so
Proposition~\ref{prop:negative-index-total} recovers his classical bound
exactly.

\subsection{Generalized Eigenvalue Criterion}
\label{subsec:rayleigh_bound}
By Proposition~\ref{prop:LDelta_psd}, $L^\Delta$ is positive semidefinite.
To obtain a well-posed generalized eigenvalue criterion on the rotation-free
subspace, we must determine whether $\ker L^\Delta$ contains any directions
beyond the residual rotational symmetry. We continue in the planar $4$-body
setting, although the argument extends to general $N$ without essential
change.

Fix body $1$ as the reference point, and let
$\mathbf q = (\mathbf q_{12}^T,\mathbf q_{13}^T,\mathbf q_{14}^T)^T
\in \mathbb R^{6}$
collect the three independent grounded pair vectors. For
$\theta \in \mathbb R$, let
\[
R_\theta =
\begin{pmatrix}
\cos\theta & -\sin\theta\\
\sin\theta & \cos\theta
\end{pmatrix}
=
e^{\theta\mathbf J},
\qquad
\mathbf J =
\begin{pmatrix}
0 & -1\\
1 & 0
\end{pmatrix}.
\]
Thus, $\mathbf J$ is the infinitesimal generator of counterclockwise planar
rotations. Let
$\mathbb J := \mathbf I_3\otimes\mathbf J$ and
$\mathbb R_\theta := \mathbf I_3\otimes R_\theta
= e^{\theta\mathbb J}$ denote their block-diagonal lifts to
$\mathbb R^6$. A rotation of the entire configuration through angle
$\theta$ acts on the grounded pair vector by
\[
\mathbf q \longmapsto \mathbb R_\theta\mathbf q.
\]
In particular, for an infinitesimal angle $\varepsilon$,
\[
\mathbb R_\varepsilon\mathbf q
=
\mathbf q+\varepsilon\,\mathbb J\mathbf q+O(\varepsilon^2).
\]
Consequently,
\[
\mathbf v:=\mathbb J\mathbf q
=
(\mathbf J\mathbf q_{12},\mathbf J\mathbf q_{13},
\mathbf J\mathbf q_{14})^T\in\mathbb R^6
\]
is the infinitesimal rotational mode at $\mathbf q$.

For each pair $i,j$, the infinitesimal rotational displacement
$\mathbf J\mathbf q_{ij}$ is orthogonal to $\mathbf q_{ij}$:
\[
\mathbf J\mathbf q_{ij}\cdot\hat{\bq}_{ij}=0.
\]
For the terms in the quadratic-form identity
\eqref{eq:LDelta_qf_final}, observe that if
$\mathbf v=\mathbb J\mathbf q$, then
$\mathbf v_a=\mathbf J\mathbf q_{1a}$ for $a=2,3,4$, while, for
$2\le a<b\le4$,
\[
\mathbf v_a-\mathbf v_b
=
\mathbf J(\mathbf q_{1a}-\mathbf q_{1b})
=
\mathbf J\mathbf q_{ab}.
\]
Thus every radial projection appearing in
\eqref{eq:LDelta_qf_final} vanishes at $\mathbf v$, and hence
\[
\mathbf v^T L^\Delta\mathbf v=0.
\]
Since $L^\Delta\succeq0$, it follows that \(\mathbf v\in\ker L^\Delta.\)
Thus $\operatorname{span}\{\mathbf v\}\subseteq\ker L^\Delta$; the question
is whether this inclusion is strict, that is, whether $L^\Delta$ has any
degeneracy beyond the rotational mode.

We therefore define the rotation-free subspace
\[
\mathcal W
:=
\{\mathbf x\in\mathbb R^6:\mathbf x\cdot\mathbf v=0\}
=
\operatorname{span}\{\mathbf v\}^{\perp}.
\]
The next lemma shows that the rotational mode spans the entire kernel of
$L^\Delta$. It follows that the restriction
$L^\Delta|_{\mathcal W}$ is positive definite. Equivalently, the
regularized matrix
\[
L^\Delta_{\mathrm{reg}}
:=
L^\Delta+\frac{\mathbf v\mathbf v^T}{\|\mathbf v\|^2}
\]
is positive definite on $\mathbb R^6$.

\begin{lemma}[Non-Degeneracy of the Regularized Gap Laplacian]
\label{lem:LDelta_reg_pd}
For any planar non-collinear four-body central configuration with strictly
positive masses $m_i > 0$, the kernel of the gap Laplacian $L^\Delta$ is exactly
$\operatorname{span}\{\mathbf{v}\}$. Consequently, the regularized gap
matrix $L^\Delta_{\mathrm{reg}} = L^\Delta + \frac{\mathbf{v}\mathbf{v}^T}{\|\mathbf{v}\|^2}$
is strictly positive definite.
\end{lemma}

\begin{proof}

Recall from Proposition~\ref{prop:LDelta_psd} that, for $N = 4$, the
quadratic form of $L^\Delta$ evaluated on a state vector
$\mathbf{x} = (\mathbf{x}_2^T, \mathbf{x}_3^T, \mathbf{x}_4^T)^T \in
\mathbb{R}^{6}$ is given by Eq.~\eqref{eq:LDelta_qf_final}:
\[
\mathbf{x}^T L^\Delta \mathbf{x}
= \sum_{a=2}^{4} \delta_{1a}\bigl(\mathbf{x}_a\cdot\hat{\bq}_{1a}\bigr)^2
+ \sum_{2\le a<b\le 4} \delta_{ab}\bigl((\mathbf{x}_a-\mathbf{x}_b)\cdot\hat{\bq}_{ab}\bigr)^2.
\]
Here the first sum ranges over the three edges $\{12, 13, 14\}$ incident to
the grounded body $1$, and the second sum ranges over the three edges
$\{23, 24, 34\}$ among the remaining bodies --- together these six edges are
exactly those of the complete graph $K_4$. Because every mass $m_i$ is
strictly positive and the configuration is collision-free ($q_{ij} > 0$),
every radial gap coefficient $\delta_{ij} = \frac{3M\mu_{ij}}{q_{ij}^3}$ is
strictly positive.

Since $L^\Delta\succeq0$, we have $\mathbf{x}\in\ker L^\Delta$ if and
only if $\mathbf{x}^T L^\Delta \mathbf{x}=0$. By
Eq.~\eqref{eq:LDelta_qf_final}, this quadratic form is a sum of
non-negative squares with strictly positive coefficients, so it
vanishes if and only if every term vanishes. Set
$\mathbf{x}_1:=\mathbf{0}$. The list $(\mathbf x_1,\mathbf x_2,\mathbf
x_3,\mathbf x_4)$ is then just $\mathbf x$ together with the (trivial)
displacement of the grounded body $1$, allowing the six edge
conditions to be written uniformly. Thus
\[
\mathbf x\in\ker L^\Delta
\quad\Longleftrightarrow\quad
\hat{\bq}_{ij}\cdot(\mathbf{x}_i-\mathbf{x}_j)=0
\ \ \text{for all } 1\leq i<j\leq4.
\]
We now show this holds if and only if $\mathbf x=c\,\mathbf v$ for some
$c\in\mathbb R$, which will establish $\ker
L^\Delta=\operatorname{span}\{\mathbf v\}$, with $ \mathbf v = 
\mathbb{J}  \mathbf q$.

Because the configuration is non-collinear, among the three vectors
$\bq_{12},\bq_{13},\bq_{14}$ there are two linearly independent ones.
Relabeling the ungrounded bodies if necessary, suppose that
$\bq_{12}$ and $\bq_{13}$ are linearly independent.

The edge constraints give
\[
\mathbf{x}_2\cdot\bq_{12}=0,
\qquad
\mathbf{x}_3\cdot\bq_{13}=0.
\]
Since, in $\mathbb R^2$, the orthogonal complement of a nonzero vector
$\mathbf u$ is $\operatorname{span}\{\mathbf J\mathbf u\}$, there exist
$a,b\in\mathbb R$ such that
\[
\mathbf{x}_2=a\,\mathbf{J}\bq_{12},
\qquad
\mathbf{x}_3=b\,\mathbf{J}\bq_{13}.
\]
Using $\bq_{23}=\bq_{13}-\bq_{12}$, the constraint on the edge $23$ gives

\begin{align*}
0
&=
(\mathbf{x}_2-\mathbf{x}_3)\cdot\bq_{23}\\
&=
(a\mathbf{J}\bq_{12}-b\mathbf{J}\bq_{13})
\cdot(\bq_{13}-\bq_{12})\\
&=
(a-b)\det(\bq_{12},\bq_{13}),
\end{align*}
where we used
\[
(\mathbf{J}\mathbf{u})\cdot\mathbf{w}
=
\det(\mathbf{u},\mathbf{w}).
\]
Since $\bq_{12}$ and $\bq_{13}$ are linearly independent,
$\det(\bq_{12},\bq_{13})\neq0$, and hence $a=b$.

The edge $14$ constraint similarly implies that
\[
\mathbf{x}_4=d\,\mathbf{J}\bq_{14}
\]
for some $d\in\mathbb{R}$.

The edge $24$ and edge $34$ constraints similarly give
\[
(a-d)\det(\bq_{12},\bq_{14})=0,
\qquad
(b-d)\det(\bq_{13},\bq_{14})=0.
\]
Since $\bq_{12}$ and $\bq_{13}$ are linearly independent, $\bq_{14}$
cannot be parallel to both of them, so at least one of
$\det(\bq_{12},\bq_{14})$, $\det(\bq_{13},\bq_{14})$ is nonzero. Together
with $a=b$, either equation then forces $d=a$.

Therefore $ a = b = d $. Write $ c := a = b = d$ for this common value. It follows that 
\[
\mathbf{x}_a=c\,\mathbf{J}\bq_{1a},
\qquad a=2,3,4,
\]
so $\mathbf{x}=c\,\mathbb{J}\mathbf{q}=c\,\mathbf{v}$. Hence
\(
\ker L^\Delta=\operatorname{span}\{\mathbf{v}\}.
\)

\vskip 1cm 

Now, let $\widehat{\mathbf v}:=\mathbf v/\|\mathbf v\|$. Every
$\mathbf{x}\in\mathbb R^6$ has an orthogonal decomposition
\[
\mathbf{x}=\mathbf{x}_{\perp}+t\,\widehat{\mathbf v},
\qquad
\mathbf{x}_{\perp}\perp\mathbf v,\quad t\in\mathbb R.
\]
Since $L^\Delta_{\mathrm{reg}}=L^\Delta+\widehat{\mathbf v}\widehat{\mathbf v}^T$
and $L^\Delta\widehat{\mathbf v}=0$, we obtain
\[
\mathbf{x}^T L^\Delta_{\mathrm{reg}}\mathbf{x}
=
\mathbf{x}_{\perp}^T L^\Delta\mathbf{x}_{\perp}+t^2.
\]
Suppose $\mathbf{x}\neq\mathbf0$, then either $t\neq0$, or
$t=0$.  If $t\neq0$, then
\[
\mathbf{x}^T L^\Delta_{\mathrm{reg}}\mathbf{x}
=
\mathbf{x}_{\perp}^T L^\Delta\mathbf{x}_{\perp}+t^2
\geq t^2>0.
\]
If $t=0$, then $\mathbf{x}=\mathbf{x}_{\perp}\neq\mathbf0$. Since
$\mathbf{x}_{\perp}\perp\mathbf v$ and
$\ker L^\Delta=\operatorname{span}\{\mathbf v\}$, we have
$\mathbf{x}_{\perp}\notin\ker L^\Delta$. Hence
\[
\mathbf{x}_{\perp}^T L^\Delta\mathbf{x}_{\perp}>0,
\]
and therefore
\(\mathbf{x}^T L^\Delta_{\mathrm{reg}}\mathbf{x}>0.\)
Thus $L^\Delta_{\mathrm{reg}}\succ0$.

\end{proof}

By Lemma~\ref{lem:LDelta_reg_pd}, the restriction of $L^\Delta$ to
$\mathcal W$ is positive definite. Equivalently, its matrix
representation $L^\Delta_{\mathrm{reg}}$ is invertible on $\mathbb R^6$.
By Corollary~\ref{cor:ell-factorization},
\[
\widetilde L=\sigma \mathbf{WW}^T,
\qquad
\mathbf W=\mathbf w\otimes\mathbf I_2.
\]
Since $\mathbf w\neq\mathbf0$, the matrix $\mathbf W$ has full column rank $2$;
hence $\widetilde L$ has rank $2$.

Since $L^\Delta|_{\mathcal W}$ is positive definite and $\widetilde L$ has
rank $2$, adding $\widetilde L$ can change the sign of at most two
eigenvalues of $H_{\mathcal C}|_{\mathcal W}=L^\Delta|_{\mathcal
W}+\widetilde L|_{\mathcal W}$: a rank-$r$ perturbation of any symmetric
matrix alters its number of negative eigenvalues by at most $r$. The
definiteness of $H_{\mathcal C}|_{\mathcal W}$ is therefore governed
entirely by how $\widetilde L$ interacts with $L^\Delta$ on the
$2$-dimensional subspace $\operatorname{range}(\mathbf W)$, and we now
make this precise by reducing the question to the spectrum of an
explicit $2\times2$ matrix. Before carrying out this reduction, we
verify that the range of $\mathbf W$ lies entirely in the rotation-free
space $\mathcal W$.

\begin{lemma}[The range of $\mathbf W$ is rotation-free]
\label{lem:F_range_in_W}
Let $\mathbf v=\mathbb J\mathbf q$ be the infinitesimal rotational mode and
let $\mathbf W=\mathbf w\otimes\mathbf I_2$, where $\mathbf w$ is as in
Corollary~\ref{cor:ell-factorization}. Then
\[
\mathbf v^T \mathbf W=\mathbf0.
\]
Consequently,
\[
\operatorname{range}(\mathbf W)\subseteq\mathbf v^\perp=\mathcal W.
\]
\end{lemma}

\begin{proof}
For $\mathbf c\in\mathbb R^2$, the three $2$-vector blocks of
$\mathbf W\mathbf c$ are $w_1\mathbf c$, $w_2\mathbf c$, and $w_3\mathbf c$.
Therefore,
\[
\begin{aligned}
	\mathbf v^T \mathbf W\mathbf c & =(\mathbf J\bq_{12} ^T )(w_{1}\mathbf c) +(\mathbf J\bq_{13})(w_{2}\mathbf c) +(\mathbf J\bq_{14})(w_{3}\mathbf c) \\
&=
\mathbf J\left(
w_1\bq_{12}+w_2\bq_{13}+w_3\bq_{14}
\right) ^T \mathbf c.
\end{aligned}
\]
By Corollary~\ref{cor:ell-factorization},
\[
w_1\bq_{12}+w_2\bq_{13}+w_3\bq_{14}=\mathbf0.
\]
Hence $\mathbf v^T\mathbf W\mathbf c=0$ for every $\mathbf c\in\mathbb R^2$.
Therefore $\mathbf v^T\mathbf W=\mathbf0$. Equivalently, every vector of the form
$\mathbf W\mathbf c$ is orthogonal to $\mathbf v$, so
\[
\operatorname{range}(\mathbf W)\subseteq\mathbf v^\perp
=\mathcal W.
\]
\end{proof}

For any $\mathbf x\in\mathbb R^6$,
\[
\widetilde L\mathbf x=\sigma\mathbf W\mathbf W^T\mathbf x=\sigma\mathbf
W\mathbf c,\qquad \mathbf c:=\mathbf W^T\mathbf x\in\mathbb R^2,
\]
so $\widetilde L\mathbf x$ always lies in $\operatorname{range}(\mathbf
W)$. By Lemma~\ref{lem:F_range_in_W}, $\operatorname{range}(\mathbf
W)\subseteq\mathcal W$, and hence $\operatorname{range}(\widetilde
L)\subseteq\mathcal W$. Moreover, since $\mathbf W$ has full column rank $2$ and $\sigma\neq0$ (by
hypothesis of Corollary~\ref{cor:ell-factorization}),
$\operatorname{range}(\widetilde L)=\operatorname{range}(\mathbf W)$, that is, $\operatorname{range}(\widetilde L)$ is  a
$2$-dimensional subspace of the five-dimensional rotation-free space
$\mathcal W$. The following proposition makes this observation precise:
it reduces the definiteness of $H_{\mathcal C}|_{\mathcal W}$ to the
spectrum of an explicit $2\times2$ matrix.

\begin{proposition}[Coordinate-free reduction of the generalized spectrum]
\label{prop:coordinate-free-reduction}
Let $\mathbf{v} = \mathbb{J} \mathbf{q}$ denote the rotational mode. Suppose
that $\mathbf W = \mathbf{w} \otimes \mathbf{I}_2 \in \mathbb{R}^{6 \times 2}$ and
$\widetilde{L} = \sigma \mathbf W \mathbf W^T$, where $\mathbf{w} \in \mathbb{R}^3$ is the
eigenvector from Corollary~\ref{cor:ell-factorization} and $\sigma \neq 
0  $ is the
corresponding eigenvalue. Then the nonzero \emph{generalized eigenvalues} of
the pair $(\widetilde{L}|_{\mathcal{W}}, L^\Delta|_{\mathcal{W}})$ coincide
with the eigenvalues of the $2 \times 2$ matrix
\[
\Omega := \sigma \mathbf W^T (L^\Delta_{\mathrm{reg}})^{-1} \mathbf W.
\]
In particular,
\[
	H_{\mathcal{C}}|_{\mathcal{W}} = L^\Delta|_{\mathcal{W}} + \widetilde{L}|_{\mathcal{W}}
\]
is positive definite if and only if both eigenvalues of $\Omega$ are
greater than $-1$.
\end{proposition}

\begin{proof}
Write $L := L^\Delta|_{\mathcal W}$ for the restriction of $L^\Delta$ to
$\mathcal W$, which is positive definite by Lemma~\ref{lem:LDelta_reg_pd}.
Since $\mathbf v \in \ker L^\Delta$ and $\mathbf v \perp \mathcal W$ by
definition of $\mathcal W$, the rank-one correction
$\mathbf v \mathbf v^T / \|\mathbf v\|^2$ vanishes on $\mathcal W$ and acts
as the identity on $\operatorname{span}\{\mathbf v\}$. Hence
$L^\Delta_{\mathrm{reg}}$ preserves the orthogonal splitting
$\mathbb R^6 = \mathcal W \oplus \operatorname{span}\{\mathbf v\}$, acting as
$L$ on $\mathcal W$ and as the identity on $\operatorname{span}\{\mathbf
v\}$. Consequently $(L^\Delta_{\mathrm{reg}})^{-1}$ also preserves this
splitting and restricts to $L^{-1}$ on $\mathcal W$:
\begin{equation}\label{eq:Lreg_inv_restriction}
(L^\Delta_{\mathrm{reg}})^{-1}\big|_{\mathcal W} = L^{-1}.
\end{equation}

\smallskip
\noindent\textbf{Step 1: Correspondence of nonzero eigenvalues.}
Consider the generalized eigenvalue problem for the pair $ (\widetilde{L}|_{ \mathcal{W} } ,  L) $, that is,  
find $ \lambda \neq 0 $ an $ \mathbf{x} \in \mathcal{W} \setminus \{{\bf 0}\}$ such that 
\[
\sigma \mathbf W\mathbf W^T\mathbf x=\lambda L\mathbf x.
\]

\smallskip
\noindent\emph{(Generalized eigenvalue \(\Rightarrow\) eigenvalue of
\(\Omega\).)}
Set
\[
\mathbf c:=\mathbf W^T\mathbf x.
\]
If \(\mathbf c=\mathbf0\), since $ \lambda \neq 0 $ the generalized eigenvalue equation gives
\(L\mathbf x=\mathbf0\). Since \(L\) is invertible, this  gives $ \mathbf{x} = \mathbf{0} $ contradicting 
\(\mathbf x\neq\mathbf0\). Hence \(\mathbf c\neq\mathbf0\). Since $ L $ is invertible solving for
\(\mathbf x\) gives
\[
\mathbf x=\frac{\sigma}{\lambda}L^{-1}\mathbf W\mathbf c.
\]
Applying \(\mathbf W^T\)  to both sides of the above equation and using \eqref{eq:Lreg_inv_restriction}, together with the fact that  $ \mathbf{W} \mathbf{c}  \in \operatorname{range}( \mathbf{W}) \subseteq \mathcal{W} $ by Lemma \ref{lem:F_range_in_W},  we obtain
\[
\mathbf c
=
\frac{\sigma}{\lambda}\mathbf W^TL^{-1}\mathbf W\mathbf c
=
\frac{1}{\lambda}
\sigma \mathbf W^T(L^\Delta_{\mathrm{reg}})^{-1}\mathbf W\mathbf c
=
\frac{1}{\lambda}\Omega\mathbf c.
\]
Thus
\[
\Omega\mathbf c=\lambda\mathbf c,
\]
so \(\lambda\) is an eigenvalue of \(\Omega\).

\smallskip
\noindent\emph{(Eigenvalue of \(\Omega\) \(\Rightarrow\) generalized
eigenvalue.)}
Conversely, let
\[
\Omega\mathbf c=\lambda\mathbf c,
\qquad
\lambda\neq0,
\qquad
\mathbf c\neq\mathbf0.
\]
Define
\[
\mathbf x:=\frac{\sigma}{\lambda}L^{-1}\mathbf W\mathbf c.
\]
Since \(\mathbf W\mathbf c\in\operatorname{range}(\mathbf
W)\subseteq\mathcal W\) by Lemma~\ref{lem:F_range_in_W}, and
\((L^\Delta_{\mathrm{reg}})^{-1}\big|_{\mathcal W} = L^{-1}\) by
\eqref{eq:Lreg_inv_restriction}, we have \(\mathbf x\in\mathcal W\).

We first check that applying \(\mathbf W^T\) to \(\mathbf x\) simply
returns \(\mathbf c\). Substituting the definition of \(\mathbf x\),
\[
\mathbf W^T\mathbf x
=
\mathbf W^T\left(\frac{\sigma}{\lambda}L^{-1}\mathbf W\mathbf c\right)
=
\frac{1}{\lambda}\bigl(\sigma\,\mathbf W^TL^{-1}\mathbf W\bigr)\mathbf c
=
\frac{1}{\lambda}\Omega\mathbf c
=
\frac{1}{\lambda}(\lambda\mathbf c)
=
\mathbf c.
\]

Separately, apply \(L\) directly to the definition of \(\mathbf x\):
\[
L\mathbf x
=
L\left(\frac{\sigma}{\lambda}L^{-1}\mathbf W\mathbf c\right)
=
\frac{\sigma}{\lambda}(LL^{-1})\mathbf W\mathbf c
=
\frac{\sigma}{\lambda}\mathbf W\mathbf c,
\]
so that
\[
\lambda L\mathbf x=\sigma\mathbf W\mathbf c.
\]
Then  substitute \(\mathbf c=\mathbf W^T\mathbf x\) (established
above) into the right-hand side to rewrite it purely in terms of
\(\mathbf x\):
\[
\lambda L\mathbf x
=
\sigma \mathbf W\mathbf c
=
\sigma \mathbf W\mathbf W^T\mathbf x,
\]
which is exactly the generalized eigenvalue equation \(\sigma\mathbf
W\mathbf W^T\mathbf x=\lambda L\mathbf x\); hence \(\mathbf x\) solves
it.


It remains to check that $\mathbf x\neq\mathbf 0$. Since $\mathbf c\neq\mathbf 0$ and
$\mathbf W$ has full column rank, $\mathbf W\mathbf c\neq\mathbf 0$; since
$L^{-1}$ is invertible and $\sigma/\lambda\neq0$, this gives
$\mathbf x=(\sigma/\lambda)L^{-1}\mathbf W\mathbf c\neq\mathbf 0$.

The two maps $\mathbf x\mapsto\mathbf c:=\mathbf W^T\mathbf x$
and $\mathbf c\mapsto\mathbf x:=(\sigma/\lambda)L^{-1}\mathbf W\mathbf c$
constructed above are mutually inverse.
Therefore, for each $\lambda\neq0$ these maps give a
linear isomorphism between the $\lambda$-eigenspace of $\Omega$ and the
$\lambda$-generalized-eigenspace of the pencil $(\widetilde L|_{\mathcal
W},L)$, so the two implications above establish a bijection, respecting
multiplicity, between the nonzero eigenvalues of $\Omega$ and the
nonzero generalized eigenvalues of $(\widetilde L|_{\mathcal W},L)$,
proving the first assertion.

\smallskip
\noindent\textbf{Step 2: The positive-definiteness criterion.}
Since \(L\) is positive definite, there is an
invertible matrix \(R\) such that
\[
L=R^TR.
\]
For \(\mathbf x\in\mathcal W\), set \(\mathbf y=R\mathbf x\). Then
\[
\mathbf x^T
\bigl(L+\widetilde L|_{\mathcal W}\bigr)
\mathbf x
=
\mathbf y^T
\bigl(\mathbf I+R^{-T}\,\widetilde L|_{\mathcal W}\, R^{-1}\bigr)
\mathbf y.
\]
Therefore \(H_{\mathcal C}|_{\mathcal W}\) is positive definite if and only
if every eigenvalue of the symmetric matrix
\[
R^{-T}\widetilde L|_{\mathcal W}\,  R^{-1}
\]
is strictly greater than \(-1\). Its eigenvalues are precisely the
generalized eigenvalues of \((\widetilde L|_{\mathcal W},L)\), because
\[
R^{-T}\widetilde L|_{\mathcal W}\,  R^{-1}\mathbf y=\lambda\mathbf y
\quad\Longleftrightarrow\quad \widetilde L|_{\mathcal W}\, \mathbf x=\lambda L\mathbf x,
\qquad
\mathbf x=R^{-1}\mathbf y.
\]
Since \(\widetilde L|_{\mathcal W}\) has rank two, exactly
three generalized eigenvalues of $(\widetilde L|_{\mathcal W},L)$ are zero. By Step~1, its two
nonzero generalized eigenvalues are the two eigenvalues of $\Omega$. 
Hence
\[
H_{\mathcal C}|_{\mathcal W}\succ0
\]
if and only if both eigenvalues of \(\Omega\) are strictly greater than
\(-1\).
\end{proof}

\subsection{Inverse-Free Algebraic Reduction}

While Proposition~\ref{prop:coordinate-free-reduction} gives a clean
theoretical approach, its explicit use requires computing
$(L^\Delta_{\mathrm{reg}})^{-1}$ symbolically. For general edge lengths,
every entry of this inverse is, by Cramer's rule, a ratio of a $5\times5$
minor to the full $6\times6$ determinant, and the resulting expressions
typically remain cumbersome even for structured families such as kites or
isosceles trapezoids. The following proposition avoids this inversion
entirely, replacing it with the generalized characteristic polynomial
\[
P(\lambda):=\det\bigl(\lambda L^\Delta_{\mathrm{reg}}-\widetilde
L\bigr).
\]
Since $P(\lambda)=\lambda^4\bigl(c_6\lambda^2+c_5\lambda+c_4\bigr)$ has
only three unknown coefficients, these can be recovered directly from a
handful of determinant evaluations --- for instance by evaluating
$P(\lambda)$ at three values of $\lambda$, or by reading off the
coefficients of $\lambda^4,\lambda^5,\lambda^6$ in its expansion ---
without ever forming a matrix inverse. This is symbolically  cheaper
than inverting a  $6\times6$ matrix.

\begin{proposition}[Inverse-free generalized characteristic polynomial]
\label{prop:inverse-free-charpoly}
With the notation of Proposition~\ref{prop:coordinate-free-reduction}, let
\[
P(\lambda)
:=
\det\bigl(\lambda L^\Delta_{\mathrm{reg}}-\widetilde L\bigr).
\]
Then
\[
P(\lambda)
=
\lambda^4\bigl(c_6\lambda^2+c_5\lambda+c_4\bigr),
\qquad
c_6=\det(L^\Delta_{\mathrm{reg}})>0.
\]
Moreover, the two roots of
\[
c_6\lambda^2+c_5\lambda+c_4=0
\]
are precisely the nonzero generalized eigenvalues of
\[
(\widetilde L|_{\mathcal W},\,L^\Delta|_{\mathcal W}).
\]
Consequently, $H_{\mathcal C}|_{\mathcal W}$ is positive definite if and
only if both roots are strictly greater than $-1$.
\end{proposition}

\begin{proof}
By Lemma~\ref{lem:LDelta_reg_pd}, $L^\Delta_{\mathrm{reg}}$ is positive
definite; in particular it is invertible, and
\[
c_6:=\det(L^\Delta_{\mathrm{reg}})>0.
\]
Recall Sylvester's determinant identity: for
$A\in\mathbb R^{m\times n}$ and $B\in\mathbb R^{n\times m}$,
\[
	\det(\mathbf{I}_m-AB)=\det(\mathbf{I}_n-BA).
\]
Writing $\widetilde L=\sigma \mathbf W\mathbf W^T$ with $\mathbf W\in\mathbb R^{6\times2}$, and
factoring
\[
\lambda L^\Delta_{\mathrm{reg}}-\widetilde L
=
\lambda L^\Delta_{\mathrm{reg}}
\left(
\mathbf{I}_6-\frac{\sigma}{\lambda}(L^\Delta_{\mathrm{reg}})^{-1}\mathbf W\mathbf W^T
\right),
\]
we apply Sylvester's identity with
$A=\dfrac{\sigma}{\lambda}(L^\Delta_{\mathrm{reg}})^{-1}\mathbf W\in\mathbb R^{6\times2}$
and $B=\mathbf W^T\in\mathbb R^{2\times6}$:
\begin{align*}
P(\lambda)& = \det(\lambda L^\Delta_{\mathrm{reg}}-\widetilde L)\\
&=
\det(\lambda L^\Delta_{\mathrm{reg}})
\det\!\left(
\mathbf{I}_6-\frac{\sigma}{\lambda}
(L^\Delta_{\mathrm{reg}})^{-1}\mathbf W\mathbf W^T
\right) \\
&=
\lambda^6\det(L^\Delta_{\mathrm{reg}})
\det\!\left(
\mathbf{I}_2-\frac{\sigma}{\lambda}
\mathbf W^T(L^\Delta_{\mathrm{reg}})^{-1}\mathbf W
\right) \\
&=
\lambda^4\det(L^\Delta_{\mathrm{reg}})
\det\!\left(
\lambda \mathbf{I}_2-\sigma
\mathbf W^T(L^\Delta_{\mathrm{reg}})^{-1}\mathbf W
\right) \\
&=
\lambda^4 c_6\det(\lambda \mathbf{I}_2-\Omega).
\end{align*}
Since $\det(\lambda \mathbf{I}_2-\Omega)$ is a monic quadratic in $\lambda$, this
proves
\[
P(\lambda)=\lambda^4\bigl(c_6\lambda^2+c_5\lambda+c_4\bigr),
\]
and the roots of $c_6\lambda^2+c_5\lambda+c_4$ are exactly the eigenvalues
of $\Omega$. By Proposition~\ref{prop:coordinate-free-reduction}, these
coincide with the two potentially  nonzero generalized eigenvalues of
$(\widetilde L|_{\mathcal W},L^\Delta|_{\mathcal W})$, and the
positive-definiteness criterion for $H_{\mathcal C}|_{\mathcal W}$ follows
directly from that proposition.
\end{proof}

\begin{example}[The Newtonian square via the inverse-free polynomial]
\label{ex:square-inverse-free}
We apply Proposition~\ref{prop:inverse-free-charpoly} to the planar
equal-mass Newtonian square. Normalize its side length to $1$, place its
center of mass at the origin, and label the vertices counterclockwise by
\[
\mathbf r_1=
\left(-\frac12,-\frac12\right)^T,
\qquad
\mathbf r_2=
\left(\frac12,-\frac12\right)^T,
\qquad
\mathbf r_3=
\left(\frac12,\frac12\right)^T,
\qquad
\mathbf r_4=
\left(-\frac12,\frac12\right)^T.
\]
The four side lengths are
equal to $1$, and the two diagonal lengths are equal to $\sqrt2$.

Taking body $1$ as the reference point, the pair vectors
$\mathbf q_{ij}=\mathbf r_i-\mathbf r_j$ are
\begin{align*} 
& \mathbf q_{12}=\mathbf r_1-\mathbf r_2=(-1,0)^T,
\qquad
\mathbf q_{13}=\mathbf r_1-\mathbf r_3=(-1,-1)^T,
\qquad
\mathbf q_{14}=\mathbf r_1-\mathbf r_4=(0,-1)^T,\\
& \mathbf q_{23}=\mathbf r_2-\mathbf r_3=(0,-1)^T,
\qquad
\mathbf q_{24}=\mathbf r_2-\mathbf r_4=(1,-1)^T,
\qquad
\mathbf q_{34}=\mathbf r_3-\mathbf r_4=(1,0)^T.
\end{align*} 
In particular,
\[
\mathbf q
=
(\mathbf q_{12}^T,\mathbf q_{13}^T,\mathbf q_{14}^T)^T
=
(-1,0,-1,-1,0,-1)^T.
\]

Let
\[
\mathbf J=
\begin{pmatrix}
0&-1\\
1&0
\end{pmatrix},
\qquad
\mathbb J=\mathbf I_3\otimes\mathbf J.
\]
The infinitesimal rotational mode is
\[
\mathbf v
=
\mathbb J\mathbf q
=
(0,-1,1,-1,1,0)^T,
\qquad
\|\mathbf v\|^2=4.
\]

$L^\Delta$ is the reduced (grounded) weighted graph Laplacian on the
position variables of the free bodies $\{2,3,4\}$, with edge weight
$\delta_{ij}=3m_im_j/q_{ij}^3$ and edge direction $\hat{\bq}_{ij}$
assigned to each pair $(i,j)$. Each diagonal block $(i,i)$ is the
\emph{sum} of the rank-one terms
$\delta_{ij}\hat{\bq}_{ij}\hat{\bq}_{ij}^T$ over every pair
$(i,j)$ incident to body $i$ (including pairs with the grounded body
$1$), while each off-diagonal block $(i,j)$ with $i,j$ both free equals
$-\delta_{ij}\hat{\bq}_{ij}\hat{\bq}_{ij}^T$, coming from
the single pair $(i,j)$. For the square, $\delta_{ij}=3$ on the four
sides and $\delta_{ij}=\tfrac{3\sqrt2}4$ on the two diagonals. This
gives
\[
L^\Delta=
\begin{pmatrix}
3+\tfrac{3\sqrt2}{8} & -\tfrac{3\sqrt2}{8} & 0 & 0 & -\tfrac{3\sqrt2}{8} & \tfrac{3\sqrt2}{8}\\
-\tfrac{3\sqrt2}{8} & 3+\tfrac{3\sqrt2}{8} & 0 & -3 & \tfrac{3\sqrt2}{8} & -\tfrac{3\sqrt2}{8}\\
0 & 0 & 3+\tfrac{3\sqrt2}{8} & \tfrac{3\sqrt2}{8} & -3 & 0\\
0 & -3 & \tfrac{3\sqrt2}{8} & 3+\tfrac{3\sqrt2}{8} & 0 & 0\\
-\tfrac{3\sqrt2}{8} & \tfrac{3\sqrt2}{8} & -3 & 0 & 3+\tfrac{3\sqrt2}{8} & -\tfrac{3\sqrt2}{8}\\
\tfrac{3\sqrt2}{8} & -\tfrac{3\sqrt2}{8} & 0 & 0 & -\tfrac{3\sqrt2}{8} & 3+\tfrac{3\sqrt2}{8}
\end{pmatrix},
\]
and one checks directly that $\mathbf v^TL^\Delta\mathbf v=0$, confirming
$\mathbf v\in\ker L^\Delta$.

Adding back the rotational mode,
\[
L^\Delta_{\mathrm{reg}}
=
L^\Delta+\frac{\mathbf v\mathbf v^T}{\|\mathbf v\|^2}
=
\begin{pmatrix}
3+\tfrac{3\sqrt2}{8} & -\tfrac{3\sqrt2}{8} & 0 & 0 & -\tfrac{3\sqrt2}{8} & \tfrac{3\sqrt2}{8}\\
-\tfrac{3\sqrt2}{8} & \tfrac{13}{4}+\tfrac{3\sqrt2}{8} & -\tfrac14 & -\tfrac{11}{4} & -\tfrac14+\tfrac{3\sqrt2}{8} & -\tfrac{3\sqrt2}{8}\\
0 & -\tfrac14 & \tfrac{13}{4}+\tfrac{3\sqrt2}{8} & -\tfrac14+\tfrac{3\sqrt2}{8} & -\tfrac{11}{4} & 0\\
0 & -\tfrac{11}{4} & -\tfrac14+\tfrac{3\sqrt2}{8} & \tfrac{13}{4}+\tfrac{3\sqrt2}{8} & -\tfrac14 & 0\\
-\tfrac{3\sqrt2}{8} & -\tfrac14+\tfrac{3\sqrt2}{8} & -\tfrac{11}{4} & -\tfrac14 & \tfrac{13}{4}+\tfrac{3\sqrt2}{8} & -\tfrac{3\sqrt2}{8}\\
\tfrac{3\sqrt2}{8} & -\tfrac{3\sqrt2}{8} & 0 & 0 & -\tfrac{3\sqrt2}{8} & 3+\tfrac{3\sqrt2}{8}
\end{pmatrix},
\]
with $L^\Delta_{\mathrm{reg}}\mathbf v=\mathbf v$ and
$\det(L^\Delta_{\mathrm{reg}})=\tfrac{243}{2}(1+2\sqrt2)$.

For the square, Corollary~\ref{cor:ell-factorization} gives
$\widetilde L=\sigma_\square\mathbf W\mathbf W^T$ with
\[
\mathbf w
=
-\frac{4(4+\sqrt2)}{7}(1,-1,1)^T,
\qquad
\sigma_\square
=
\frac{25\sqrt2-44}{256}.
\]
Since the factorization $\widetilde L=\sigma\mathbf w\mathbf w^T$ is
unique only up to the rescaling $\mathbf w\mapsto c\mathbf w,\
\sigma\mapsto\sigma/c^2$ (chosen so that $\sigma\mathbf w\mathbf w^T$,
and hence $\widetilde L$, is unchanged), it is convenient to instead
record the rescaled pair (with $c=-\tfrac{4(4+\sqrt2)}{7}$)
\[
\mathbf w=(1,-1,1)^T,
\qquad
\sigma_\square=-\frac12+\frac{\sqrt2}{8},
\qquad
\mathbf W=\mathbf w\otimes\mathbf I_2,
\]
which yields the identical matrix $\widetilde L$ and is the version we
use below; by construction, $\operatorname{range}(\mathbf W)\subseteq
\mathcal W=\mathbf v^\perp$, as required by
Proposition~\ref{prop:coordinate-free-reduction}. Explicitly,
\[
\widetilde L=\left(-\frac12+\frac{\sqrt2}{8}\right)
\begin{pmatrix}
1&0&-1&0&1&0\\
0&1&0&-1&0&1\\
-1&0&1&0&-1&0\\
0&-1&0&1&0&-1\\
1&0&-1&0&1&0\\
0&1&0&-1&0&1
\end{pmatrix},
\]
which has rank two by construction and satisfies
$\mathbf v^T\widetilde L=\mathbf 0$.

By Proposition~\ref{prop:inverse-free-charpoly}, since $\widetilde L$
has rank two, the generalized characteristic polynomial
$P(\lambda)=\det(\lambda L^\Delta_{\mathrm{reg}}-\widetilde L)$ takes
the form $\lambda^4(c_6\lambda^2+c_5\lambda+c_4)$. Using the matrices
above, direct evaluation yields
\[
P(\lambda)
=
\frac{27}{32}(1+2\sqrt2)
\bigl(\sqrt2-4-12\lambda\bigr)^2\lambda^4.
\]
The coefficient of $\lambda^6$ is therefore
\[
c_6
=
\frac{144}{4}\cdot\frac{27}{8}(1+2\sqrt2)
=
\frac{243}{2}(1+2\sqrt2)
=
\det(L^\Delta_{\mathrm{reg}})>0.
\]

The quadratic factor has the double root
\[
\lambda_0
=
\frac{\sqrt2-4}{12}
=
-\frac13+\frac{\sqrt2}{12}
\approx-0.21548.
\]
Thus the two nonzero generalized eigenvalues of
\[
(\widetilde L|_{\mathcal W},L^\Delta|_{\mathcal W})
\]
coincide and are equal to $\lambda_0$. Since
\[
\lambda_0>-1,
\]
Proposition~\ref{prop:inverse-free-charpoly} gives
\[
H_{\mathcal C}|_{\mathcal W}\succ0.
\]
Hence the square is a nondegenerate strict local minimum modulo the
infinitesimal rotational direction. This conclusion is obtained without
computing $(L^\Delta_{\mathrm{reg}})^{-1}$ or choosing a basis of
$\mathcal W$.
\end{example}

\section{Reflection Symmetry and Symmetry Reduction}
\label{sec:reflection_symmetry}
While the equal-mass square analyzed earlier possesses multiple reflection
symmetries, its stability could be determined directly without symmetry
reduction because its corresponding matrices \(L^\Delta_{\mathrm{reg}}\)
and \(\widetilde{L}\) are purely numerical. For more general symmetric
configurations, however, these matrices contain symbolic parameters such as
variable masses or unspecified edge lengths. In such cases, expanding the
full \(6\times 6\) characteristic polynomial directly quickly becomes
impractical, as the resulting expressions grow too large to analyze
effectively. We now exploit reflection symmetry to rigorously reduce the
dimension of these symbolic calculations.

The argument applies whenever the planar central configuration is
invariant under the joint action of a reflection $\rho$ of the plane and a
matching permutation $\sigma$ of the body labels, in the sense that
applying $\rho$ to every position reproduces the same configuration with
the bodies relabeled by $\sigma$; the precise relationship between $\rho$
and $\sigma$ is given in \eqref{eq:reflection_invariance} below. Here
$O(2)$ denotes the group of linear orthogonal transformations of the
plane: its elements with $\det\rho=1$ are rotations, and those with
$\det\rho=-1$ are reflections, each automatically satisfying
$\rho^2=\mathbf I_2$. More precisely, let $\rho\in O(2)$ be a reflection,
and let $\sigma\in S_4$ satisfy $\sigma^2=\mathrm{id}$ with $\sigma\ne
\mathrm{id}$. We assume that
\begin{equation}\label{eq:reflection_invariance}
    \br_{\sigma(i)}=\rho\br_i,
    \qquad
    m_{\sigma(i)}=m_i,
    \qquad i=1,\dots,4.
\end{equation}

Since $\sigma\ne\mathrm{id}$ satisfies $\sigma^2=\mathrm{id}$, it is a
nontrivial involution, so its cycle decomposition contains only fixed
points and transpositions; for four labels, this leaves exactly two
possible cycle types: two fixed points and one transposition, or two
transpositions. These two cycle types correspond exactly to the two ways
the axis of $\rho$ can meet the configuration: if $\sigma$ fixes two
labels, the corresponding two bodies lie on the axis, giving a {\bf kite}
configuration; if $\sigma$ consists of two transpositions, no body lies on
the axis, giving an {\bf isosceles trapezoidal} configuration. Thus the
restriction to these two combinatorial types is not an additional
assumption, but a direct consequence of $\sigma$ being a nontrivial
involution on four labels. We treat each case in turn.

\subsection{Kite Symmetry}
\label{sec:kite_symmetry}

We consider the reflection type with two fixed labels. Label the
bodies so that the symmetry axis contains bodies $1$ and $3$, while
bodies $2$ and $4$ lie on opposite sides of the axis and are exchanged
by reflection. Thus the configuration is a kite, with cyclic ordering
$1,2,3,4$. Choose the symmetry axis to be the $y$-axis. Then bodies $1$
and $3$ have the form
\[
    \br_1=(0,y_1)^T,
    \qquad
    \br_3=(0,y_3)^T,
\]
and, writing $\br_2=(x,y)^T$, reflection gives
\[
    \br_4=\rho\br_2=(-x,y)^T.
\]
Accordingly,
\[
    \rho=\diag(-1,1),
    \qquad
    \sigma=(2\;4),
    \qquad
    m_2=m_4.
\]

The reflection-invariance relation \eqref{eq:reflection_invariance}
determines how the pair vector $\bq=(\bq_{12},\bq_{13},\bq_{14})$
transforms under $\rho$. Since bodies $1$ and $3$ lie on the axis,
$\rho\br_1=\br_1$ and $\rho\br_3=\br_3$, while $\rho\br_2=\br_4$ and
$\rho\br_4=\br_2$. Consequently,
\[
    \rho\bq_{14}=\rho(\br_1-\br_4)=\br_1-\br_2=\bq_{12},
    \qquad
    \rho\bq_{13}=\rho(\br_1-\br_3)=\br_1-\br_3=\bq_{13},
\]
and applying $\rho$ once more to the first identity, using
$\rho^2=\mathbf I_2$, gives $\rho\bq_{12}=\bq_{14}$.

These relations identify the linear map that governs the
reflection--relabeling operation on grounded pair space. Since
$\sigma(1)=1$, this map acts on the three independent blocks by
\[
    (R\,\cdot)_{1j}=\rho\,(\cdot)_{1\sigma(j)},
\]
and hence takes the matrix form
\begin{equation}\label{eq:R_axial_action}
    R=
    \begin{pmatrix}
        \mathbf{0}&\mathbf{0}&\rho\\
        \mathbf{0}&\rho&\mathbf{0}\\
        \rho&\mathbf{0}&\mathbf{0}
    \end{pmatrix}.
\end{equation}
Since $\rho^2=\mathbf I_2$ and $\rho^T=\rho$, the matrix $R$ is a
symmetric orthogonal involution: $R^2=\mathbf I_6$ and $R^T=R$.

The reflection--relabeling operation acts linearly on the
realizability space $\mathcal C$, and $R$ is precisely its matrix
representative in the coordinates $(\bq_{12},\bq_{13},\bq_{14})$.
Since $\mathcal C$ is a linear subspace, its tangent space is
canonically identified with $\mathcal C$ itself,
\[
    T\mathcal C\simeq\mathcal C\simeq\mathbb R^6,
\]
so $R$ is also the tangent lift of this action. In particular, for
\(\bu=(\bu_{12},\bu_{13},\bu_{14})\in\mathcal C\), the action is given
by \eqref{eq:R_axial_action}, with \(\bq\) replaced by \(\bu\).

Consequently,
\[
    \mathcal C=\mathcal C_+\oplus\mathcal C_-,
    \qquad
    \mathcal C_\pm=\{\bu\in\mathcal C:R\bu=\pm\bu\},
\]
is an orthogonal decomposition into the $+1$ and $-1$ eigenspaces of
$R$. We call these the symmetric and antisymmetric sectors,
respectively. The corresponding orthogonal projections are
\[
    P_\pm=\tfrac12(\mathbf I_6\pm R).
\]

To construct an explicit orthonormal basis for $\mathcal C_+$, we solve
the equation $R\bu=\bu$. By \eqref{eq:R_axial_action}, this means
\[
    \rho\bu_{14}=\bu_{12},
    \qquad
    \rho\bu_{13}=\bu_{13},
    \qquad
    \rho\bu_{12}=\bu_{14}.
\]
Since $\rho^2=\mathbf I_2$, the first and third conditions are equivalent.
Thus $\bu_{12}=(x_{12},y_{12})^T$ is free and
$\bu_{14}=\rho\bu_{12}$. Since $\rho=\diag(-1,1)$, the middle condition
$\rho\bu_{13}=\bu_{13}$ reads
\[
    (-x_{13},y_{13})^T=(x_{13},y_{13})^T,
\]
so $x_{13}=0$, while $y_{13}$ remains free. Hence every vector in
$\mathcal C_+$ has the form
\begin{equation}\label{eq:even_param}
    \bu_+=(x_{12},\,y_{12},\,0,\,y_{13},\,-x_{12},\,y_{12})^T.
\end{equation}

Similarly, the equation $R\bu=-\bu$ gives
\[
    \rho\bu_{14}=-\bu_{12},
    \qquad
    \rho\bu_{13}=-\bu_{13},
    \qquad
    \rho\bu_{12}=-\bu_{14}.
\]
Again, the first and third conditions are equivalent. The middle condition
reads
\[
    (-x_{13},y_{13})^T=(-x_{13},-y_{13})^T,
\]
so $y_{13}=0$, while $x_{13}$ remains free. Thus the diagonal
perturbation $\bu_{13}$ is horizontal, and every vector in
$\mathcal C_-$ has the form
\begin{equation}\label{eq:odd_param}
    \bu_-=(x_{12},\,y_{12},\,x_{13},\,0,\,x_{12},\,-y_{12})^T.
\end{equation}

Normalizing the three coordinate directions in \eqref{eq:even_param}
and \eqref{eq:odd_param} gives the orthonormal basis matrices
\begin{equation}\label{eq:reflection_adapted_bases}
B_+
=
\frac{1}{\sqrt{2}}
\begin{pmatrix}
 1 & 0 & 0\\
 0 & 1 & 0\\
 0 & 0 & 0\\
 0 & 0 & \sqrt{2}\\
-1 & 0 & 0\\
 0 & 1 & 0
\end{pmatrix},
\qquad
B_-
=
\frac{1}{\sqrt{2}}
\begin{pmatrix}
 1 & 0 & 0\\
 0 & 1 & 0\\
 0 & 0 & \sqrt{2}\\
 0 & 0 & 0\\
 1 & 0 & 0\\
 0 &-1 & 0
\end{pmatrix}.
\end{equation}
Their ranges are $\mathcal C_+$ and $\mathcal C_-$, respectively, and
\[
    B_\pm^TB_\pm=\mathbf I_3,
    \qquad
    B_+^TB_-=0,
    \qquad
    P_\pm=B_\pm B_\pm^T.
\]
Therefore,
\[
    Q=[B_+\mid B_-]\in O(6)
\]
is an orthogonal change-of-basis matrix to reflection-adapted
coordinates.

Before reducing the generalized eigenvalue problem by parity, we locate
the configuration vector and the rotational mode in the two reflection
sectors.
\begin{lemma}[Parity of the configuration and rotational modes for the kite]
\label{lem:kite-parity}
For the kite reflection $R$ defined by \eqref{eq:R_axial_action}, the
unperturbed configuration vector $\bq$ is reflection-even,
$\bq\in\mathcal C_+$, and the rotational mode $\bv=\mathbb J\bq$ is
reflection-odd, $\bv\in\mathcal C_-$.
\end{lemma}
\begin{proof}
The relations $\rho\bq_{14}=\bq_{12}$, $\rho\bq_{13}=\bq_{13}$, and
$\rho\bq_{12}=\bq_{14}$ established above are exactly the condition
$R\bq=\bq$:
\[
    R\bq
    =
    \begin{pmatrix}\rho\bq_{14}\\ \rho\bq_{13}\\ \rho\bq_{12}\end{pmatrix}
    =
    \begin{pmatrix}\bq_{12}\\ \bq_{13}\\ \bq_{14}\end{pmatrix}
    =\bq,
\]
so $\bq\in\mathcal C_+$.

For the rotational mode, recall $\mathbb J=\mathbf I_3\otimes \mathbf{J}$,
$\mathbf{J}=\begin{pmatrix}0&-1\\1&0\end{pmatrix}$, and $\bv=\mathbb J\bq$.
Since $\rho$ is a planar reflection, it reverses orientation, so
$\rho \mathbf{J}=-\mathbf{J}\rho$, and therefore $R$ anticommutes with $\mathbb J$:
$R\mathbb J=-\mathbb JR$. Hence
\[
    R\bv=R\mathbb J\bq=-\mathbb JR\bq=-\mathbb J\bq=-\bv,
\]
using $R\bq=\bq$ from above. Thus $\bv\in\mathcal C_-$.
\end{proof}

\subsection{Isosceles Trapezoid Symmetry}
\label{sec:trapezoid_symmetry}

We now consider the second reflection type, in which the symmetry axis
contains no body. Label the vertices of a convex isosceles trapezoid
counterclockwise so that bodies $1$ and $2$ form the lower base and
bodies $3$ and $4$ form the upper base, with body $1$ reflected to
body $2$ and body $3$ reflected to body $4$. Align the symmetry axis
with the $y$-axis. Then
\[
    \rho=\diag(-1,1),
    \qquad
    \sigma=(1\;2)(3\;4),
\]
and reflection symmetry requires
\[
    \br_2=\rho\br_1,
    \qquad
    \br_4=\rho\br_3,
    \qquad
    m_1=m_2,
    \qquad
    m_3=m_4.
\]

The reflection-invariance relation \eqref{eq:reflection_invariance}
determines how the pair vector $\bq=(\bq_{12},\bq_{13},\bq_{14})$
transforms under $\rho$. Since $\br_2=\rho\br_1$ and $\br_4=\rho\br_3$,
and $\rho^2=\mathbf I_2$,
\[
    \rho\bq_{12}
    =\rho(\br_1-\br_2)
    =\rho\br_1-\br_1
    =-(\br_1-\rho\br_1)
    =-\bq_{12}.
\]
Similarly,
\[
    \bq_{14}-\bq_{12}
    =(\br_1-\br_4)-(\br_1-\br_2)
    =\br_2-\br_4
    =\rho\br_1-\rho\br_3
    =\rho\bq_{13},
\]
and
\[
    \rho(\bq_{13}-\bq_{12})
    =\rho(\br_2-\br_3)
    =\rho(\rho\br_1-\br_3)
    =\br_1-\rho\br_3
    =\bq_{14}.
\]

These three relations show that reflecting and relabeling the grounded
pair vector is implemented by a fixed linear map, obtained by sending
edge $(1,j)$ to $(\sigma(1),\sigma(j))=(2,\sigma(j))$ and rewriting the
result in terms of the independent pairs via the triangle relations
$\bq_{23}=\bq_{13}-\bq_{12}$, $\bq_{24}=\bq_{14}-\bq_{12}$. Carrying
this out for each independent edge yields the matrix
\begin{equation}\label{eq:R_trapezoid_action}
    R=
    \begin{pmatrix}
        -\rho & \mathbf{0} & \mathbf{0} \\
        -\rho & \mathbf{0} & \rho \\
        -\rho & \rho & \mathbf{0}
    \end{pmatrix}.
\end{equation}

Because the reflection--relabeling operation acts linearly on the
realizability space $\mathcal C$, the matrix $R$ is also the tangent
lift of this action: it governs every tangent perturbation
$\bu=(\bu_{12},\bu_{13},\bu_{14})\in\mathcal C$ by
\[
    R\bu
    =
    \begin{pmatrix}
        -\rho\bu_{12}\\
        \rho(\bu_{14}-\bu_{12})\\
        \rho(\bu_{13}-\bu_{12})
    \end{pmatrix},
\]
with no separate definition required.
Since $R^2=\mathbf I_6$, every
$\bu\in\mathcal C$ decomposes as
\[
    \bu=\bu_++\bu_-,
    \qquad
    \bu_+:=\frac12(\bu+R\bu),
    \qquad
    \bu_-:=\frac12(\bu-R\bu),
\]
with $R\bu_+=\bu_+$ and $R\bu_-=-\bu_-$. Thus
\[
    \mathcal C=\mathcal C_+\oplus\mathcal C_-,
    \qquad
    \mathcal C_\pm=\{\bu\in\mathcal C:R\bu=\pm\bu\}.
\]

To construct explicit bases for these sectors, we solve the eigenvalue
conditions directly. For the even sector $\mathcal C_+$, the condition
$R\bu=\bu$ applied to \eqref{eq:R_trapezoid_action} requires
$\bu_{12}=-\rho\bu_{12}$. Since $\rho=\diag(-1,1)$, this implies
$(x_{12},y_{12})=(x_{12},-y_{12})$, which forces $y_{12}=0$.

The remaining equation gives $\bu_{13}=\rho(\bu_{14}-\bu_{12})$.
Writing $\bu_{13}=(x_{13},y_{13})^T$ and $\bu_{14}=(x_{14},y_{14})^T$
yields $(x_{13},y_{13})^T=(-x_{14}+x_{12},\,y_{14})^T$, which implies
$y_{14}=y_{13}$ and $x_{14}=x_{12}-x_{13}$. Thus every vector in
$\mathcal C_+$ has the form
\begin{equation}\label{eq:even_param_trapezoid}
    \bu_+=(x_{12},\,0,\,x_{13},\,y_{13},\,x_{12}-x_{13},\,y_{13})^T.
\end{equation}

Similarly, the odd condition $R\bu=-\bu$ requires $\bu_{12}=\rho\bu_{12}$.
This implies $(x_{12},y_{12})=(-x_{12},y_{12})$, which forces $x_{12}=0$.
The subsequent constraint $-\bu_{13}=\rho(\bu_{14}-\bu_{12})$ expands
to $(-x_{13},-y_{13})=(-x_{14},y_{14}-y_{12})$, yielding $x_{14}=x_{13}$
and $y_{14}=y_{12}-y_{13}$. Thus every vector in $\mathcal C_-$ has the
form
\begin{equation}\label{eq:odd_param_trapezoid}
    \bu_-=(0,\,y_{12},\,x_{13},\,y_{13},\,x_{13},\,y_{12}-y_{13})^T.
\end{equation}

Unlike the kite case, isolating the independent coordinate directions in
\eqref{eq:even_param_trapezoid} and \eqref{eq:odd_param_trapezoid} produces
basis vectors that are not mutually orthogonal (for instance, the
$x_{12}$ and $x_{13}$ directions in $\bu_+$ overlap). Applying the
Gram-Schmidt orthonormalization process to these directions within each
subspace produces the following $6\times3$ orthonormal basis matrices:
\begin{equation}\label{eq:trapezoid_adapted_bases}
B_+
=
\begin{pmatrix}
 \frac{1}{\sqrt{2}} & \frac{1}{\sqrt{6}} & 0 \\
 0 & 0 & 0 \\
 \frac{1}{\sqrt{2}} & -\frac{1}{\sqrt{6}} & 0 \\
 0 & 0 & \frac{1}{\sqrt{2}} \\
 0 & \frac{2}{\sqrt{6}} & 0 \\
 0 & 0 & \frac{1}{\sqrt{2}}
\end{pmatrix},
\qquad
B_-
=
\begin{pmatrix}
 0 & 0 & 0 \\
 \frac{1}{\sqrt{2}} & \frac{1}{\sqrt{6}} & 0 \\
 0 & 0 & \frac{1}{\sqrt{2}} \\
 \frac{1}{\sqrt{2}} & -\frac{1}{\sqrt{6}} & 0 \\
 0 & 0 & \frac{1}{\sqrt{2}} \\
 0 & \frac{2}{\sqrt{6}} & 0
\end{pmatrix}.
\end{equation}

A subtle but important distinction arises here. Unlike the kite action
\eqref{eq:R_axial_action}, the matrix $R$ in
\eqref{eq:R_trapezoid_action} is neither symmetric nor orthogonal with
respect to the standard Euclidean inner product on rooted pair
coordinates:
\[
    R^T\neq R,
    \qquad
    R^TR\neq\mathbf I_6.
\]
Consequently, its $+1$ and $-1$ eigenspaces need not be Euclidean
orthogonal. Indeed, $B_+^TB_-\neq0$. Thus $Q=[B_+\mid B_-]$ is
invertible, but not orthogonal ($Q^TQ\neq\mathbf I_6$).

Although the rooted trapezoid action is not orthogonal in the standard
Euclidean metric, the two quadratic forms arising from the Hessian
splitting remain invariant under it. The following proposition establishes
this invariance.
\begin{proposition}\label{prop:trapezoid-split-invariance}
For the trapezoidal reflection $\sigma=(1\;2)(3\;4)$ and
$\rho=\diag(-1,1)$, the two terms in the splitting
\[
    H_{\mathcal C}=L^\Delta+\widetilde L
\]
are separately invariant under the rooted-coordinate action
\eqref{eq:R_trapezoid_action}:
\[
    R^T L^\Delta R=L^\Delta,
    \qquad
    R^T\widetilde L R=\widetilde L.
\]
\end{proposition}
\begin{proof}
Work first in the ungrounded body-coordinate space $(\mathbb R^2)^4$. Let
$\mathbf z=(\mathbf z_1,\dots,\mathbf z_4)\in(\mathbb R^2)^4$, and let
$\mathcal R\in\mathbb R^{8\times8}$ be the combined reflection--relabeling
action
\[
    (\mathcal R\mathbf z)_i=\rho\mathbf z_{\sigma(i)}.
\]
For each edge $\{i,j\}$, let $\mathbf e_1,\ldots,\mathbf e_4$ denote the
standard basis of $\mathbb R^4$ and set
\[
    E_{ij}=(\mathbf e_i-\mathbf e_j)(\mathbf e_i-\mathbf e_j)^T
    \in\mathbb R^{4\times4},
    \qquad\text{so that}\qquad
    \mathbf z^T(E_{ij}\otimes\mathbf I_2)\mathbf z=\lVert\mathbf z_i-\mathbf z_j\rVert^2.
\]
The full transverse and gap Laplacians are
\[
    L_{\mathrm{full}}
    =\sum_{i<j}\alpha_{ij}^{\perp}(E_{ij}\otimes\mathbf I_2),
    \qquad
L^{\Delta}_{\mathrm{full}}
    =\sum_{i<j}\delta_{ij}
      \bigl(E_{ij}\otimes\hat{\bq}_{ij}\hat{\bq}_{ij}^{T}\bigr).
\]
\smallskip
\noindent\textbf{Claim: these are the full Laplacians of
Section~\ref{sec:hessian}.} For any $k,l\in\{1,\dots,4\}$,
\[
    (E_{ij})_{kl}=(\mathbf e_i-\mathbf e_j)_k(\mathbf e_i-\mathbf e_j)_l
    =\begin{cases}
        1 & k=l\in\{i,j\},\\
        -1 & \{k,l\}=\{i,j\},\ k\neq l,\\
        0 & \text{otherwise,}
    \end{cases}
\]
since $(\mathbf e_i-\mathbf e_j)_k$ vanishes unless $k\in\{i,j\}$, in which
case it equals $1$ (if $k=i$) or $-1$ (if $k=j$). Consequently, summing
over all edges,
\[
    \Bigl(\sum_{i<j}\alpha_{ij}^\perp E_{ij}\Bigr)_{kk}
    =\sum_{j\neq k}\alpha_{kj}^\perp,
    \qquad
    \Bigl(\sum_{i<j}\alpha_{ij}^\perp E_{ij}\Bigr)_{kl}
    =-\alpha_{kl}^\perp\quad(k\neq l),
\]
which are precisely the entrywise defining relations for
$\ell_{\mathrm{full}}$.
Hence $\sum_{i<j}\alpha_{ij}^\perp E_{ij}=\ell_{\mathrm{full}}$, so
$L_{\mathrm{full}}=\ell_{\mathrm{full}}\otimes\mathbf I_2$ exactly as
defined in Section~\ref{sec:pairspace} and recalled in
Section~\ref{sec:hessian}. The identical argument, with the scalar
weight $\alpha_{ij}^\perp$ replaced by the matrix weight
$\mathbf D_{ij}=\delta_{ij}\hat{\bq}_{ij}\hat{\bq}_{ij}^T$, shows
that $\sum_{i<j}\delta_{ij}(E_{ij}\otimes\hat{\bq}_{ij}\hat{\bq}_{ij}^T)$
agrees with the full gap Laplacian $L^\Delta_{\mathrm{full}}$ already
introduced in the proof of Proposition~\ref{prop:LDelta_psd} in
Section~\ref{sec:hessian}.
\smallskip

The reflection sends the edge $\{i,j\}$ to $\{\sigma(i),\sigma(j)\}$. By
\eqref{eq:reflection_invariance}, $m_{\sigma(i)}=m_i$ and
$\bq_{\sigma(i)\sigma(j)}=\rho\bq_{ij}$, hence
\[
    \alpha_{\sigma(i)\sigma(j)}^\perp=\alpha_{ij}^\perp,
    \qquad
    \delta_{\sigma(i)\sigma(j)}=\delta_{ij},
    \qquad
    \hat{\bq}_{\sigma(i)\sigma(j)}\hat{\bq}_{\sigma(i)\sigma(j)}^T
    =\rho\hat{\bq}_{ij}\hat{\bq}_{ij}^T\rho^T.
\]
Thus $\mathcal R$ merely permutes the edge summands, and therefore
\[
\mathcal R^T L^{\Delta}_{\mathrm{full}}\mathcal R=L^{\Delta}_{\mathrm{full}},
    \qquad
    \mathcal R^T L_{\mathrm{full}}\mathcal R= L_{\mathrm{full}}.
\]

Let $S:(\mathbb R^2)^3\to(\mathbb R^2)^4$,
$S(\mathbf u_2,\mathbf u_3,\mathbf u_4)=(0,\mathbf u_2,\mathbf u_3,\mathbf u_4)$
be the inclusion of the gauge slice with body $1$ grounded, so that
\[
L^\Delta=S^TL^{\Delta}_{\mathrm{full}}S,
    \qquad
    \widetilde L=S^T L_{\mathrm{full}}S.
\]

For every $\mathbf u\in(\mathbb R^2)^3$, the vectors $S R\mathbf u$ and
$\mathcal RS\mathbf u$ differ by a common translation. Using $R$ from
\eqref{eq:R_trapezoid_action} and $S(\mathbf v_2,\mathbf v_3,\mathbf v_4)=(0,\mathbf v_2,\mathbf v_3,\mathbf v_4)$,
\[
    SR\mathbf u=\bigl(0,\ -\rho\mathbf u_2,\ -\rho\mathbf u_2+\rho\mathbf u_4,\ -\rho\mathbf u_2+\rho\mathbf u_3\bigr),
\]
while, using $(\mathcal R\mathbf z)_i=\rho\mathbf z_{\sigma(i)}$ with
$\sigma(1)=2,\ \sigma(2)=1,\ \sigma(3)=4,\ \sigma(4)=3$,
\[
    \mathcal RS\mathbf u=\bigl(\rho\mathbf u_2,\ 0,\ \rho\mathbf u_4,\ \rho\mathbf u_3\bigr).
\]
Subtracting slot by slot,
\[
\begin{aligned}
    (SR\mathbf u)_1-(\mathcal RS\mathbf u)_1&=0-\rho\mathbf u_2=-\rho\mathbf u_2,\\
    (SR\mathbf u)_2-(\mathcal RS\mathbf u)_2&=-\rho\mathbf u_2-0=-\rho\mathbf u_2,\\
    (SR\mathbf u)_3-(\mathcal RS\mathbf u)_3&=(-\rho\mathbf u_2+\rho\mathbf u_4)-\rho\mathbf u_4=-\rho\mathbf u_2,\\
    (SR\mathbf u)_4-(\mathcal RS\mathbf u)_4&=(-\rho\mathbf u_2+\rho\mathbf u_3)-\rho\mathbf u_3=-\rho\mathbf u_2.
\end{aligned}
\]
Thus every slot agrees on the same discrepancy, so
\[
    SR\mathbf u=\mathcal RS\mathbf u+\mathbf 1_4\otimes\mathbf a,
    \qquad
    \mathbf a=-\rho\mathbf u_2.
\]

Since both full Laplacians annihilate translation vectors, replacing
$SR\mathbf u$ by $\mathcal RS\mathbf u$ does not change either
associated quadratic form.

Hence, for $A_{\mathrm{full}}$ equal to $L^{\Delta}_{\mathrm{full}}$ or
$L_{\mathrm{full}}$,
\[
\begin{aligned}
    \mathbf u^TR^T(S^TA_{\mathrm{full}}S)R\mathbf u
    &=(SR\mathbf u)^TA_{\mathrm{full}}(SR\mathbf u)\\
    &=(\mathcal RS\mathbf u)^TA_{\mathrm{full}}(\mathcal RS\mathbf u)\\
    &=(S\mathbf u)^T\mathcal R^TA_{\mathrm{full}}\mathcal RS\mathbf u\\
    &=(S\mathbf u)^TA_{\mathrm{full}}S\mathbf u\\
    &=\mathbf u^T(S^TA_{\mathrm{full}}S)\mathbf u.
\end{aligned}
\]
Since this holds for every $\mathbf u\in(\mathbb R^2)^3$, it follows that
\[
    R^T(S^TA_{\mathrm{full}}S)R=S^TA_{\mathrm{full}}S,
\]
that is, taking $A_{\mathrm{full}}=L^{\Delta}_{\mathrm{full}}$ and
$A_{\mathrm{full}}=L_{\mathrm{full}}$ in turn, and recalling
$L^\Delta=S^TL^{\Delta}_{\mathrm{full}}S$ and
$\widetilde L=S^TL_{\mathrm{full}}S$,
\[
    R^TL^\Delta R=L^\Delta,
    \qquad
    R^T\widetilde L R=\widetilde L.
\]
\end{proof}
We now locate the configuration vector and the rotational mode in the
two eigenspaces $\mathcal C_+$ and $\mathcal C_-$ of the trapezoidal
reflection.
\begin{lemma}[Parity of the configuration and rotational modes for the
trapezoid]
\label{lem:trapezoid-parity}
For the trapezoidal reflection $R$ defined by
\eqref{eq:R_trapezoid_action}, the unperturbed configuration vector
$\bq$ is reflection-even, $\bq\in\mathcal C_+$, and the rotational
mode $\bv=\mathbb J\bq$ is reflection-odd, $\bv\in\mathcal C_-$.
\end{lemma}
\begin{proof}
The relations $\rho\bq_{12}=-\bq_{12}$, $\bq_{14}-\bq_{12}=\rho\bq_{13}$,
and $\rho(\bq_{13}-\bq_{12})=\bq_{14}$ established above give
\[
    R\bq
    =
    \begin{pmatrix}
        -\rho\bq_{12}\\
        \rho(\bq_{14}-\bq_{12})\\
        \rho(\bq_{13}-\bq_{12})
    \end{pmatrix}
    =
    \begin{pmatrix}
        \bq_{12}\\
        \rho(\rho\bq_{13})\\
        \bq_{14}
    \end{pmatrix}
    =
    \begin{pmatrix}
        \bq_{12}\\ \bq_{13}\\ \bq_{14}
    \end{pmatrix}
    =\bq,
\]
so $\bq\in\mathcal C_+$.

For the rotational mode, recall $\mathbb J=\mathbf I_3\otimes \mathbf{J}$,
$\mathbf{J}=\begin{pmatrix}0&-1\\1&0\end{pmatrix}$, and $\bv=\mathbb J\bq$.
Since $\rho$ is a planar reflection, $\rho \mathbf{J}=-\mathbf{J}\rho$. Using this
together with the three relations above,
\[
    -\rho\bv_{12}=-\rho \mathbf{J}\bq_{12}=\mathbf{J}\rho\bq_{12}=\mathbf{J}(-\bq_{12})=-\bv_{12},
\]
\[
    \rho(\bv_{14}-\bv_{12})
    =\rho \mathbf{J}(\bq_{14}-\bq_{12})
    =\rho \mathbf{J}(\rho\bq_{13})
    =\rho(-\rho \mathbf{J})\bq_{13}
    =-\mathbf{J}\bq_{13}
    =-\bv_{13},
\]
\[
    \rho(\bv_{13}-\bv_{12})
    =\rho \mathbf{J}(\bq_{13}-\bq_{12})
    =\rho \mathbf{J}(\rho\bq_{14})
    =\rho(-\rho \mathbf{J})\bq_{14}
    =-\mathbf{J}\bq_{14}
    =-\bv_{14}.
\]
Substituting into \eqref{eq:R_trapezoid_action},
\[
    R\bv=(-\bv_{12},-\bv_{13},-\bv_{14})=-\bv,
\]
so $\bv\in\mathcal C_-$.
\end{proof}
By construction, the columns of $B_+$ and $B_-$ span the $+1$ and $-1$
eigenspaces of $R$, respectively. Equivalently,
\[
    RB_+=B_+,
    \qquad
    RB_-=-B_-.
\]
Thus, if \(Q=[B_+\mid B_-],\) then the matrix of a bilinear form represented by
$A\in\mathbb R^{6\times6}$ in these reflection-adapted coordinates is
\[
    Q^TAQ
    =
    \begin{pmatrix}
        B_+^TAB_+ & B_+^TAB_-\\
        B_-^TAB_+ & B_-^TAB_-
    \end{pmatrix}.
\]
In particular, the off-diagonal block $B_+^TAB_-$ measures the coupling
between the reflection-even and reflection-odd sectors.

Proposition~\ref{prop:trapezoid-split-invariance} shows that both
$L^\Delta$ and $\widetilde L$ satisfy the invariance relation
$R^TAR=A$. The following lemma shows that this invariance forces the
even--odd coupling block to vanish. Consequently, although
$\mathcal C_+$ and $\mathcal C_-$ are not orthogonal in the standard
Euclidean inner product, the quadratic forms associated with
$L^\Delta$ and $\widetilde L$ are block-diagonal in the
reflection-adapted coordinates.
\begin{lemma}\label{lem:parity-vanishing}
Let $A\in\mathbb R^{6\times6}$ satisfy
\(R^TAR=A.\)
Then
\(B_+^TAB_-=0.\)
\end{lemma}
\begin{proof}
Using $R^TAR=A$ together with $RB_+=B_+$ and $RB_-=-B_-$,
\[
    B_+^TAB_-
    =B_+^T(R^TAR)B_-
    =(RB_+)^TA(RB_-)
    =B_+^TA(-B_-)
    =-B_+^TAB_-,
\]
so $B_+^TAB_-=0$.
\end{proof}

By Proposition~\ref{prop:trapezoid-split-invariance}, both $L^\Delta$
and $\widetilde L$ satisfy the hypothesis of
Lemma~\ref{lem:parity-vanishing}. Hence
\(B_+^T L^\Delta B_-=0\), \(B_+^T\widetilde L\,B_-=0.\)

A subtlety arises when passing from $L^\Delta$ to its regularized
version.  Since $R$ is not Euclidean-orthogonal in rooted coordinates,
the Euclidean regularization
$\bv\bv^T/(\bv^T\bv)$ need not be invariant under the rooted action,
in the sense that it need not satisfy
\[
    R^T\left(\frac{\bv\bv^T}{\bv^T\bv}\right)R
    =
    \frac{\bv\bv^T}{\bv^T\bv},
\]
even though $R\bv=-\bv$, by Lemma~\ref{lem:trapezoid-parity}, for the
rotational mode $\bv$.

To obtain an invariant regularization, we instead use the
mass metric induced on the rooted pair coordinates.
Let $\mu_{ij}=m_im_j/M$, as in Section~\ref{sec:pairspace}. The moment
of inertia, expressed in the independent rooted pair vectors, has the
form
\[
    I(\bq)
    =
    \bq^TG\bq,
\]
where
\begin{equation}\label{eq:Gdef}
    G=G_3\otimes\mathbf I_2,
    \qquad
    G_3=
    \begin{pmatrix}
    \mu_{12}+\mu_{23}+\mu_{24} & -\mu_{23} & -\mu_{24}\\
    -\mu_{23} & \mu_{13}+\mu_{23}+\mu_{34} & -\mu_{34}\\
    -\mu_{24} & -\mu_{34} & \mu_{14}+\mu_{24}+\mu_{34}
    \end{pmatrix}.
\end{equation}

Indeed, substituting the triangle relations
\[
    \bq_{23}=\bq_{13}-\bq_{12},\qquad
    \bq_{24}=\bq_{14}-\bq_{12},\qquad
    \bq_{34}=\bq_{14}-\bq_{13}
\]
into
\[
    I(\bq)=\sum_{i<j}\mu_{ij}\|\bq_{ij}\|^2
\]
and collecting the coefficients of
$\|\bq_{12}\|^2$, $\|\bq_{13}\|^2$,
$\|\bq_{14}\|^2$, and the three cross terms gives
\[
    I(\bq)=\bq^T(G_3\otimes\mathbf I_2)\bq.
\]
Since $m_{\sigma(i)}=m_i$, the reflection preserves the pair-space
inertia form $I(\bq)$. Since $I(\bq)=\bq^TG\bq$ in rooted coordinates,
this is equivalent to
\[
    R^TGR=G.
\]

We therefore define the $G$-regularized gap Laplacian by
\[
    L^\Delta_{\mathrm{reg},G}
    =
    L^\Delta+\frac{G\bv\bv^TG}{\bv^TG\bv}.
\]
This correction is itself invariant under the rooted reflection action.
Indeed, $R\bv=-\bv$ by Lemma~\ref{lem:trapezoid-parity} and $R^TGR=G$.
The latter identity also implies
\(R^TG=GR^{-1}=GR,\)
because $R^2=\mathbf I_6$. Hence
\[
\begin{aligned}
R^T\left(
    \frac{G\bv\bv^TG}{\bv^TG\bv}
\right)R
&=
\frac{(R^TG\bv)(\bv^TGR)}{\bv^TG\bv}\\
&=
\frac{G(R\bv)(R\bv)^TG}{\bv^TG\bv}\\
&=
\frac{G\bv\bv^TG}{\bv^TG\bv}.
\end{aligned}
\]
Together with $R^TL^\Delta R=L^\Delta$, this gives
\[
    R^TL^\Delta_{\mathrm{reg},G}R
    =
    L^\Delta_{\mathrm{reg},G}.
\]
Lemma~\ref{lem:parity-vanishing} therefore applies to the regularized
gap Laplacian and yields
\[
    B_+^TL^\Delta_{\mathrm{reg},G}B_-=0.
\]
\begin{remark}[Compatibility with the earlier regularization]
 Replacing the Euclidean
regularization $\bv\bv^T/\bv^T\bv$ in
Proposition~\ref{prop:coordinate-free-reduction} by this $G$-weighted
version does not affect its validity, since the proof there only uses
that the regularized matrix is symmetric positive definite. This
persists here: $L^\Delta$ is positive semidefinite with kernel
$\operatorname{span}(\bv)$ by Proposition~\ref{prop:LDelta_psd} and Lemma~\ref{lem:LDelta_reg_pd},
and $G$ is positive definite, so for $\mathbf x\neq0$,
\[
    \mathbf x^T L^\Delta_{\mathrm{reg},G}\mathbf x
    =
    \mathbf x^TL^\Delta \mathbf x
    +
    \frac{(\bv^TG\mathbf x)^2}{\bv^TG\bv}
    >0.
\]
Indeed, if the first term vanishes, then
$\mathbf x=c\bv$ for some $c\neq0$, and the second term equals
$c^2\bv^TG\bv>0$.
Hence $L^\Delta_{\mathrm{reg},G}$ is symmetric positive definite.
The proof of Proposition~\ref{prop:coordinate-free-reduction} uses only
this property, so the proposition remains valid after replacing
$L^\Delta_{\mathrm{reg}}$ by $L^\Delta_{\mathrm{reg},G}$. Consequently,
the inverse-free adjugate formula \eqref{eq:lambda_adjugate}, presented
in Subsection~\ref{subssec:symmetry_reduced_pencils}
 below, applies unchanged to this
$G$-weighted version. 
\end{remark}

Consequently, despite the non-orthogonality of $B_+$ and $B_-$, the
invertible congruence transformation $Q=[B_+\mid B_-]$ block-diagonalizes
both $L^\Delta_{\mathrm{reg},G}$ and $\widetilde L$:
\[
    Q^TL^\Delta_{\mathrm{reg},G}Q
    =
    \begin{pmatrix}
        B_+^TL^\Delta_{\mathrm{reg},G}B_+ & 0\\
        0 & B_-^TL^\Delta_{\mathrm{reg},G}B_-
    \end{pmatrix},
\]
and similarly for $\widetilde L$. Since $Q$ is invertible,
\[
    \det\!\bigl(Q^T(\lambda
    L^\Delta_{\mathrm{reg},G}-\widetilde L)Q\bigr)
    =
    \det(Q)^2
    \det\!\bigl(\lambda L^\Delta_{\mathrm{reg},G}-\widetilde L\bigr),
\]
so the generalized eigenvalues are unchanged. 
The trapezoidal stability analysis therefore splits into the two
$3\times3$ reflection sectors, to which the inverse-free
characteristic-polynomial method of Subsection~\ref{subssec:symmetry_reduced_pencils}
 below applies directly.

\begin{remark}
The $G$-regularization may be used uniformly. Any mass-preserving
symmetry satisfies $R^TGR=G$, and, whenever $R\bv=\pm\bv$, the
correction $G\bv\bv^TG/(\bv^TG\bv)$ is $R$-invariant. Hence
$L^\Delta_{\mathrm{reg},G}$ respects both the kite and trapezoidal
reflections. In the absence of symmetry  it remains symmetric positive definite,
so the generalized-eigenvalue reduction still applies.
\end{remark}

\subsection{The Symmetry-Reduced Stability Pencils}
\label{subssec:symmetry_reduced_pencils}

We now derive the stability criterion that applies uniformly to both the
kite and trapezoidal cases above. In either case, the bases $B_+$ and
$B_-$ span complementary subspaces $\mathcal C_+$ and $\mathcal C_-$ with
$B_+^TAB_-=0$ for $A\in\{L^\Delta_{\mathrm{reg}},\widetilde L\}$ (using
$L^\Delta_{\mathrm{reg},G}$ in place of $L^\Delta_{\mathrm{reg}}$ for the
trapezoidal case). Consequently, applying the invertible change-of-basis
matrix $Q = [B_+ \mid B_-]$ block-diagonalizes the regularized Hessian
matrices into independent $3 \times 3$  blocks:
\[
    Q^T L^\Delta_{\mathrm{reg}} Q =
    \begin{pmatrix}
        L^\Delta_+ & \mathbf{0} \\
        \mathbf{0} & L^\Delta_-
    \end{pmatrix},
    \qquad
    Q^T \widetilde{L} Q =
    \begin{pmatrix}
        \widetilde{L}_+ & \mathbf{0} \\
        \mathbf{0} & \widetilde{L}_-
    \end{pmatrix},
\]
where $L^\Delta_\pm = B_\pm^T L^\Delta_{\mathrm{reg}} B_\pm$ and
$\widetilde{L}_\pm = B_\pm^T \widetilde{L} B_\pm$. Because the
regularized Laplacian $L^\Delta_{\mathrm{reg}}$ is strictly positive
definite, both $L^\Delta_+$ and $L^\Delta_-$ are invertible $3\times3$
positive-definite matrices. Moreover, $\widetilde L$ has exactly rank
two by Corollary~\ref{cor:tildeL-eigenvalues}. The following two
lemmas show that this rank splits evenly between the two parity
sectors.

\begin{lemma}[Equipartition of the transverse eigenspace, kite case]
\label{lem:transverse-equipartition-kite}
For the kite reflection, $\widetilde L_+$ and $\widetilde L_-$ each
have rank exactly one.
\end{lemma}
\begin{proof}
By definition, $\widetilde L_\pm=B_\pm^T\widetilde LB_\pm$.
Substituting the factorization $\widetilde L=\sigma\mathbf W\mathbf
W^T$ from Corollary~\ref{cor:ell-factorization},
\[
    \widetilde L_\pm
    =B_\pm^T\bigl(\sigma\mathbf W\mathbf W^T\bigr)B_\pm=\sigma\bigl(B_\pm^T\mathbf W\bigr)\bigl(B_\pm^T\mathbf
    W\bigr)^T,
\]
which has the form $\sigma AA^T$ with $A:=B_\pm^T\mathbf
W\in\mathbb R^{3\times2}$. For any real matrix $A$, $\operatorname{rank}(AA^T)=\operatorname{rank}(A^T)=\operatorname{rank}(A)$.
As $\sigma\neq0$, this gives
\[
    \operatorname{rank}(\widetilde L_\pm)
    =\operatorname{rank}(B_\pm^T\mathbf W).
\]

Using \eqref{eq:reflection_adapted_bases} and $\mathbf W=\mathbf
w\otimes\mathbf I_2$, direct computation gives
\[
    B_+^T\mathbf W=
    \begin{pmatrix}p&0\\0&q\\0&r\end{pmatrix},
    \qquad
    B_-^T\mathbf W=
    \begin{pmatrix}q&0\\0&p\\r&0\end{pmatrix},
\]
where
\[
    p=\frac{w_1-w_3}{\sqrt2},\qquad
    q=\frac{w_1+w_3}{\sqrt2},\qquad
    r=w_2.
\]
Swapping the two columns of $B_+^T\mathbf W$, then swapping its first
two rows, produces exactly $B_-^T\mathbf W$. Since row and column
permutations preserve rank,
$\operatorname{rank}(B_+^T\mathbf W)=\operatorname{rank}(B_-^T\mathbf
W)=:r_0$. Since $\operatorname{rank}(\widetilde
L_+)+\operatorname{rank}(\widetilde L_-)=\operatorname{rank}(\widetilde
L)=2$ (Corollary~\ref{cor:tildeL-eigenvalues}), we get $2r_0=2$, so
$r_0=1$.
\end{proof}

\begin{lemma}[Equipartition of the transverse eigenspace, trapezoid
case]
\label{lem:transverse-equipartition-trapezoid}
For the trapezoidal reflection, $\widetilde L_+$ and $\widetilde L_-$
each have rank exactly one.
\end{lemma}
\begin{proof}
As in the proof of Lemma~\ref{lem:transverse-equipartition-kite},
$\widetilde L_\pm=\sigma(B_\pm^T\mathbf W)(B_\pm^T\mathbf W)^T$, so
\[
    \operatorname{rank}(\widetilde L_\pm)
    =\operatorname{rank}(B_\pm^T\mathbf W).
\]
Using \eqref{eq:trapezoid_adapted_bases} and $\mathbf W=\mathbf
w\otimes\mathbf I_2$, direct computation gives
\[
    B_+^T\mathbf W=
    \begin{pmatrix}a&0\\b&0\\0&c\end{pmatrix},
    \qquad
    B_-^T\mathbf W=
    \begin{pmatrix}0&a\\0&b\\c&0\end{pmatrix},
\]
where
\[
    a=\frac{w_1+w_2}{\sqrt2},\qquad
    b=\frac{w_1-w_2+2w_3}{\sqrt6},\qquad
    c=\frac{w_2+w_3}{\sqrt2}.
\]
Since $B_-^T\mathbf W$ is obtained from $B_+^T\mathbf W$ by swapping
its two columns, $\operatorname{rank}(B_+^T\mathbf
W)=\operatorname{rank}(B_-^T\mathbf W)=:r$. Since
$\operatorname{rank}(\widetilde L_+)+\operatorname{rank}(\widetilde
L_-)=\operatorname{rank}(\widetilde L)=2$ (Corollary
\ref{cor:tildeL-eigenvalues}), we get $2r=2$, so $r=1$.
\end{proof}

For both the kite and trapezoidal symmetry reductions, the preceding
lemmas imply that the projected matrices $\widetilde L_+$ and
$\widetilde L_-$ each have rank one. Consequently, the pencil matrix
$\lambda L^\Delta_\pm-\widetilde L_\pm$ is obtained from
$\lambda L^\Delta_\pm$ by a rank-one correction, namely
$-\widetilde L_\pm$. As we make precise below, this means each
$3\times3$ parity block can contribute at most one nonzero
generalized eigenvalue.

The full $6 \times 6$ generalized characteristic polynomial factors
by symmetry as
\[
    \det\bigl(\lambda L^\Delta_{\mathrm{reg}}-\widetilde L\bigr)
    =
    \det\bigl(\lambda L^\Delta_+-\widetilde L_+\bigr)
    \det\bigl(\lambda L^\Delta_- -\widetilde L_-\bigr).
\]
Consequently, the two nonzero physical generalized eigenvalues can be
computed separately from the two $3\times3$ rank-one-perturbation
problems, without forming any matrix inverses.

To evaluate these $3\times3$ determinants analytically, recall that for any
invertible $3\times3$ matrix $A$ and rank-one matrix $B$, the product
$A^{-1}B$ also has rank one. Its two-dimensional kernel guarantees that $0$ is an eigenvalue of
algebraic multiplicity at least two; since the three eigenvalues always
sum to the trace, the remaining eigenvalue must equal
$\operatorname{tr}(A^{-1}B)$. Its characteristic polynomial is therefore simply
$\det(\lambda \mathbf{I} - A^{-1}B) = \lambda^2(\lambda - \operatorname{tr}(A^{-1}B))$.
Factoring out $A$ and substituting $A^{-1} = \operatorname{adj}(A)/\det(A)$
yields the identity
\[
    \det(\lambda A - B)
    = \lambda^2 \left[ \lambda \det(A) - \operatorname{tr}(\operatorname{adj}(A)B) \right].
\]

Applying this identity to our $3\times3$  blocks with $A=L^\Delta_\pm$
and $B=\widetilde L_\pm$ yields
\[
    \det\bigl(\lambda L^\Delta_\pm-\widetilde L_\pm\bigr)
    =
    \lambda^2
    \left[
        \lambda\det(L^\Delta_\pm)
        -
        \operatorname{tr}
        \bigl(
            \operatorname{adj}(L^\Delta_\pm)\widetilde L_\pm
        \bigr)
    \right].
\]
Accordingly, the unique nonzero generalized eigenvalue in each parity sector is given by the explicit rational function 
\begin{equation}\label{eq:lambda_adjugate}
    \lambda_\pm
    =
    \frac{
        \operatorname{tr}
        \bigl(
            \operatorname{adj}(L^\Delta_\pm)\widetilde L_\pm
        \bigr)
    }{
        \det(L^\Delta_\pm)
    }.
\end{equation}
For the trapezoidal case, $L^\Delta_{\mathrm{reg},G}$ is understood
in place of $L^\Delta_{\mathrm{reg}}$ throughout.

Because \eqref{eq:lambda_adjugate} is expressed entirely in terms of
adjugate matrices and determinants, both its numerator and denominator
are polynomial expressions in the original matrix entries. The
rotational direction is excluded from
$\mathcal W=\operatorname{span}(\mathbf v)^\perp$ by definition; its
contribution has instead been absorbed into the rank-one correction
defining $L^\Delta_{\mathrm{reg}}$, or
$L^\Delta_{\mathrm{reg},G}$ in the trapezoidal case. The scaling
direction contributes a positive eigenvalue to $H_{\mathcal C}$
by Lemma~\ref{lem:dilation-positive}. Consequently, strict local
minimality of the central configuration is equivalent to positive
definiteness of
\[
    H_{\mathcal C}|_{\mathcal W}
    =
    L^\Delta|_{\mathcal W}+\widetilde L|_{\mathcal W}.
\]
Since $\lambda_+$ and $\lambda_-$ are precisely the two nonzero
generalized eigenvalues of the pencil
$(\widetilde L,L^\Delta_{\mathrm{reg}})$---equivalently, the two
nonzero eigenvalues of $\Omega$ in
Proposition~\ref{prop:coordinate-free-reduction}---that proposition
implies that strict local minimality holds if and only if
\[
    \lambda_+>-1
    \qquad\text{and}\qquad
    \lambda_->-1.
\]

\section{The Rhombus Family}
\label{sec:rhombus}
In this section, we apply the general theory of Section \ref{sec:four-body-problem}, together with its specialization to configurations invariant under a reflection symmetry developed in Section \ref{sec:reflection_symmetry}, to the family of rhombus central configurations. The rhombus is a convex, reflection-symmetric planar four-body configuration with equal opposite masses and perpendicular diagonals.

The existence and uniqueness of these configurations were established in \cite{corbera2014central} and \cite{long2002four}, respectively, while their nondegeneracy was recently proved in \cite{sun2025degeneracy}. Here, we revisit this configuration from the perspective developed above, and show that the index of the rhombus configuration is zero. 
Since the Morse index is the number of negative eigenvalues of the
Hessian on the normalized realizability subspace, counted with
multiplicity, this proves that the rhombus is a nondegenerate local
minimum modulo rotations. This gives an independent algebraic proof of nondegeneracy within the
pair-space framework developed in this paper.

Consider four bodies with opposite masses equal,
\[
    m_1=m_3=m>0,
    \qquad
    m_2=m_4=1,
    \qquad
    M=2m+2,
\]
placed at the vertices of a rhombus
\begin{equation}
\label{eq:rhombus-coordinates}
    \mathbf{r}_1=(0,a),\qquad
    \mathbf{r}_2=(-1,0),\qquad
    \mathbf{r}_3=(0,-a),\qquad
    \mathbf{r}_4=(1,0),
    \qquad a>0.
\end{equation}
The center of mass is automatically at the origin. The vertical
diagonal has length \(2a\) and joins the two bodies of mass \(m\); the
horizontal diagonal has length \(2\) and joins the two bodies of mass
\(1\).
When convenient for later algebraic simplifications, we write
\begin{equation}
\label{eq:rhombus-side-length}
s=\sqrt{a^2+1},
\end{equation}
the common length of the four sides of the rhombus.

\subsection{Pair vectors and the central-configuration equation}
\label{subsec:rhombus-cc}

The independent pair vectors based at body \(1\) are
\[
    \mathbf{q}_{12}=(1,a),
    \qquad
    \mathbf{q}_{13}=(0,2a),
    \qquad
    \mathbf{q}_{14}=(-1,a),
\]
and the dependent pairs, obtained from the triangle relations
\(\mathbf{q}_{12}+\mathbf{q}_{2j}-\mathbf{q}_{1j}=0\), are
\[
    \mathbf{q}_{23}=(-1,a),
    \qquad
    \mathbf{q}_{24}=(-2,0),
    \qquad
    \mathbf{q}_{34}=(-1,-a).
\]
Thus \(q_{12}=q_{14}=q_{23}=q_{34}=s\), \(q_{13}=2a\), \(q_{24}=2\),
and the corresponding mass parameters are
\begin{equation}
\label{eq:rhombus-mu}
    \mu_{12}=\mu_{14}=\mu_{23}=\mu_{34}=\frac{m}{M},
    \qquad
    \mu_{13}=\frac{m^2}{M},
    \qquad
    \mu_{24}=\frac1M.
\end{equation}
Following the general construction of
Section~\ref{sec:pairspace}, form the augmented functional
\(\mathcal F\) defined in \eqref{eq:F_4body},
\[
	\mathcal F= U+ \frac{ \omega^2}{ 2}I+\sum\boldsymbol \phi_{ij}\cdot(\text{triangle relation})_{ij},
\]
on the ambient pair space, with one vector multiplier
\(\boldsymbol\phi_{ij}\in\mathbb{R}^2\) for each dependent pair
\(ij\in\{23,24,34\}\). Setting \(\nabla \mathcal{F} =0\) and substituting the
rhombus data~\eqref{eq:rhombus-coordinates}--\eqref{eq:rhombus-mu}
yields a linear system in \(\omega^2\) and the six multiplier
components.
Eliminating these seven unknowns reduces the twelve
pair-space equations to a single nontrivial scalar equation, the {\bf rhombus central-configuration equation}
\begin{equation}
\label{eq:rhombus-cc-equation}
    \Phi(a,m)
    =
    \frac{
        m\Bigl[(a^2+1)^{3/2}\bigl(a^3-am\bigr)+8a^3(m-1)\Bigr]
    }{
        4(a^2+1)^{3/2}a^2
    }
    =0.
\end{equation}
The remaining equations are automatically satisfied by the rhombus
symmetry.

Since \(m>0\), the central-configuration equation is equivalent to the
vanishing of the  bracketed numerator of equation~\eqref{eq:rhombus-cc-equation}. Solving the resulting equation
for \(m\) gives
\begin{equation}
\label{eq:rhombus-m-of-a}
    m(a)
    =
    \frac{
        a^2\bigl[(a^2+1)^{3/2}-8\bigr]
    }{
        (a^2+1)^{3/2}-8a^2
    }.
\end{equation}
This expression satisfies \(m(1)=1\), corresponding to the square
configuration, and \(m(a)>0\) for
\(a\in\left(\sqrt{3}/3,\sqrt{3}\right)\).
\subsection{The reduced Hessian and its Laplacian splitting}
\label{subsec:rhombus-hessian}

Let \(\mathbf{u}=(\mathbf{q}_{12},\mathbf{q}_{13},\mathbf{q}_{14})\in
\mathbb{R}^{6}\) be the set of independent pair vectors. 
Following the general construction of Section~\ref{sec:hessian}, the restricted Hessian
\(H_{\mathcal C}(a,m)\) is evaluated at the rhombus configuration
\eqref{eq:rhombus-coordinates}--\eqref{eq:rhombus-mu}, prior to
imposing the central-configuration relation~\eqref{eq:rhombus-cc-equation}.

Rotational invariance of the functional $F$ implies that the tangent vector
to the \(SO(2)\)-orbit through a central configuration lies in the kernel
of the restricted Hessian. At the rhombus, this infinitesimal rotation
vector is
\begin{equation}
\label{eq:rhombus-rotation-vector}
    \mathbf{v}=(-a,1,-2a,0,-a,-1)^{T},
\end{equation}
and hence \(H_{ \mathcal{C}}(a,m)\mathbf{v}=0\) whenever
\eqref{eq:rhombus-cc-equation} holds.

Following the framework of Section~\ref{subsec:hessian_splitting}, the restricted Hessian splits into two distinct components:
\[
H_{\mathcal{C} } = L^\Delta + \widetilde L, \qquad \text{where} \qquad \widetilde L = \ell \otimes \mathbf{I}_2.
\]

The gap Laplacian \(L^\Delta\) has the rotational vector
\(\mathbf v\) in its kernel. Following the generalized eigenvalue
criterion of Section~\ref{subsec:rayleigh_bound}, we remove this
degeneracy by setting
\[
    L^\Delta_{\mathrm{reg}}
    =
    L^\Delta+
    \frac{\mathbf v\mathbf v^T}{\mathbf v^T\mathbf v}.
\]
The resulting matrix is positive definite on the full pair space.
For the rhombus, it has the form
\begin{equation} \label{eq:L_reg_matrix}
L^{\Delta}_{\text{reg}} = \begin{bmatrix}
L_{11} & L_{12} & L_{13} & L_{14} & L_{15} & -L_{12} \\
L_{12} & L_{22} & L_{23} & L_{24} & L_{12} & L_{26} \\
L_{13} & L_{23} & L_{33} & 0 & L_{13} & -L_{23} \\
L_{14} & L_{24} & 0 & L_{44} & -L_{14} & L_{24} \\
L_{15} & L_{12} & L_{13} & -L_{14} & L_{11} & -L_{12} \\
-L_{12} & L_{26} & -L_{23} & L_{24} & -L_{12} & L_{22}
\end{bmatrix},
\end{equation}
where the  rational functions $L_{ij}$ are given in Appendix \ref{appendix-LDelta}. The reduced transverse Laplacian is explicitly given by
\begin{equation}\label{eq:widetildeL-rhombus}
\widetilde L = \ell \otimes \mathbf{I}_2, \qquad \ell = \sigma \begin{pmatrix}
 1 & -1 &  1 \\
-1 &  1 & -1 \\
 1 & -1 &  1
\end{pmatrix}, \qquad \sigma = \frac{\left(s a^2 + s - 8\right)^2 a^3}{8 \left(\left(a^5 + a^3 + a^2 + 1\right) s - 16 a^3\right) s^3},
\end{equation} 
where  $s = \sqrt{a^2+1}$ denotes the common diagonal distance.

The rhombus is invariant under the reflection $(x,y)\mapsto(-x,y)$. As in
Section~\ref{sec:reflection_symmetry}, this reflection induces a symmetric orthogonal
involution $R$ on the pair space $\mathcal C\simeq\mathbb R^6$, which
commutes with $H_{\mathcal C}$, $L^\Delta$, and $\widetilde L$. Hence
$\mathcal C$ splits orthogonally into the symmetric and antisymmetric
sectors introduced in \eqref{eq:reflection_invariance}--\eqref{eq:R_axial_action}:
\[
    \mathcal C=\mathcal C_+\oplus\mathcal C_-,
    \qquad
    \dim\mathcal C_+=\dim\mathcal C_-=3.
\]

By projecting our matrices onto these invariant subspaces  using the adapted bases $B_+$ and $B_-$ from \eqref{eq:reflection_adapted_bases}, the regularized constraint matrix $L^\Delta_{\text{reg}}$ block-diagonalizes into the two $3\times 3$ symmetry-reduced matrices
\begin{align*}
L^\Delta_+ = B_+^T L^\Delta_{\text{reg}} B_+ &= \begin{pmatrix}
L_{11} - L_{15} & 0 & \sqrt{2} L_{14} \\[0.5em]
0 & L_{22} + L_{26} & \sqrt{2} L_{24} \\[0.5em]
\sqrt{2} L_{14} & \sqrt{2} L_{24} & L_{44}
\end{pmatrix}, \\[1em]
L^\Delta_- = B_-^T L^\Delta_{\text{reg}} B_- &= \begin{pmatrix}
L_{11} + L_{15} & 2 L_{12} & \sqrt{2} L_{13} \\[0.5em]
2 L_{12} & L_{22} - L_{26} & \sqrt{2} L_{23} \\[0.5em]
\sqrt{2} L_{13} & \sqrt{2} L_{23} & L_{33}
\end{pmatrix}.
\end{align*}

Similarly, the transverse Laplacian reduces to \(\widetilde L_\pm=B_\pm^T\widetilde L B_\pm\)
and these reduced matrices are determined entirely by the scalar
factor \(\sigma\):
\begin{equation*}
\widetilde{L}_+ = \begin{pmatrix}
0 & 0 & 0 \\
0 & 2\sigma & -\sqrt{2}\sigma \\
0 & -\sqrt{2}\sigma & \sigma
\end{pmatrix}, \qquad 
\widetilde{L}_- = \begin{pmatrix}
2\sigma & 0 & -\sqrt{2}\sigma \\
0 & 0 & 0 \\
-\sqrt{2}\sigma & 0 & \sigma
\end{pmatrix}.
\end{equation*}
These symmetric reductions dramatically simplify the spectral analysis. Substituting the symmetry-reduced matrices \(L_{+}^\Delta\), \(L_{-}^\Delta\), and \(\widetilde L_{\pm}\) into the general determinant formula \eqref{eq:lambda_adjugate} yields closed-form expressions for the two eigenvalues \(\lambda_\pm(a)\). On the reflection-even sector \(\mathcal C_+\), we obtain:
\begin{equation} \label{eq:rhombus-lambda+}
    \lambda_{+}(a) = \frac{\operatorname{tr}\bigl(\operatorname{adj}(L_{+}^{\Delta})\,\widetilde L_{+}\bigr)}{\det\bigl(L_{+}^{\Delta}\bigr)} = -\frac{\left(a^{2} s+s-8\right) \left(8 a^{3}-a^{2} s-s\right) \left(a^{2}+1\right)}{24 \left(\left(a^{5}+a^{3}+a^{2}+1\right) s-16 a^{3}\right) a^{2}}.
\end{equation} 
On the reflection-odd sector \(\mathcal C_-\), we obtain:
\begin{equation}
\label{eq:rhombus-lambda-}
    \lambda_{-}(a) = \frac{\operatorname{tr}\bigl(\operatorname{adj}(L_{-}^{\Delta})\,\widetilde L_{-}\bigr)}{\det\bigl(L_{-}^{\Delta}\bigr)} = \frac{16 s (s^3 - 8)^2 (8a^3 - s^3) \big[ P_1(a) s + P_2(a) s^4 \big]}{Q_1(a) s^3 + Q_2(a)},
\end{equation}
where the numerator polynomials $P_1(a)$ and $P_2(a)$ are defined as:
\begin{align*}
P_1(a) &= -\frac{1}{16} \Big( 65a^{15} + 198a^{13} + 449a^{12} + 207a^{11} + 1350a^{10} + 8532a^9 + 1359a^8 \\
&\quad + 783a^7 + 468a^6 + 774a^5 + 15a^4 + 257a^3 + 6a^2 + 1 \Big), \\[1em]
P_2(a) &= \frac{1}{2} (2a^{12} + 6a^{10} + 197a^9 + 6a^8 + 15a^7 + 323a^6 + 15a^5 + 3a^4 + 5a^3 + 3a^2 + 1),
\end{align*}
and the denominator polynomials $Q_1(a)$ and $Q_2(a)$ are compactly factored as:
\begin{align*}
Q_1(a) &= -384(a^3 + 1)(a^6 + 3a^4 + 192a^3 + 3a^2 + 1) \\
&\quad \times (a^{12} + 3a^{10} + 4a^9 + 3a^8 + 12a^7 + 130a^6 + 12a^5 + 3a^4 + 4a^3 + 3a^2 + 1), \\[1em]
Q_2(a) &= 24 \Big( 65a^{24} + 393a^{22} + 898a^{21} + 996a^{20} + 5394a^{19} + 56533a^{18} + 13512a^{17} \\
&\quad + 172359a^{16} + 158248a^{15} + 189090a^{14} + 436860a^{13} + 1193768a^{12} + 436860a^{11} \\
&\quad + 189090a^{10} + 158248a^9 + 172359a^8 + 13512a^7 + 56533a^6 + 5394a^5 \\
&\quad + 996a^4 + 898a^3 + 393a^2 + 65 \Big).
\end{align*}
Although the expression for \(\lambda_-(a)\) is more involved, it admits
a tractable sign structure once the radical \(\sqrt{a^2+1}\) is removed
by a suitable rational parametrization. Together with the analogous
reduction for \(\lambda_+(a)\), this parametrization lets us prove
positivity of the rhombus Hessian on the rotation-free subspace in the
next proposition.

\begin{proposition}
\label{prop:rhombus-index}
For every rhombus central configuration with shape parameter
\(a\in\left(\sqrt{3}/3,\sqrt3\right)\)
and mass \(m=m(a)\) satisfying~\eqref{eq:rhombus-cc-equation}, the
restricted Hessian $H_{\mathcal{C}}$ is positive definite on the
rotation-free subspace. Consequently, the Morse index of the rhombus
central configuration is zero.
\end{proposition}

\begin{proof}
Both \(\lambda_+(a)\) and \(\lambda_-(a)\) are algebraic functions of
\(a\) involving the radical \(\sqrt{a^2+1}\). We remove it using the
rational parametrization
\begin{equation}
\label{eq:rhombus-parametrization}
    a=\frac{1-t^2}{2t},
    \qquad
    \sqrt{a^2+1}=\frac{1+t^2}{2t},
    \qquad t>0.
\end{equation}
Since \(a(t)\) is strictly decreasing for \(t>0\), the interval
\(a\in\bigl(\sqrt3/3,\sqrt3\bigr)\) corresponds exactly to
\(t\in\bigl(2-\sqrt3,\sqrt3/3\bigr)\), with
\[
a(2-\sqrt3)=\sqrt3,
\qquad
a\Bigl(\tfrac{\sqrt3}{3}\Bigr)=\frac{\sqrt3}{3}.
\]
It therefore suffices to prove positivity for
\(t\in\bigl(2-\sqrt3,\tfrac{\sqrt3}{3}\bigr)\).

Under~\eqref{eq:rhombus-parametrization}, the quantity
\(1+\lambda_{+}(a)\) becomes the rational function
\[
\label{eq:rhombus-rational-eta+}
    1+\lambda_{+}(t)=\frac{N_+(t)}{D_+(t)},
\]
where
\begin{align*}
N_+(t)={}&
15t^{16}-72t^{14}-2688t^{13}-60t^{12}+14976t^{11}+136t^{10}
-30720t^{9}\\
&\quad -70t^{8}+32768t^{7}-248t^{6}-14720t^{5}+4t^{4}+2432t^{3}
+56t^{2}-17,
\\[1ex]
D_+(t)={}&
24\bigl(t^{12}-136t^{9}-3t^{8}+360t^{7}-408t^{5}+3t^{4}+120t^{3}-1\bigr)
(t+1)^2(t-1)^2.
\end{align*}

A direct Sturm-sequence computation shows that \(N_{+}\) and \(D_{+}\)
each have the same number of sign variations at the two endpoints
\(2-\sqrt3\) and \(\sqrt3/3\); hence neither polynomial has a root in
\(\bigl(2-\sqrt3,\sqrt3/3\bigr)\), and each therefore has constant
sign throughout the interval.

Since
\[
    N_+\Bigl(\frac13\Bigr)=\frac{458104064}{14348907}>0,
    \qquad
    D_+\Bigl(\frac13\Bigr)=\frac{533233664}{14348907}>0,
\]
it follows that \(1+\lambda_+(t)>0\) throughout the interval, and
hence \(1+\lambda_+(a)>0\) for all
\(a\in\bigl(\sqrt3/3,\sqrt3\bigr)\).

Similarly, the quantity \(1+\lambda_{-}(a)\) becomes
\[
    1+\lambda_{-}(t)=\frac{N_-(t)}{D_-(t)},
\]
where
\begin{align*}
N_-(t)={}&
-9t^{16}+72t^{14}+576t^{13}-12t^{12}-13248t^{11}-296t^{10}
+34176t^{9}\\
&\quad -70t^{8}-37504t^{7}+184t^{6}+12352t^{5}-44t^{4}-448t^{3}
-88t^{2}+7,
\\[1ex]
D_-(t)={}&
96\,t^{2}\bigl(t^{12}-136t^{9}-3t^{8}+360t^{7}-408t^{5}+3t^{4}
+120t^{3}-1\bigr).
\end{align*}
Again, Sturm's theorem shows that \(N_-(t)\) and \(D_-(t)\) have no
roots in \(\bigl(2-\sqrt3,\sqrt3/3\bigr)\). Since
\[
    N_-\Bigl(\frac13\Bigr)=\frac{74938112}{4782969}>0,
    \qquad
    D_-\Bigl(\frac13\Bigr)=\frac{33327104}{1594323}>0,
\]
it follows that \(1+\lambda_-(t)>0\) throughout the interval, and
hence \(1+\lambda_-(a)>0\) for all
\(a\in\bigl(\sqrt3/3,\sqrt3\bigr)\).

Thus \(1+\lambda_+(a)>0\) and \(1+\lambda_-(a)>0\) throughout
\(\bigl(\sqrt3/3,\sqrt3\bigr)\). By the Generalized Eigenvalue Criterion
of Section~\ref{subsec:rayleigh_bound}, the restriction of
\(H_{\mathcal C}\) to \(\mathcal W\) is positive definite. Since
\(\mathbf v\in\ker H_{\mathcal C}\), it follows that \(H_{\mathcal C}\)
is positive semidefinite on all of \(\mathcal C\), with kernel exactly
\(\operatorname{span}\{\mathbf v\}\). In particular, the Morse index of
the rhombus central configuration is zero.
\end{proof}

\section*{Acknowledgements}

During the preparation of this manuscript, the author used the Perplexity AI
web platform for assistance with language and exposition. The service
provided access to multiple generative language models, whose individual
versions were not recorded. The author critically reviewed and revised all
AI-assisted material and takes full responsibility for the final content of
the manuscript.
\begin{appendices}

\section{Matrix-weighted graph Laplacians}
\label{appendix:matrix_laplacians}
For completeness, we record the standard positivity identity for
matrix-weighted graph Laplacians used in the proof of
Proposition~\ref{prop:LDelta_psd}.
\begin{lemma}[Positive Semidefiniteness of Matrix-Weighted Laplacians]
\label{lem:matrix_weighted_laplacian_psd}
Let $G=(V,E)$ be a finite graph on vertices $\{1,\dots,n\}$. Suppose each edge
$\{i,j\}\in E$ is assigned a symmetric positive semidefinite matrix
$\mathbf W_{ij}\in\mathbb R^{d\times d}$, with $\mathbf W_{ij}=\mathbf W_{ji}$.
Let $L_G$ be the associated matrix-weighted graph Laplacian, defined blockwise by
\[
(L_G)_{ii}=\sum_{j\neq i}\mathbf W_{ij},
\qquad
(L_G)_{ij}=-\mathbf W_{ij}\quad (i\neq j),
\]
where $\mathbf W_{ij}=0$ if $\{i,j\}\notin E$. Then $L_G$ is positive semidefinite.
More precisely, for every $\mathbf x=(\mathbf x_1,\dots,\mathbf x_n)\in\mathbb R^{dn}$,
\begin{equation}\label{eq:matrix_laplacian_qf}
\mathbf x^T L_G \mathbf x
=
\sum_{\{i,j\}\in E}(\mathbf x_i-\mathbf x_j)^T\mathbf W_{ij}(\mathbf x_i-\mathbf x_j)
\ge 0.
\end{equation}
Consequently, every reduced matrix-weighted Laplacian obtained by deleting
the block of any single vertex is also positive semidefinite, since the
vertices may be relabeled arbitrarily.
\end{lemma}

\begin{proof}
Expanding the quadratic form blockwise gives
\[
\mathbf x^T L_G \mathbf x
=
\sum_{i=1}^n \mathbf x_i^T\Bigl(\sum_{j\neq i}\mathbf W_{ij}\Bigr)\mathbf x_i
-\sum_{\substack{i,j=1\\i\neq j}}^n \mathbf x_i^T\mathbf W_{ij}\mathbf x_j.
\]
Grouping together the two ordered pairs $(i,j)$ and $(j,i)$ associated with each
unordered edge $\{i,j\}\in E$, and using the symmetry $\mathbf W_{ij}=\mathbf W_{ji}$,
we obtain
\[
\mathbf x^T L_G \mathbf x
=
\sum_{\{i,j\}\in E}
\Bigl(
\mathbf x_i^T\mathbf W_{ij}\mathbf x_i
+\mathbf x_j^T\mathbf W_{ij}\mathbf x_j
-2\,\mathbf x_i^T\mathbf W_{ij}\mathbf x_j
\Bigr).
\]
Each summand is
\[
\mathbf x_i^T\mathbf W_{ij}\mathbf x_i
+\mathbf x_j^T\mathbf W_{ij}\mathbf x_j
-2\,\mathbf x_i^T\mathbf W_{ij}\mathbf x_j
=
(\mathbf x_i-\mathbf x_j)^T\mathbf W_{ij}(\mathbf x_i-\mathbf x_j),
\]
which is nonnegative because $\mathbf W_{ij}\succeq 0$. This proves
Eq.~\eqref{eq:matrix_laplacian_qf}, hence $L_G\succeq 0$.

Now let $\widetilde L_G$ be the reduced Laplacian obtained by deleting, say, the
first vertex block. For any $\mathbf y=(\mathbf y_2,\dots,\mathbf y_n)\in\mathbb R^{d(n-1)}$,
define $\mathbf x=(0,\mathbf y_2,\dots,\mathbf y_n)\in\mathbb R^{dn}$. Then
\[
\mathbf y^T \widetilde L_G \mathbf y
=
\mathbf x^T L_G \mathbf x
\ge 0,
\]
so $\widetilde L_G\succeq 0$ as well.
\end{proof}

\section{Matrix Components of $L^{\Delta}_{\text{reg}}$}
\label{appendix-LDelta}
The entries of the matrix $L^{\Delta}_{\text{reg}}$  defined in Equation~\eqref{eq:L_reg_matrix} are given by the expressions below.   We begin by defining the simplest non-zero components directly:

\begin{align*}
L_{12} &= -\frac{a}{2(3a^2 + 1)} & L_{15} &= -\frac{5a^2 + 3}{8(3a^2 + 1)} \\
L_{26} &= -\frac{1}{2(3a^2 + 1)} & L_{14} &= \frac{3a^4(s a^2 + s - 8)}{s^5(s a^2 - 8a^3 + s)} \\
L_{24} &= -a \cdot L_{14}
\end{align*}

For the intermediate entries, it is convenient to introduce the shared denominator $D_1 = 24s^5a^5 - 3s^6a^4 + 8s^5a^3 - 4s^6a^2 - s^6$, which yields:
\begin{align*}
L_{13} &= -\frac{a^2(a^8 - 8s a^7 + 4a^6 - 25s a^5 + 6a^4 + (72 - 20s)a^3 + 4a^2 + (24 - 3s)a + 1)}{D_1} \\[1em]
L_{23} &= \frac{a(a^8 - 17s a^7 + 4a^6 + (72 - 28s)a^5 + 6a^4 + (24 - 11s)a^3 + 4a^2 + 1)}{D_1} \\[1em]
L_{22} &= \frac{-36s a^9 - a^8 + (288 - 40s)a^7 - 4a^6 + (96 + 4s)a^5 - 6a^4 + 8s a^3 - 4a^2 - 1}{2s^5(3a^2 + 1)(8a^3 - s a^2 - s)}
\end{align*}

The remaining diagonal entries are substantially larger and are best expressed as rational functions of the form $L_{ii} = \frac{N_{ii}}{D_{ii}}$. For $L_{11}$, the numerator and denominator are given by:
\begin{align*}
N_{11} &= 104a^{16} + 296a^{14} + (312 - 13s^3)a^{13} + 240a^{12} + (1176 - 55s^3)a^{11} \\
&\quad - (1677s^3 + 16)a^{10} + (1680 - 90s^3)a^9 + (1353s^3 - 18520)a^8 \\
&\quad + (1104 - 70s^3)a^7 + (678s^3 - 6168)a^6 + (1127s^3 + 312)a^5 \\
&\quad - 70s^3a^4 + (381s^3 + 24)a^3 - 25s^3a^2 - 3s^3 \\[1em]
D_{11} &= 192 s^5 \left(a^2 + \frac{1}{3}\right)\left(s a^5 + (s - 16)a^3 + s a^2 + s\right)\left(a^3 - \frac{1}{8}s a^2 - \frac{1}{8}s\right)
\end{align*}

Similarly, for $L_{33}$, we have:
\begin{align*}
N_{33} &= -2a^2 \Big(16s^3a^{14} - 65a^{17} + 55s^3a^{12} - 191a^{15} + 1056s^3a^{11} - 321a^{14} \\
&\quad + 66s^3a^{10} - 165a^{13} - 585s^3a^9 - 999a^{12} + 28s^3a^8 - 3a^{11} - 414s^3a^7 \\
&\quad - 981a^{10} - 1730s^3a^6 + 9333a^9 + 92s^3a^5 - 163a^8 - 579s^3a^4 + 3315a^7 \\
&\quad + 14s^3a^3 + 221a^6 + 281a^5 - 3s^3a + 75a^4 + 143a^3 - 7a^2 + 24a - 1 \Big) \\[1em]
D_{33} &= s^5(3a^2 + 1)\left(s a^5 + s a^3 + s a^2 - 16a^3 + s\right)\left(s a^2 - 8a^3 + s\right)^2
\end{align*}

Finally, the terms for $L_{44}$ are:
\begin{align*}
N_{44} &= 3a^3(s a^2 + s - 8)\Big(s a^{13} + (5s - 8)a^{11} - 127s a^{10} + (10s - 48)a^9 \\
&\quad + (2056 - 123s)a^8 - (246s + 96)a^7 + (10s + 16)a^6 - (123s + 80)a^5 \\
&\quad + 10s a^4 + (129s - 24)a^3 + (5s - 16)a^2 + s - 8 \Big) \\[1em]
D_{44} &= 512 s^5 \left(s a^5 + (s - 16)a^3 + s a^2 + s\right)\left(a^3 - \frac{1}{8}s a^2 - \frac{1}{8}s\right)^2
\end{align*}

\end{appendices}

\printbibliography

\end{document}